\documentclass[11pt, reqno]{article}
\usepackage{preamble}

\usepackage[mathlines]{lineno}

\usepackage{etoolbox} 

\newcommand*\linenomathpatchAMS[1]{%
  \expandafter\pretocmd\csname #1\endcsname {\linenomathAMS}{}{}%
  \expandafter\pretocmd\csname #1*\endcsname{\linenomathAMS}{}{}%
  \expandafter\apptocmd\csname end#1\endcsname {\endlinenomath}{}{}%
  \expandafter\apptocmd\csname end#1*\endcsname{\endlinenomath}{}{}%
}

\expandafter\ifx\linenomath\linenomathWithnumbers
  \let\linenomathAMS\linenomathWithnumbers
  \patchcmd\linenomathAMS{\advance\postdisplaypenalty\linenopenalty}{}{}{}
\else
  \let\linenomathAMS\linenomathNonumbers
\fi

\linenomathpatchAMS{gather}
\linenomathpatchAMS{multline}
\linenomathpatchAMS{align}
\linenomathpatchAMS{alignat}
\linenomathpatchAMS{flalign}

\title{
The multistochastic Monge--Kantorovich problem\thanks{The article was prepared within the framework of the HSE University Basic Research Program. The second named author was supported by RFBR project 20-01-00432.}}
\date{}

\author{Nikita~A. Gladkov\thanks{UCLA Department of Mathematics, Los Angeles} \and Alexander~V. Kolesnikov \thanks{Faculty of Mathematics, HSE University, Russian Federation} \and Alexander~P. Zimin\thanks{Faculty of Mathematics, HSE University, Russian Federation \& Center for Advanced Studies, Skoltech, Moscow, Russian Federation} }

\begin{document}

\maketitle

\begin{abstract}
The multistsochastic Monge--Kantorovich problem on the product   $X = \prod_{i=1}^n X_i$ of $n$ spaces is a generalization 
of the multimarginal Monge--Kantorovich problem. For a given integer number $1 \le k<n$ we consider the minimization problem
$\int c d \pi \to \inf$ of the space of measures with fixed projections onto every  $X_{i_1} \times \dots \times X_{i_k}$
for arbitrary set of $k$ indices $\{i_1, \dots, i_k\} \subset \{1, \dots, n\}$. 
In this paper we study  basic properties of the multistochastic problem, including well-posedness, existence of a dual solution, boundedness
and continuity of a dual solution.
\end{abstract}
\tableofcontents
\section{Introduction}

This paper is a continuation of our previous work \cite{GlKoZi}, where we studied a natural generalization of the transportation or
Monge--Kantorovich problem.

Let $\mu$ and $\nu$  be probability measures on measurable spaces $X$ and $Y$, and let $c: X \times Y \to \mathbb{R}$ be a measurable function.
The classical Kantorovich  problem is the minimization problem  
\[
\int_{X \times Y}c(x, y)~d\pi \to \inf
\]
on the space  $\Pi(\mu, \nu)$  of probability measures on $X \times Y$ with fixed marginals  $\mu$ and $\nu$.

It is well-known that this problem is closely related to another linear programming problem, which is called 
``dual transportation problem'' 
\[
\int f~d\mu + \int g~d\nu \to \sup.
\]
The dual transportation problem is considered on the couples of integrable functions $(f, g)$, satisfying $f(x) + g(y) \le c(x, y)$ for all $x \in X$, $y \in Y$.

Nowadays, the Monge--Kantorovich theory attracts growing attention. The reader can find  huge amount of information in the following books and surveys papers: \cite{AGS}, \cite{BoKo}, \cite{PeyreCuturi},\cite{Galichon},   \cite{McCannGuill}, \cite{HL}, \cite{RR}, \cite{Santambrogio}, \cite{Villani}, \cite{Villani2}.

A particular case of the multistochastic problem is the multimarginal transportation problem. In the multimarginal problem one considers the product of
$n > 2$ spaces and $n$ independent marginals $\mu_1, \dots, \mu_n$. Some classical results on the multimarginal problem is contained in book
\cite{RR}, in particular, functional-analytical duality theorems, applications to probability etc. Nevertheless, till recent, only the case of two marginals
was in focus of research. A revival of interest in the case of many marginals is partially motivated   by applications in economics and
 quantum physics \cite{CoDPMa}, \cite{CFK},  \cite{MGN}, \cite{Pass}.
 Our motivation to study  the cost function $xyz$ in $\mathbb{R}^3$ is partially related to the multimarginal problem  considered 
in  \cite{GlKoZi2}.

In \cite{GlKoZi} we introduce a more general problem, which we call ``multistochastic problem''. Compare to the classical (multimarginal) case
this new problem is genuinely more difficult. Even its well-posedness  depends on the structure of the marginals in a complicated way.
The aim of this work is to fill many gaps related to basic properties of the problem.

The paper is organized as follows: the reader can consider Section 2 as an extended introduction, where we present the results of the paper, 
our previous results, open questions, examples, and discuss relations to other problems. In Section 3 we study sufficient conditions  for existence
of a feasible measure for the multistochastic   problem. In Section 4 we give a proof of a duality theorem  which is based
on the duality theory  for linearly constrained transportation problem. In Section 5 we study sufficient conditions for existence of a dual solution and construct an
example of non-existence. In Section 6 we give explicit uniform bounds for the dual solution under assumption that the cost function is bounded. Then we prove uniqueness of the primal and dual solutions in our main example studied in \cite{GlKoZi}.
Finally, we give an example showing that a dual solution can be discontinuous even for a nice cost function $c$.  

\section{The multistochastic Monge--Kantorovich problem. Preliminaries, examples, and open questions.}

We start with the formulation of the multistochastic problem in the most general setting. 
Let $X_1$, $X_2$, \dots, $X_n$ be measurable spaces equipped with $\sigma$-algebras 
$\mathcal{B}_1, \dots, \mathcal{B}_n$.
It will be assumed throughout that $X_i$ are Polish spaces and $\mathcal{B}_i$ are Borel sigma algebras.

\begin{definition}
Let $p$, $q$ be nonnegative integers, $q \le p$. Let us denote by $\mathcal{I}_{pq}$ the family of subsets
 $\{1, 2, \dots, p\}$ of cardinality $q$. In addition, the family of all subsets of $\{1, 2, \dots, p\}$ will be denoted by
  $\mathcal{I}_p = \cup_{q = 0}^{p}\mathcal{I}_{pq}$.
\end{definition}

\begin{definition}
For all $\alpha \in \mathcal{I}_n$ let us set  $X_\alpha = \prod_{i \in \alpha}X_i$. The product of all spaces  $X = \prod_{i = 1}^n X_i$ will be denoted by $X$. For a fixed $\alpha \in \mathcal{I}_n$ the projection of $X$ onto $X_\alpha$ will be denoted by  $\prj_\alpha$. In addition, for arbitrary   $x \in X$  the image of $x$ under projection $\prj_\alpha$ will be denoted by $x_{\alpha}$:
$x_{\alpha} = \prj_\alpha(x)$.
\end{definition}

For arbitrary space  $X$ let us denote by $\mathcal{P}(X)$ the space of all probability measures on $X$.

\begin{problem}[Primal $(n, k)$-Monge--Kantorovich problem]
Given Polish spaces  $X_1, \dots, X_n$, fixed family of measures $\mu_\alpha \in \mathcal{P}(X_\alpha)$, $\alpha \in \mathcal{I}_{nk}$, and a measurable cost function $c$. Assume in addition that there exist integrable functions $c_\alpha \in L_1(X_\alpha, \mu_\alpha)$, $\alpha \in \mathcal{I}_{nk}$, such that $|c(x)| \le \sum_{\alpha \in \mathcal{I}_{nk}}c_\alpha(x_\alpha)$. Then we are looking for  
\[
\inf_{\pi \in \Pi(\mu_\alpha)}\int_X c~d\pi,
\]
where infimum is taken among  the all uniting measures $\pi$.
\end{problem}
Note that under that assumptions the cost function $c$ is integrable with respect to every uniting measure $\mu \in \Pi(\mu_\alpha)$. Indeed, one has  $\int_X |c|\,d\mu \le \sum_{\alpha \in \mathcal{I}_{nk}}\int_X c_\alpha(x_\alpha)~d\mu_\alpha$, if $|c(x)| \le \sum_{\alpha \in \mathcal{I}_{nk}} c_\alpha(x_\alpha)$.

In what follows, we will additionally assume that $c$ is continuous, and we will work with the following functional spaces
\begin{align*}
    &C_L(X_\alpha, \mu_\alpha) = C(X_\alpha) \cap L^1(\mu_\alpha),\\
    &C_L(X, \mu_\alpha) = \left\{c \in C(X): |c(x)| \le \sum_{\alpha \in \mathcal{I}_{nk}} c_\alpha(x_\alpha) \text{ for some } \ c_\alpha \in C_L(X_\alpha, \mu_\alpha)\right\}.
\end{align*}
In addition, $C_b(X)$ is the space of all continuous bounded functions on  $X$, $C_b(X) \subset C_L(X)$.

\begin{definition}
Assume that for every  $\alpha \in \mathcal{I}_{nk}$ we are given a probability measure $\mu_\alpha$ on $X_\alpha$. We say that a measure  $\mu \in \mathcal{P}(X)$ is {\bf uniting} if $\prj_\alpha(\mu) = \mu_\alpha$ 
for all $\alpha \in \mathcal{I}_{nk}$. The set of all uniting measures will be denoted by $\Pi(\mu_\alpha)$.
\end{definition}

\begin{example}{\bf ($(3, 2)$-problem)}
Consider a product of three spaces $X=X_1 \times X_2 \times X_3$, probability measures 
$\mu_{12}$, $\mu_{23}$, $\mu_{13}$ on $X_1 \times X_2$, $X_2 \times X_3$, $X_1 \times X_3$ respectively. Then 
$\mu \in \Pi(\mu_\alpha)$ if and only if $\mu$ is a measure on $X$ such that
\[
\prj_{12}(\mu) =\mu_{12}, \prj_{13}(\mu) =\mu_{13}, \prj_{23}(\mu) =\mu_{23}.
\]
\end{example}

In this introductory section we briefly describe several aspects of this problem. In particular, we discuss previously known results, examples,  open problems, and relation to other 
research. 

\subsection{Feasibility of the problem,  Latin squares and descriptive geometry.} The multistochastic problem is overdetermined and a uniting measure does not always exist.
 It is clear that a necessary condition for existence of a uniting measure is the following consistency condition:
\[
\prj_{\alpha \cap \beta}(\mu_\alpha) = \prj_{\alpha \cap \beta}(\mu_\beta) = \prj_{\alpha \cap \beta}(\mu).
\]
 This condition is not sufficient (see \cite{GlKoZi} and other examples below), but we show that this condition is sufficient for
 existence of a {\bf signed} uniting measure (see \cref{explicit_signed_construction}).
 
 Nevertheless, in certain situations the set of feasible measures is very rich. This happens, for instance, if $X_i$  are finite sets of the same cardinality
 and all the measures $\mu_{\alpha}$ are uniform. The natural continuous generalization is: $X_i = [0, 1]$ and $\mu_{\alpha}$ are the Lebesgue measures
 on  $[0, 1]^k$ of the corresponding dimension $k$.  A natural related discrete combinatorial object is a {\bf Latin square}.
 To see the relation let us consider an $n \times n$ Latin square $S$ containing first $n$ integers.
 Then the discrete measure
 \[
\frac{1}{n^2} \sum_{i, j} \delta_{i, j, S(i, j)}
 \]
 on $[1, \dots, n]^3$
 has uniform projections to discrete $xy$, $xz$, $yz$  planes.
 
 More generally, the $(n, k)$-multistochastic problem is always feasible for the system of measures
 \[
 \mu_{\alpha} = \prod_{i \in \alpha} \mu_i, \ \alpha \in {\mathcal{I}}_{nk},
 \]
 where $\mu_1, \dots, \mu_n$ are fixed measures on $X_1, \dots, X_n$.
 
 We believe that this example provides a natural source of applications, this is why a big part of our results 
 is related to this particular case. 
 
 Other source of applications might arise from the engineering, in particular, the descriptive geometry.  
 One of the founding fathers of descriptive geometry, Gaspard Monge, developed a method of reconstruction of a three-dimensional
 body using its two-dimensional orthogonal projections. This procedure is known as {\bf ``projection de Monge''},
 in our language it gives a recipe of finding a uniting measure in $(3, 2)$-problem.

 A necessary and sufficient condition for existence of a measure with a given system  of marginal distributions in the spirit of linear programming duality
 was established by H.~Kellerer \cite{Kellerer64}.
Assume we are  given  a system of marginal distributions 
 $\mu_{\alpha}$, where $\alpha$ belongs to some system $A$ of subsets of $\{1, \dots, n\}$.
 This system admits a uniting measure if and only if
 \[
\sum_{\alpha \in A} \int f_{\alpha}(x_{\alpha})\,d \mu_{\alpha} \ge 0
 \]
 for all bounded continuous system of functions $f_{\alpha}(x_{\alpha})$
 satisfying $\sum_{\alpha \in A}  f_{\alpha}(x_{\alpha})\ge 0$.
 We give an independent proof of this fact  for $A= {\mathcal{I}}_{nk}$ in Section 3.
 Note, however, that this criterion does not seem to be very practical. We establish some
 easy-to-check sufficient conditions for existence of uniting measure in terms of uniform bounds for densities.
 In particular, we prove the following (see  \cref{thm:density_condition}):

\begin{theorem}  
For given natural numbers $1 \le k <n$
there exists a constant $\lambda_{nk} > 1$
which admits the following property.

Assume we are given a consistent family of probability measures $\mu_\alpha \in \mathcal{P}(X_\alpha)$,  $\alpha \in \mathcal{I}_{nk}$, and another family of probability measures  $\nu_i \in \mathcal{P}(X_i)$, $1 \le i \le n$.  
Assume that  every measure  $\mu_\alpha$,  $\alpha \in \mathcal{I}_{nk}$, is absolutely continuous with respect to  $\nu_\alpha =  \prod_{i \in \alpha}\nu_i$:
\[
\mu_\alpha = \rho_{\alpha} \cdot \nu_{\alpha}.
\]
Finally, assume that  there exist constants $0 < m \le M$ such that every density $\rho_{\alpha}$ satisfies $m \le \rho_\alpha \le M$ $\nu_\alpha$-almost everywhere for all $\alpha \in \mathcal{I}_{nk}$.

Then  $\Pi(\mu_\alpha)$ is not empty provided  $\frac{M}{m} \le \lambda_{nk}$.
\end{theorem}

We will give precise bounds for the constant $\lambda_{32} $.

\begin{remark} {\bf Solvability of the primal problem.} As soon as the  set of uniting measures is not empty, the proof of  existence of a solution to the primal problem
 for a lower semicontinuous cost is a standard exercise.
 \end{remark}
 
\begin{theorem}  [\cite{GlKoZi}] \label{thm:primal_solution_exists}
Assume that the cost function
$c\ge 0$ is lower semicontinuous. If  $\Pi(\mu_\alpha)$ is not empty, then there exists a solution to the multistochastic problem.
\end{theorem}

 \subsection{Examples. Fractal structure versus smooth structure.}
  
 The main example of an explicit solution to a multistochastic problem was found in \cite{GlKoZi}. 
 The unexpected beauty of this example  was the main motivation for us for subsequent study of the multistochastic problem.

 In the following example we consider a $(3, 2)$-problem. Denote by 
  $\Pi(\mu_{xy}, \mu_{yz}, \mu_{xz})$ the set of measures with projections
 ${\rm Pr}_{xy}\pi  = \mu_{xy}, {\rm Pr}_{xz}\pi  = \mu_{xz}, {\rm Pr}_{yz}\pi  = \mu_{yz}$.
 
\begin{theorem}[\cite{GlKoZi}]
\label{mainprimal}

Let $\mu_{xy}= \lambda_{xy}, \mu_{xz}= \lambda_{xz}, \mu_{yz}=\lambda_{xz}$ be the two--dimensional Lebesgue measures  $\left[0, 1\right]^2$ and let $c=xyz$.
Then there exists a unique solution to the  corresponding
$(3, 2)$-problem
\[
\int xyz \, d \pi \to \min, \ \pi \in \Pi(\mu_{xy}, \mu_{yz}, \mu_{xz}).
\]
 It is concentrated on the set 
\[
S = \{ (x, y, z) \colon x\oplus y \oplus z =0\}, 
\]
where  $\oplus$ is the bitwise addition. See 
\cref{fig:tetrahedron}.
\end{theorem}

\begin{figure}
    \centering
    \includegraphics[scale=0.4]{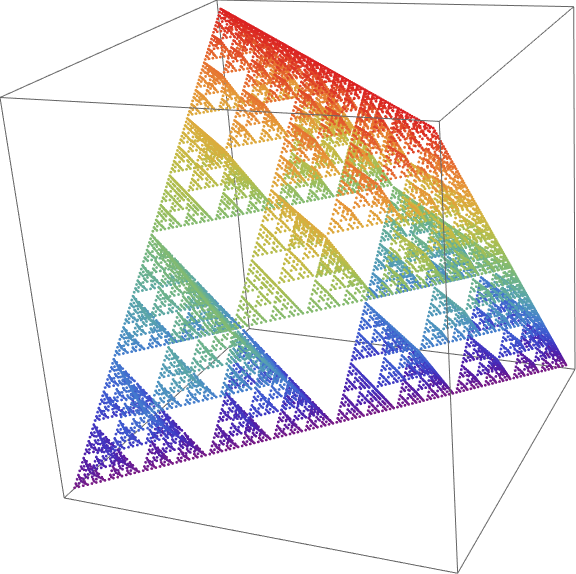}
    \caption{The solution is supported on Sierpi{\'n}sky tetrahedron.}
    \label{fig:tetrahedron}
\end{figure}


The set $S$ is called \textit{Sierpi{\'n}sky tetrahedron}.

We stress that some fractal solutions to a multimarginal transportation problem were known before our work. See, for instance,
\cite{MGN}, where multimarginal problem with the cost function of the type $h (\sum_{i=1}^n x_i)$ and the Lebesgue measure projections
was considered. Though we don't see any direct relation between these examples, they have something in common: in both cases the set of feasible measures contains more than one element and  the entire construction relies  on the dyadic decomposition.

\begin{remark}
\label{fghzero}
The $(3, 2)$-problem can admit not only fractal but also smooth solutions.
For instance, consider measurable functions $f(x)$, $g(y)$ and $h(z)$ on $\left[0, 1\right]$. Assume that  $h$ is injective, 
the set  $\Gamma = \{f(x) + g(y) + h(z)=0\}$ is not empty,  and $\mu$ is a probability measure concentrated on $\Gamma$: $\mu(\Gamma)=1$. 
Set  $\mu_{xy}=\prj_{xy} \mu$, $\mu_{xz}=\prj_{xz} \mu$, $\mu_{yz}=\prj_{yz} \mu$. Then  $\mu$ is the unique element of $\Pi(\mu_{xy}, \mu_{yz}, \mu_{xz})$. 
Indeed, let  $\nu \in \Pi(\mu_{xy}, \mu_{yz}, \mu_{xz})$. Clearly, 
$\int \bigl( f(x)+ g(y)+ h(z)\bigr)^2 d \nu$ depends solely on the integrals of pairwise products of functions $f, g, h$ with respect to  measures  $\mu_{xy}, \mu_{yz}, \mu_{xz}$.
Hence
\[
\int \bigl( f(x)+ g(y)+ h(z)\bigr)^2 d \nu = \int \bigl( f(x)+ g(y)+ h(z)\bigr)^2 d \mu =0, 
\]
this implies that  $\nu$ is concentrated on  $\Gamma$. Since  $h$ is injective,  $\Gamma$ is the graph of the mapping   $(x, y) \to  h^{-1}(-f(x) - g(y))$, 
hence  $\nu$ is uniquely determined by its projection 
$\mu_{xy}$, thus coincides with  $\mu$.
\end{remark}

In particular, this observation can be applied to construct an example of a solution concentrated on a smooth set.
\begin{example}
\label{plane}
The Lebesgue measure on $\left[0, 1\right]^3 \cap \{x_1+x_2+x_3=1\}$ is a solution to the $(3, 2)$-problem, where marginals are the two-dimensional Lebesgue measures
concentrated on the set  $\{x_i + x_j \le 1\} \subset \left[0, 1\right]^2$ and arbitrary cost function.
\end{example}
 
 It is clear, that the smoothness of the solution in this example is just a matter of fact that $ \Pi(\mu_{xy}, \mu_{yz}, \mu_{xz})$ contains a unique (smooth) element.
 However, it is natural to expect that the solution may have a fractal/non-regular structure provided uniting measures constitute a sufficiently large set.
 
 The following problem, yet vaguely formulated, seems to be crucial for understanding of the structure of solutions to $(n, k)$-problem.
 
 {\bf Open problem 1.}
  Is it true that solutions to $(n, k)$-problem have ``fractal structure'' provided $\Pi(\mu_\alpha)$ contains sufficiently ``rich'' set of measures?

 \subsection{Duality and the Kantorovich problem with linear constraints}  
 
 As in the classical case the multistochastic problem admits  the corresponding  dual problem:
 
\begin{problem}[Dual $(n, k)$-Monge--Kantorovich problem]
Assume we are given Polish spaces  $X_1, \dots, X_n$, a fixed family of measures $\mu_\alpha \in \mathcal{P}(X_\alpha)$ and a cost function  $c \in C_L(X, \mu_\alpha)$. 
Find 
\[
\sup_{f \le c}\sum_{\alpha \in \mathcal{I}_{nk}}\int_{X_\alpha}f_\alpha~d\mu_\alpha,
\]
where the supremum is taken among the functions $f$ having the form  $f(x) = \sum_{\alpha \in \mathcal{I}_{nk}}f_\alpha(x_\alpha)$, where  $f_\alpha \in L^1(X_\alpha, \mu_\alpha)$.
\end{problem}

\begin{definition}
We say that there is no duality gap for the $(n, k)$-problem
if 
\[
\min_{\pi \in \Pi(\mu_\alpha)}\int c~d\pi = \sup_{f \le c}\sum_{\alpha \in \mathcal{I}_{nk}}\int_{X_\alpha}f_\alpha~d\mu_\alpha,
\]
where 
$f_\alpha \in L^1(X_\alpha, \mu_\alpha)$, $f(x) = \sum_{\alpha \in \mathcal{I}_{nk}}f_\alpha(x_\alpha)$.
\end{definition}

The absence of duality gap was shown in \cite{GlKoZi}
under assumption of compactness of the spaces $X_i$.
In this work we prove the following result:
\begin{theorem}
There is no duality gap for $(n, k)$-problem provided $X_i$ are Polish spaces and
  $c \in C_L(X, \mu_\alpha)$.
\end{theorem}

Our approach is based on the result of D.~Zaev \cite{Zaev} on duality for the classical Kantorovich problem with 
{\bf linear constraints}.  The transportation problem with linear constraints is the standard Kantorovich problem
with additional constraint of the type
$l(P)=0$, where $l$ is a linear functional on the space of measures. The proof of Zaev is based on the general minimax principle.

\subsection{Structure of dual solutions. Monge problem}

Our main example of a dual solution is given in the following theorem.

\begin{theorem}[\cite{GlKoZi}]
\label{maindual}
Let $\mu_{xy}= \lambda_{xy}, \mu_{xz}= \lambda_{xz}, \mu_{yz}=\lambda_{xz}$ be the two dimensional Lebesgue measures on $\left[0, 1\right]^2$ and $c=xyz$. 
Then the triple of functions $(f(x, y), f(x, z), f(y, z))$, where
\[
f(x, y) = \int_0^x \int_0^y t \oplus s~dtds - \frac{1}{4}  \int_0^x \int_0^x t \oplus s~dtds - \frac{1}{4}  \int_0^y \int_0^y t \oplus s~dtds
\]
solves the corresponding dual multistochastic problem.
\end{theorem}

\begin{remark}
The uniqueness result for this problem under assumption of continuity of the dual solution is proved in the present paper in \cref{tetrahedron-dual-unique}
\end{remark}

The solution to the dual problem given in \cref{maindual},
has the following relation to the solution $\pi$ to the primal problem (see \cref{mainprimal}): $\pi$ is concentrated on the graph of the mapping
$(x, y) \mapsto f_{xy}(x, y)$, i.e. 
\begin{equation}
\label{zfxy}
z = f_{xy}(x, y)
\end{equation}
$\pi$-almost everywhere.

Let us note that $f$ admits a non-negative mixed derivative 
$f_{xy}$, but derivatives  $f_{xx}, f_{yy}$ do not exist (at least in the classical sense). 

The relation \cref{zfxy} can be derived from the fact that the support $S$ of the solution  $\pi$ is a fractal set.  Indeed,
function  $f(x, y) +  f(x, z) + f(y, z) - xyz$ is non positive and equals zero  $\pi$-a.e. Thus for   $\pi$-almost all points 
the first order condition 
\begin{equation}
\label{fxyz}
f_x(x, y) + f_{x}(x, z) = yz
\end{equation}
is satisfied.

Next, it is easy to show that for  $\pi$-almost every point  $M = (x_0, y_0, z_0) \in S$ the set  $S$ contains points of the type $M + t_n v$, where $t_n$ is a sequence tending to zero and 
vector  $v$ belongs to a set $V$ containing three independent vectors. One can prove this using the fractal structure of  $S$. 
Consequently, one can differentiate 
\cref{fxyz} along $V$ and deduce \cref{zfxy} from these relations.

Thus in this particular case the solution admits the following properties.

\begin{enumerate}[label=\upshape{(\alph*)}]
\item The solution is concentrated on the graph of a mapping $z=T(x, y)$.

\item This mapping $T$ has the form $T(x, y) = f_{xy}(x, y)$, where $(f, g, h)$ is a solution to the dual problem.
The same holds for $g, h$.

\item Function $f(x, y)$ is a cumulative distribution  function (up to a  term depending on $x$ and a term depending on $y$) of a positive measure on a plane.
Equivalently, $f_{xy}(x, y) \ge 0$ almost everywhere.
\end{enumerate}
 
 \begin{definition} {\bf(Optimal mapping.)}
 Let $T$ satisfy (a). Then we say that $T$ is an optimal mapping. 
 \end{definition}
 
 One can ask whether any solution to $(3, 2)$-problem (under natural assumptions on the marginals) with the cost function $xyz$ does satisfy properties (a), (b), (c).
We show that in  fact no one of these properties are satisfies in general.

%
 
 \begin{example}
{\bf The solutions to $(3, 2)$-problems are not always concentrated on graphs; (a) fails.}
Consider the sphere $S = \{x^2 + y^2 + z^2 = 1\}$, and consider the quarter sphere $S_1 = S \cap \{x \ge 0, y \ge 0\}$, $S_2 = S \cap \{x < 0, y \ge 0\}$, $S_3 = S \cap \{x < 0, y < 0\}$ and $S_4 = S \cap \{x \ge 0, y < 0\}$. Let $\pi$ be the surface measure on the 3/4-part of the sphere $S_1 \sqcup S_2 \sqcup S_4$, and let $\mu_{xy}, \mu_{xz}, \mu_{yz}$ be the corresponding two-dimensional projections. 

Slightly modifying the arguments of \cref{fghzero} we prove that if $\widehat{\pi}$ is a measure with projections $\mu_{xy}$, $\mu_{xz}$ and $\mu_{yz}$, then
$\widehat{\pi}$ is concentrated on the set $S_1 \sqcup S_2 \sqcup S_4$. For each point of $S_2$ there is no other point of $S_1 \sqcup S_2 \sqcup S_4$ with the same projection onto the coordinate plane $Oxz$, and therefore the restriction of the measure $\widehat{\pi}$ to $S_2$ is fully determined by its projection $\mu_{xz}$ and coincides with $\pi|_{S_2}$.

Similarly, the restriction of $\widehat{\pi}$ to $S_4$ is fully determined by its projection $\mu_{yz}$ and coincides with $\pi|_{S_4}$. Hence, $\widehat{\pi}|_{S_1} = \widehat{\pi} - \pi|_{S_2} - \pi|_{S_4}$. Thus, the projections of $\widehat{\pi}|_{S_1}$ and $\pi|_{S_1}$ to the coordinate planes are the same, and then $\widehat{\pi}|_{S_1} = \pi|_{S_1}$. So we conclude that $\pi$ is the only measure with projections $\mu_{xy}, \mu_{xz}, \mu_{yz}$, and there is no optimal mappings $T_{xy}$, $T_{xz}$ and $T_{yz}$.

See also \cref{ex:discrete_nonuniform} for a discrete counterexample. 
 \end{example}

 \begin{example}
 {\bf Example without dual solutions satisfying (\ref{zfxy}); (b) fails. }
 This example is considered in  \cref{dualdiscontinuous}. In this example
 $f_{xy}$ is either zero or not defined.
 \end{example}
 
\begin{example}
\label{polynoms-dual}
{\bf Non-uniqueness for the dual problem; (c) fails.}
In the problem  considered in  \cref{plane}
there exist  many dual solutions.
To see this let us note that the following inequality  holds for all 
$(x, y, z) \in [0, 1]^3$ and a fixed constant $A > 0$, equality holds if and only if $x+y+z=1$:
\[
(x+y+z -1)^2 (x+y+z + A) \ge 0.
\]
Developing the left-hand side 
we see that this inequality is equivalent to
\[
xyz \ge f_A(x, y) + f_A(x, z) + f_A(y, z),
\]
where 
\[
f_A(x, y) = - \frac{1}{12} (x^3 + y^3)
-\frac{1}{2} xy(x+y) -(A-2) \Bigl( \frac{x^2}{12} + \frac{xy}{3} + \frac{y^2}{12}\Bigr)
- \frac{1-2A}{12} (x+y) - \frac{A}{18}. 
\]
Clearly, the triple $(f_A(x, y), f_A(x, z),f_A(y, z))$
solves the dual problem for every $A \ge 0$.
Note that  \cref{zfxy} and (c) fails for all $A \ge 0$.

We believe  that there are no other dual solution, but can not prove this. 
\end{example}

Thus we see that the particular form \cref{zfxy}  of the optimal mapping related to $(3, 2)$-problem with cost function $xyz$
is  related to the fractal structure of the solution. Motivated by these observations we state the following problem.

{\bf Open problem 2.}
Assume that $\pi$ is a solution to a $(3, 2)$-problem
with the cost function $xyz$.
Find general sufficient conditions for presentation of $\pi$ in the form
\[
z = f_{xy}(x, y),
\]
where $f(x, y), g(x, z), h(y, z)$ solve the corresponding dual multistochastic problem.

It seems  quite difficult to describe  the general structure of solutions to $(3,2)$-problem with $c=xyz$, since it is very sensitive to non-local properties of the marginals. Something can be established under very strong "smoothness" assumptions, as presented in the proposition below. But we stress that this situation can not pretend to describe a reasonable model case.

 \begin{proposition}
 Consider a tuple of twice continuously differentiable functions $f(x, y)$, $g(x,z)$, $h(y,z)$ satisfying
 $f(x,y) + g(x,z) + h(y,z) \ge xyz$.
  Assume, in addition, that
\[
\Gamma = \Bigl\{  
f(x,y) + g(x,z) + h(y,z) = xyz\Bigr\}
\] is a two-dimensional smooth surface.

Let $\Gamma_x, \Gamma_y, \Gamma_z$ be sets defined by equations:
\[
\Gamma_x = \{x = h_{yz}\}, \ \Gamma_y = \{y = g_{xz}\}, \ \Gamma_z = \{z = f_{xy}\}.
\]
Then for every point $(x_0, y_0, z_0) \in 
\Gamma$
the following alternative holds:
 \begin{enumerate}[label=\upshape{(\Alph*)}]
     \item
 $(x_0, y_0, z_0) \in \Gamma_x \cap \Gamma_y \cap \Gamma_z$, i.e. at this point
 \[
 x = h_{yz}, \ y = g_{xz}, \ z = f_{xy}.
 \]
 \item
 $(x_0, y_0, z_0) \notin \Gamma_x \cup \Gamma_y \cup \Gamma_z$ and the vector field
 \[
 N = \Bigl( \frac{1}{x-h_{yz}}, \frac{1}{y-g_{xz}}, \frac{1}{z-f_{xy}} \Bigr)
 \]
 is orthogonal to $\Gamma$ at $(x_0, y_0,z_0)$.
 \end{enumerate}
  \end{proposition}
 \begin{proof}
 Since every $(x, y, z) \in \Gamma$  is a minimum point of $f(x,y) + g(x,z) + h(y,z) - xyz$,
 then  the functions 
 \[
 u =yz - f_x(x,y) - g_x(x,z), \ v =xz - f_y(x,y) - h_y(y,z), \ w=xy - g_z(x,z) - h_z(y,z)
 \]
 vanish on $\Gamma$.
 Hence their gradients 
 \[ \nabla u=
 (-f_{xx} - g_{xx}, z - f_{xy}, y - g_{xz})
 \]
 \[\nabla v=
 (z-f_{xy}, - f_{yy} - h_{yy}, x - h_{yz})
 \]
 \[\nabla w=
 (y-g_{xz}, x - h_{yz}, - g_{zz} - h_{zz})
 \]
 are orthogonal to $\Gamma$.
 Then they are colinear, because $\Gamma$ is two-dimensional.
 Hence either all these coordinates are zero (case (1)) or
 \[
 f_{yy} + h_{yy} = - \frac{(x-h_{yz})(z-f_{xy})}{y-g_{xz}}
 \]
 (similarly for other coordinates). This gives that $N$ is orthogonal to $\Gamma$.
\end{proof}
 
\begin{remark}
Example of  (B) is given in \cref{polynoms-dual}.
We emphasize that in the main example we have  (A), but neither $\Gamma$ is not a smooth surface, nor $f,g,h$ are twice differentiable.
In fact, the fractal structure of $\Gamma$ is exactly the reason why (1) holds (see explanation above).
\end{remark}

 \begin{remark} {\bf (Vector fields orthogonal to smooth solutions)}.
 Assume that $\pi$ is a solution to a $(3, 2)$-problem concentrated on the surface $\Gamma$
 and alternative (B) holds. 
 Let $\pi$ is given by its density with respect to the two-dimensional Hausdorff measure
 \[
 \pi = p(x,y,z) \cdot \mathcal{H}^2|_{\Gamma}.
 \]
 Denote by $\rho_{xy}, \rho_{xz}, \rho_{yz}$ the density of the corresponding projections
 $\mu_{xy}, \mu_{xz}, \mu_{yz}$.
 Then 
 $\rho_{xy}(x,y) |\cos(N, (0, 0, 1))| =  p(x, y, z) 
 $ for every $(x,y,z) \in \Gamma$
 and
 \[
\rho_{xy}(x,y)  =   p(x,y,z) |z- f_{xy}| \sqrt{  \frac{1}{(x-h_{yz})^2} + \frac{1}{(y-g_{xz})^2}+ \frac{1}{(z-f_{xy})^2} }. 
 \]
 Similarly for the other densities.
 This easily leads to the following relations: for every $(x,y,z) \in \Gamma$ the vector field
 \[
\Bigl(  \frac{{\rm sign}(x-h_{yz})}{\rho_{yz}}, \frac{{\rm sign}(y-g_{xz})}{\rho_{xz}}, \frac{{\rm sign}(z-f_{xy})}{\rho_{xy}} \Bigr)
\]
 is orthogonal to $\Gamma$ and
 \[
 \frac{1}{p^2(x,y,z)} = \frac{1}{\rho^2_{xy}(x,y)}  + \frac{1}{\rho^2_{xz}(x,z)} + \frac{1}{\rho^2_{yz}(y,z)}.
 \]
In particular, we obtain that one of the vector fields
\[
\Bigl(  \frac{\pm 1}{\rho_{yz}}, \frac{\pm 1}{\rho_{xz}}, \frac{{\pm 1}}{\rho_{xy}} \Bigr)
\]
is (locally) orthogonal to $\Gamma$.
 \end{remark}

\begin{example}
\label{uniformbound}
{\bf (c) fails; relation to the transportation problem with uniform bound on density.}
Consider the $(3, 2)$-problem with   $c=xyz$ and $\mu_{xy} = \mu_x \otimes \mu_y$, $\mu_{xz} = \mu_x \otimes \mu_z$,  $\mu_{yz} = \mu_y \otimes \mu_z$, 
where $\mu_x =\mu_y$ is the Lebesgue measure on $[0,1]$, and $\mu_z$ is the uniform discrete measure on $\{0, 1, 2\}$.
Then the solution is concentrated on the graph of a function  $z=T(x, y)$, where $T$ takes values in $\{0, 1, 2\}$.

In this example we were able to verify numerically that there exists a dual solution $f(x,y), g(x,z), h(y,z)$ (maybe not unique) which does not satisfy $f_{xy} \ge 0$, equivalently $f(x_1, y_1) + f(x_2, y_2) - f(x_1, y_1) -  f(x_2, y_2) \ge 0$ for all $x_1 < x_2$, $y_1 < y_2$. In particular, relation $z=f_{xy}$ fails again.

Note that some elements of sets  $\{z=0\}$, $\{z=1\}$ (see \cref{fig:uniformbound}) are solutions to an optimal transportation problem with capacity constraints  
\cite{McCannKor}, \cite{McCannKor2}.
\end{example}

\begin{figure}

\begin{minipage}[b]{0.32\textwidth}
    \centering
    \includegraphics[width=\textwidth]{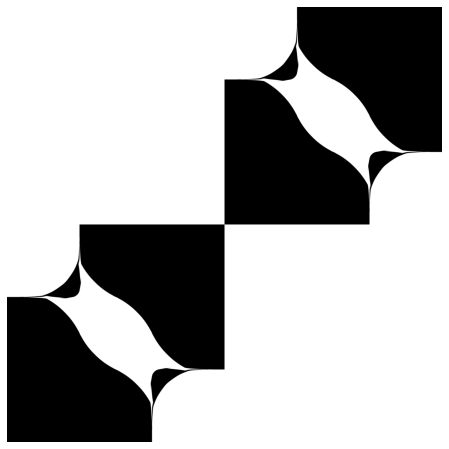}
    The slice $z = 0$.
\end{minipage}
\begin{minipage}[b]{0.32\textwidth}
    \centering
    \includegraphics[width=\textwidth]{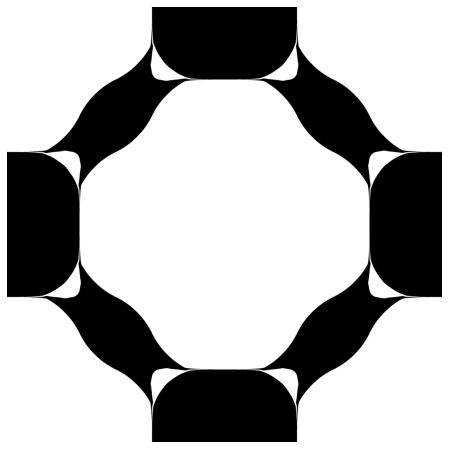}
    The slice $z = 1$.
\end{minipage}
\begin{minipage}[b]{0.32\textwidth}
    \centering
    \includegraphics[width=\textwidth]{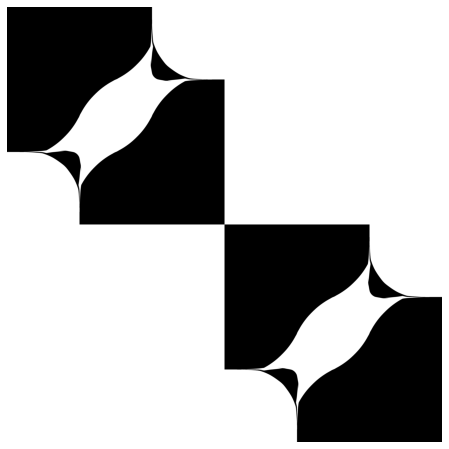}
    The slice $z = 2$.
\end{minipage}

\caption{The visualization of the primal solution to the problem considered in \cref{uniformbound}. Each picture shows the restriction of the primal solution to the set $z = 0, 1, 2$. In the white points the density function is equal to 0, and in the black points it is equal to 3. Almost every horizontal and vertical section of the black body has a length $1/3$. Compare to \cref{fig:caterpillar}.}
\label{fig:uniformbound}
\end{figure}
\begin{remark}
\label{stochdom}
It worth noting that the condition $f(x_1, y_1) + f(x_2, y_2) - f(x_1, y_1) -  f(x_2, y_2) \ge 0$ for all $x_1 < x_2$, $y_1 < y_2$
corresponds to a bit different primal problem, where assumptions on the marginals are replaced by assumptions that the marginals are {\bf stochastically dominated} by given measures. But we don't pursue this viewpoint here.
\end{remark}

\subsection{Solvability of the dual problem}

Section 5 is devoted to existence of a solution to the dual problem. We establish a sufficient existence condition
for the dual problem in the spirit of a classical result of Kellerer \cite{Kell} for the multistochastic
problem, but with a self-contained independent proof.

The main assumption on the cost function for {\bf solvability of the dual problem} is the following bound:
\begin{equation}
\label{Kell-bound}
 |c(x)| \le \sum_{\alpha \in \mathcal{I}_{nk}} C_\alpha(x_\alpha) .
\end{equation}
for some integrable functions, $C_\alpha \colon X_\alpha \to \mathbb{R}\cup\{+\infty\}$
This is a generalization of the Kellerer's assumption.

 However, yet another assumption, which is specific for $(n,k)$-problem, should be done on marginals.
 Namely, we have to assume that the system of measures $\{\mu_{\alpha}\}$ is {\bf reducible}.  
The latter means that there exists a measure $\mu \in \Pi(\mu_{\alpha})$ and the system of probability measures $\nu_i \in \mathbb{P}(X_i)$
such that for some $0 < c < C$
\begin{equation}
\label{reducebound}
c \nu \le \mu \le C \nu,
\end{equation}
where $\nu=\prod_i \nu_i$.
Our main existence/nonexistence result is the following Theorem (see details in \cref{thm:existence_multistochastic_dual_solution,X123}):
\begin{theorem}
If the system $\{ \mu_{\alpha} \}$ is reducible, then under assumption \cref{Kell-bound}
there exists a solution to the dual multistochastic problem.

Without assumption  of reducibility the dual solution may not exist.
More precisely, there exists an example of  a probability measure $\mu$ on the space 
$X = \mathbb{N}^3$  and the cost function   $c \colon X \to \{0,1\}$
such that there is no solution to the dual multistochastic  problem for the system
\[
\mu_{ij} = \prj_{ij}\mu.
\]
\end{theorem}

\subsection{Other properties of dual solutions: boundedness and (dis)continuity}

In Section 6 we study  basic properties of solutions to the dual  $(3, 2)$ problem: boundedness and continuity.
It is known that for the classical (multimarginal) problem the dual solution is bounded provided $|c|$ is bounded.
But this is crucial that in the classical case the dual solution is a sum 
of independent functions. This is the reason why it is hard to extend the arguments to the general $(n,k)$-case.   
We establish the following result on the boundedness of solutions.

\begin{theorem}
Let $X_1$, $X_2$, $X_3$ be Polish spaces,  $\mu_i \in \mathcal{P}(X_i)$ for $1 \le i \le 3$, and let $\mu_{ij} = \mu_i \otimes \mu_j$ for all $\{i, j\} \in \mathcal{I}_{3,2}$. Let $c \colon X \to \mathbb{R}_+$ be a bounded continuous cost function. 
If $\{f_{ij}\}$ is a solution to the related dual problem, then \[
f_{12}(x_1, x_2) + f_{13}(x_1, x_3) + f_{23}(x_2, x_3) \ge -12\norm{c}_{\infty}
\]
for $\mu_1 \otimes \mu_2 \otimes \mu_3$-almost all points $x \in X$.

Moreover,  there exists a solution $\{f_{ij}\}$ to the standard dual problem such that
\[
    -26\frac{2}{3} \norm{c}_{\infty} \le f_{ij}(x_i, x_j) \le 13\frac{1}{3}\norm{c}_{\infty}.
\]
\end{theorem}

Another important feature of the classical Monge--Kantorovich  problem: for a cost function $c$ with nice geometric/regularity properties
the corresponding dual solutions are regular. This happens because the dual functions are related by  Legendre  transform, which is highly regularizing.
We can not expect this for the $(n,k)$-problem, the following example demonstrates that a solution can be unique and discontinuous even for very
simple and nice cost: maximum of two linear functions.

\begin{example}
Let $X=Y=Z=[0,1]$. Consider the $(3,2)$-problem with the cost function
\[
c=\max(0,x+y+3z-3),
\]
where $\mu_{zy}, \mu_{xz}, \mu_{yz}$  are the Lebesgue measures restricted to $[0,1]^2$.
Then the dual problem admits a unique discontinuous solution, given by the following formulas:
\begin{align*}
    &f_{12}(x_1, x_2) = 0 \text{ for all points $(x_1, x_2) \in [0, 1]^2$};\\
    &f_{13}(x_1, x_3) = \begin{cases}
    0, &\text{if $x_3 < \frac{2}{3}$},\\
    x_1 + \frac{3}{2}x_3 - \frac{3}{2}, &\text{if $x_3 \ge \frac{2}{3}$};
    \end{cases}\\
    &f_{23}(x_2, x_3) = \begin{cases}
    0, &\text{if $x_3 < \frac{2}{3}$},\\
    x_2 + \frac{3}{2}x_3 - \frac{3}{2}, &\text{if $x_3 \ge \frac{2}{3}$}.
    \end{cases}
\end{align*}
\end{example}

\subsection{Uniqueness result for the main example}

In Section 6 we establish the following results for our main example: $(3,2)$-problem with the
two-dimensional Lebesgue marginals.

\begin{theorem}
\label{tetrahedron-dual-unique}
If a tuple of functions $\{f_{ij}\}$ is a solution to the problem from \cref{maindual} and every $f_{ij}$ is continuous for all $\{i, j\} \in \mathcal{I}_{3,2}$, then there exist continuous functions $f_i \colon [0, 1] \to \mathbb{R}$, $1 \le i \le 3$, such that 
\begin{align*}
&f_{12}(x_1, x_2) = f(x_1, x_2) + f_1(x_1) - f_2(x_2), \\
&f_{23}(x_2, x_3) = f(x_2, x_3) + f_2(x_2) - f_3(x_3),\\ 
\intertext{and}
&f_{13}(x_1, x_3) = f(x_1, x_3) + f_3(x_3) - f_1(x_1),
\end{align*}
where
\[
f(x, y) = \int_0^x\int_0^y s\oplus t\,dsdt - \frac{1}{4}\int_0^x\int_0^x s\oplus t\,dsdt - \frac{1}{4}\int_0^y\int_0^y s\oplus t\,dsdt.
\]
\end{theorem}

\begin{remark}
We believe that this problem admits no other (discontinuous) solutions, but have no proof of this.
\end{remark}

\subsection{Relation to other problems}

We mentioned already that the multistochastic problem is closely related to the Kantorovich problem
with linear constraints  studied by Zaev in \cite{Zaev}.
More precisely, our problem can be reduced to the Kantorovich problem
with linear constraints, see explanations in Section 4.

Another related problem is, of course, 
problem with uniform constraint on the density, sometimes called "the capacity constrained problem" (see \cite{McCannKor}, \cite{McCannKor2}, \cite{Doledenok}).
The solution to the problem from \cref{uniformbound}
admits the following  structure: there is a partition of the unit square into several parts, each of them 
is either  a homothetic image of the body shown on \cref{fig:caterpillar}
\begin{figure}
    \centering
    \includegraphics[width=0.33\textwidth]{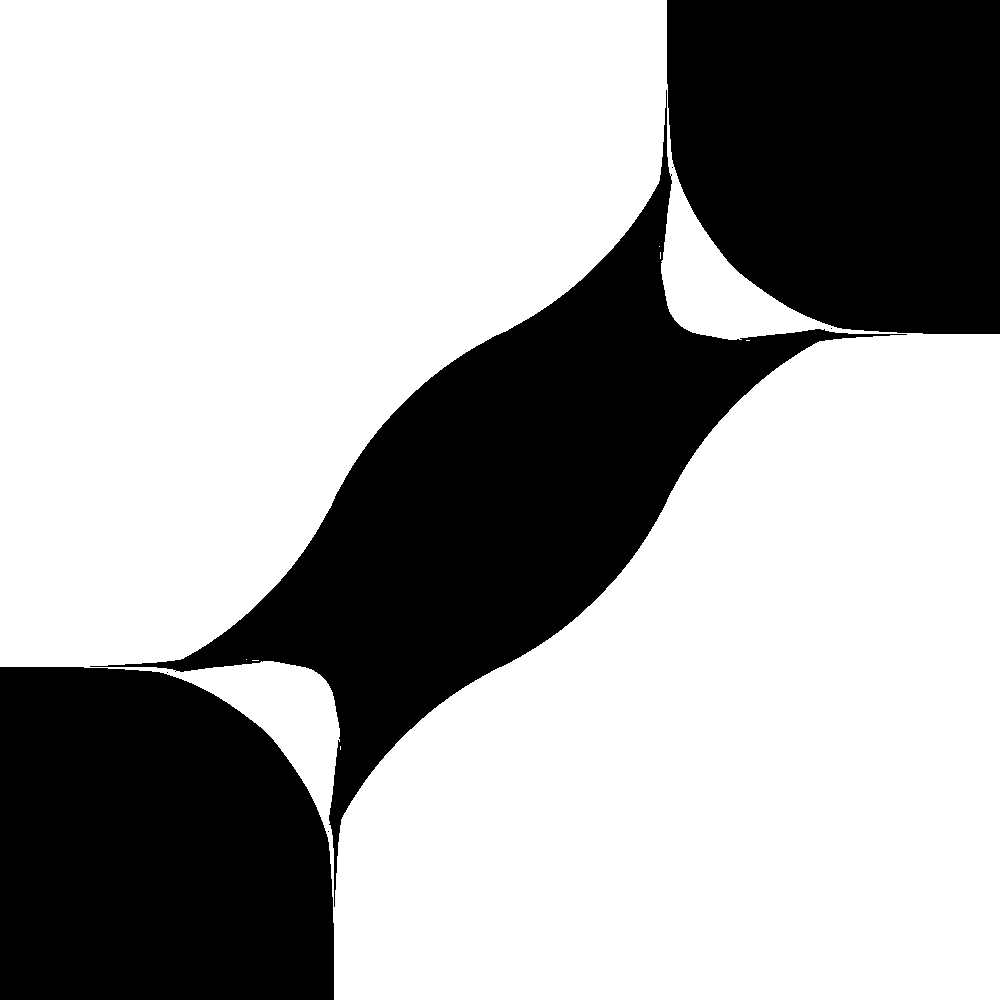}
    \caption{The  support of a solution to a capacity constrained problem (see \cite{McCannKor}, \cite{McCannKor2}). Compare to \cref{fig:uniformbound}.}
    \label{fig:caterpillar}
\end{figure}
or its complement. This set is a solution to a capacity constrained problem and appeared for the first time
in \cite{McCannKor2}: find a  function $0 \le h \le 3$ on $[0,1]^2$ maximizing integral
\[
\int_{A} xy  h(x,y) \ dx dy
\] 
such that $h(x,y) dx dy$ has the Lebesgue projections onto both axes. Then the solution $h$ takes values in $\{0,3\}$
and $\{h=3\}$ is the body on \cref{fig:caterpillar}.
We leave to the reader as an exercise the precise construction relating these two problems.
It seems to be a highly nontrivial task to give the precise description of this figure. This is especially
difficult, because numerical experiments demonstrate that it coincides up to a very small set with a figure, which boundary is piecewise smooth and can be
parametrized by  piecewise elementary functions (polynomials).

Among the other problems which can be ``embedded'' into the linearly constrained transportation problem let us mention
the martingale transportation problem \cite{HL,BeigJuill}, problems with symmetries \cite{GhM,KZ1,KZ2}.

Finally, there is a connection between the multistochastic problem and the  transportation problem with convex constraints, in particular, problems on the space of measures
with given ordering. In particular, in the $(3,2)$-problem with the cost function $xyz$ the natural ordering on the space of measure
is stochastic ordering, i.e. for two measures $\mu, \nu$ on the plane we say that $\mu$ is bigger than $\nu$ if the distribution function $F_{\mu}$
is bigger than $F_{\nu}$ (see \cref{stochdom}). We plan to study the related modified $(3,2)$-problem in the subsequent work. Here we just mention that there are
many recent paper with very interesting results dealing with convex ordering and optimal transportation, see \cite{GRST2,GozlanJuillet}.

\section{Existence of a uniting measure for \texorpdfstring{$(n,k)$}{(n, k)}-problem.}

\subsection{Setting of the problem, basic facts}

Unlike the classical Monge--Kantorovich problem, existence of a uniting measure for a $(n,k)$-problem is a nontrivial task.
In the multimarginal Monge--Kantorovich problem, which is a particular case of $(n,k)$-problem with $k=1$, 
the uniting measure always exists: this is  $\prod_{i = 1}^n\mu_i$. 
In the case of  $(n, k)$-problem one has the following necessary  condition:

\begin{proposition}
Assume that the set  $\Pi(\mu_\alpha)$ is not empty. Let $\mu \in \Pi(\mu_\alpha)$ be arbitrary 
uniting measure. Then for all  $\alpha, \beta \in \mathcal{I}_{nk}$ the following relation holds:
\[
\prj_{\alpha \cap \beta}(\mu_\alpha) = \prj_{\alpha \cap \beta}(\mu_\beta) = \prj_{\alpha \cap \beta}(\mu).
\]
\end{proposition}

\begin{definition}\label{weak_condition}
We say that the set of measures $\mu_\alpha \in \mathcal{P}(X_\alpha)$ is   \textit{consistent}, if it satisfies  
$\prj_{\alpha \cap \beta}(\mu_\alpha) = \prj_{\alpha \cap \beta}(\mu_\beta)$
for all  $\alpha, \beta \in \mathcal{I}_{nk}$.
\end{definition}

The consistency assumption for $n = 3$, $k = 2$ was considered in \cite{GlKoZi}. In what follows, we consider only consistent sets of measures. 
For a consistent set the measures  $\mu_\alpha$ are well-defined for all $\alpha \in \mathcal{I}_{nt}$, where $t \le k$. Indeed, denote $\mu_\alpha = \prj_\alpha(\mu_\beta)$ for arbitrary $\beta \in \mathcal{I}_{nk}$ containing $\alpha$. 
The consistency assumption implies that the result is independent of the choice of  $\beta$.

\begin{proposition}
Unlike the multimarginal problem, the consistency assumption is not sufficient for  $1 < k < n$.
\end{proposition}
\begin{proof} 
Let $X_i = \{0, 1, \dots, k - 1\}$ for all $1 \le i \le n$. For every $\alpha \in \mathcal{I}_{nk}$ let us construct  the corresponding measure $\mu_\alpha$ on the set $X_\alpha$. If $\alpha = \{i_1, i_2, \dots, i_k\}$, then every point of $X_\alpha$ is given by coordinates $x = (x_{i_1}, x_{i_2}, \dots, x_{i_k})$, where
 $x_{i_t} \in \{0, 1, \dots, k - 1\}$ for all $1 \le t \le k$. Set $\mu_\alpha(x) = k^{1 - k}$, if $\sum_{t = 1}^k x_{i_t} \equiv 1 \Mod k$ and $\mu_\alpha(x) = 0$ in the opposite case.

It is easy to check that the consistency assumption of \cref{weak_condition} holds: the projection of any measure $\mu_\alpha$ onto $X_\beta$ is uniform if $|\beta|<|\alpha|$. 
Assume that a uniting measure $\mu$ exists. Since the projections are non-zero, $\mu$ is not zero itself. Take a point $x = (x_1, x_2, \dots, x_n)$ such that $\mu(x) > 0$. Then for all  $\alpha = \{i_1, \dots, i_k\} \in \mathcal{I}_{nk}$ 
the relation $\sum_{t = 1}^kx_{i_t} \equiv 1 \Mod k$ holds, in the opposite case the $\mu$-mass of the projection  of  $(x_{i_1}, x_{i_2}, \dots, x_{i_k})$ onto  $X_\alpha$ is zero, hence projection of  $\mu$ does not coincide with $\mu_\alpha$.

We extract from condition $\sum_{t = 1}^kx_{i_t} \equiv 1 \Mod k$, which holds for all $\{i_1, \dots, i_k\} \in \mathcal{I}_{nk}$, $k<n$ that  $x_i \equiv x_j \Mod k$ for all $1 \le i, j \le n$. Then $\sum_{t = 1}^kx_{i_t} \equiv 0 \not\equiv 1 \Mod k$.
We obtain a contradiction.
\end{proof}

Another example for  $n = 3$, $k = 2$ the reader can find in \cite{GlKoZi}.

\subsection{Existence of a signed measure}

It follows from the previous proposition that the consistency assumption is not sufficient for existence of a uniting measure.
Nevertheless, it is sufficient for existence of a signed measure.

Let $\nu_i \in \mathcal{P}(X_i)$ be an arbitrary family of probability measures. 

\begin{definition}
For all $\alpha \in \mathcal{I}_{nt}$, $0 \le t \le k$ let us extend  $\mu_\alpha$ to $X$
in the following way: $\widetilde{\mu}_\alpha = \mu_{\alpha} \times \prod_{i \not\in \alpha}\nu_i$. 
In addition, set $\widetilde{\mu}_t = \sum_{\alpha \in \mathcal{I}_{nt}}\widetilde{\mu}_\alpha$, where $0 \le t \le k$.
\end{definition}

The following theorem contains a construction of a uniting signed measure.

\begin{theorem}\label{explicit_signed_construction}
There exists a linear combination $\mu = \sum_{t = 0}^k\lambda_t\widetilde{\mu}_t$ satisfying $\prj_\alpha(\mu) = \mu_\alpha$ for all $\alpha \in \mathcal{I}_{nk}$. 
The coefficients of $\lambda_t$ do not depend on  the choice of  $\nu_i$.
\end{theorem}

\begin{proof}
Fix  $\alpha \in \mathcal{I}_{nk}$. 
Introduce the following notations: 
\begin{align*}
    &\widetilde{\mu}^\alpha_\beta = \mu_{\beta} \times \prod_{\substack{i \not\in \beta \\ i \in \alpha}}\nu_i, \beta \subset \alpha, \\
    &\widetilde{\mu}^\alpha_t = \sum_{\substack{\beta \in \mathcal{I}_{nt} \\ \beta \subset \alpha}}\widetilde{\mu}^\alpha_\beta.
\end{align*}

For arbitrary $\beta \in \mathcal{I}_{nt}$, where $t \le k$, find a projection  $\widetilde{\mu}_\beta$ 
onto  $X_\alpha$. It is easy to realise that one obtains   $\widetilde{\mu}^\alpha_{\beta \cap \alpha}$. Let us project  $\widetilde{\mu}_t$ onto $X_\alpha$.
Applying definition of  $\widetilde{\mu}_t$ one can get
\[
\prj_\alpha(\widetilde{\mu}_t) = \sum_{\beta \in \mathcal{I}_{nt}}\widetilde{\mu}^\alpha_{\beta \cap \alpha} = \sum_{i = 0}^t\sum_{\substack{\gamma \in \mathcal{I}_{ni} \\ \gamma \subset \alpha}}\binom{n - k}{t - i}\widetilde{\mu}^\alpha_{\gamma} = \sum_{i = 0}^t\binom{n - k}{t - i}\widetilde{\mu}_i^\alpha.
\]

Thus we express $\prj_\alpha(\widetilde{\mu}_t)$ through $\widetilde{\mu}_i^\alpha$ 
with fixed coefficients.  
We get the following system of linear equations
\[
\sum_{t = 0}^k\lambda_t\prj_\alpha(\widetilde{\mu}_t) = \widetilde{\mu}^\alpha_k = \mu_{\alpha}.
\]
on $\lambda_t$.
The coefficient of $\widetilde{\mu}_i^\alpha$ equals  $0$ for  $i > t$ and equals $1$ for  $i = t$. Thus, the given system 
corresponds to a triangular matrix with units on the diagonal. This means that there exist a unique set of numbers $\lambda_t$, $0 \le t \le k$, satisfying 

In addition, we observe that these coefficients do not depend on $\alpha$. Thus, the signed measure $\sum_{t = 0}^k\lambda_t\widetilde{\mu}_t$ is uniting.
\end{proof}

\begin{example}
Let us give an example in the  $(3,2)$-case. One has
\begin{align*}
    &\widetilde{\mu}_0 = \nu_1 \times \nu_2 \times \nu_3,\\
    &\widetilde{\mu}_1 = \mu_1 \times \nu_2 \times \nu_3 + \nu_1 \times \mu_2 \times \nu_3 + \nu_1 \times \nu_2 \times \mu_3,\\
    &\widetilde{\mu}_2 = \mu_{12} \times \nu_3 + \mu_{13}\times \nu_2 + \mu_{23} \times \nu_1.
\end{align*}
The projections of these measures onto  $X_1 \times X_2$ are given by
\begin{align*}
    \prj_{12}(\widetilde{\mu}_0) &= \nu_1 \times \nu_2,\\
    \prj_{12}(\widetilde{\mu}_1) &= \prj_{12}(\mu_1 \times \nu_2 \times \nu_3) + \prj_{12}(\nu_1 \times \mu_2 \times \nu_3) + \prj_{12}(\nu_1 \times \nu_2 \times \mu_3)\\
    &=\mu_1 \times \nu_2 + \nu_1 \times \mu_2 + 1 \cdot \nu_1 \times \nu_2,\\
    \prj_{12}(\widetilde{\mu}_2) &= \prj_{12}(\mu_{12} \times \nu_3) + \prj_{12}(\mu_{13}\times \nu_2) + \prj_{12}(\mu_{23} \times \nu_1)\\
    &= \mu_{12} + \mu_1 \times \nu_2 + \nu_1 \times \mu_2.
\end{align*}

Thus for arbitrary coefficients  $\lambda_0$, $\lambda_1$, $\lambda_2$ one can find projection of $\lambda_0\widetilde{\mu}_0 + \lambda_1\widetilde{\mu}_1+\lambda_2\widetilde{\mu}_2$ onto $X_1 \times X_2$:
\[
\prj_{12}(\lambda_0\widetilde{\mu}_0 + \lambda_1\widetilde{\mu}_1+\lambda_2\widetilde{\mu}_2) = (\lambda_0 + \lambda_1)\nu_1 \times \nu_2 + (\lambda_1 + \lambda_2)(\mu_1 \times \nu_2 + \nu_1 \times \mu_2) + \lambda_2\mu_{12}.
\]

In order to have equality $\prj_{12}(\lambda_0\widetilde{\mu}_0 + \lambda_1\widetilde{\mu}_1+\lambda_2\widetilde{\mu}_2) = \mu_{12}$ it is sufficient to require $\lambda_0 + \lambda_1 = 0$, $\lambda_1 + \lambda_2 = 0$, $\lambda_2 = 1$. 
This system has a unique solution  $\lambda_0 = 1$, $\lambda_1 = -1$, $\lambda_2 = 1$. Thus $\prj_{12}(\widetilde{\mu}_0 - \widetilde{\mu}_1 + \widetilde{\mu}_2) = \mu_{12}$. 
By the reason of symmetry   $\prj_{13}(\widetilde{\mu}_0 - \widetilde{\mu}_1 + \widetilde{\mu}_2) = \mu_{13}$ and $\prj_{23}(\widetilde{\mu}_0 - \widetilde{\mu}_1 + \widetilde{\mu}_2) = \mu_{23}$.

\end{example}

\subsection{Dual condition for existence of a uniting measure.}

The following existence criterion for uniting measure is a particular case of a result obtained by Kellerer in \cite{Kellerer64}. We give an independent proof based on the use of  the minimax theorem.

\begin{theorem}
Let $X_1$, $X_2$, \dots, $X_n$ be compact metric spaces and let $\mu_\alpha \in \mathcal{P}(X_\alpha)$, $\alpha \in \mathcal{I}_{nk}$ be a fixed family of measures. 
Then  $\Pi(\mu_\alpha)$  is not empty if and only if for every set of functions   $f_\alpha \in L^1(X_\alpha, \mu_\alpha)$ satisfying assumption $\sum_{\alpha \in I_{nk}}f_\alpha(x_\alpha) \ge 0$ for all  $x \in X$ the following inequality holds: 
\[
\sum_{\alpha \in \mathcal{I}_{nk}}\int_{X_\alpha}f_\alpha~d\mu_\alpha \ge 0.
\]
\end{theorem}

\begin{proof}
The existence of a uniting measure trivially implies the inequality. If  $\mu \in \Pi(\mu_\alpha)$ and the set of functions  $f_\alpha$ satisfies the assumption of the theorem, the function 
 $F(x) = \sum_{\alpha \in \mathcal{I}_{nk}}f_\alpha(x_\alpha)$ is integrable with respect to $\mu$ and the following inequality holds:
\[
\sum_{\alpha \in \mathcal{I}_{nk}}\int_{X_\alpha}f_\alpha~d\mu_\alpha = \int_X F~d\mu \ge \int_X 0~d\mu = 0.
\]

Let us prove the theorem in the other direction.
Assume that the set of measures $\mu_\alpha$ does not satisfy assumptions of \cref{weak_condition}. Then there exists  $\alpha, \beta \in \mathcal{I}_{nk}$, such that the measures  
$\nu_1 = \prj_{\alpha \cap \beta}\mu_\alpha$ and $\nu_2 = \prj_{\alpha \cap \beta}\mu_\beta$ are different. Let $A$ be a subset of $X_{\alpha \cap \beta}$ satisfying $\nu_1(A) < \nu_2(A)$. 
Set: $f_\alpha(x_\alpha) = 1$ if $x_{\alpha \cap \beta} \in A$ and $0$ in the opposite case. In addition, set $f_\beta(x_\beta) = -1$ if $x_{\alpha \cap \beta} \in A$, and $0$ in the opposite case; $f_\gamma(x_\gamma) = 0$, if  $\gamma \not\in \{\alpha, \beta\}$. 
Then  $\sum_{\gamma \in \mathcal{I}_{nk}}f_\gamma(x_\gamma) = 0$ for all  $x \in X$. On the other hand 
\[
\sum_{\gamma \in \mathcal{I}_{nk}}\int_{X_\gamma}f_\gamma~d\mu_\gamma = \int_{X_\alpha}f_\alpha~d\mu_\alpha + \int_{X_\beta}f_\beta~d\mu_\beta = \nu_1(A) - \nu_2(A) < 0.
\]

Thus, one can assume without loss of generality that the set of measures  $\mu_\alpha$  satisfies \cref{weak_condition}. 
We apply the following version of the minimax theorem (see \cite{Brezis}, \cite{Villani}):
\begin{theorem}[Fenchel-Rockafellar Duality]
Let $E$ be a normed vector space and $E^*$ be the corresponding dual space. Consider convex functions  $\Phi$ and $\Psi$ on $E$, taking values in $\mathbb{R} \cup \{+\infty\}$.
Let $\Phi^*$ and $\Psi^*$ be the corresponding Legendre transforms. In addition, assume that there exists  $z \in E$ satisfying $\Phi(z) < +\infty$, $\Psi(z) < +\infty$. Then  
\[
\inf_{E}[\Psi + \Phi] = \max_{z \in E^*}[-\Phi^*(-z) - \Psi^*(z)].
\]
\end{theorem}

Let  $E$ be the space of continuous (bounded) functions on  $X$ equipped with the uniform convergence norm  $||~||_{\infty}$. According to Radon theorem  $E^*$ is the space of finite signed measures on  $X$ equipped with the full variation norm.
Set: 
\begin{align*}
\Phi: u \in C_b(X) \to 
\begin{cases}
0 \text{, if } u \ge 0, \\
+\infty \text{ otherwise}.
\end{cases}
\end{align*}
\begin{align*}
\Psi: u \in C_b(X) \to 
\begin{cases}
\sum_{\alpha \in \mathcal{I}_{nk}}\int_{X_\alpha}u_\alpha~d\mu_\alpha \text{, if } u(x) = \sum_{\alpha \in \mathcal{I}_{nk}}u_\alpha(x_\alpha), \\
+\infty \text{ otherwise}.
\end{cases}
\end{align*}

Function $\Psi$ is well-defined, indeed, if  $\mu$ is a signed measure satisfying $\prj_\alpha\mu = \mu_\alpha$ for all $\alpha \in \mathcal{I}_{nk}$, then $\int_X u~d\mu = \sum_{\alpha \in \mathcal{I}_{nk}}\int_{X_\alpha}u_\alpha~d\mu_\alpha$. 
The signed measure  $\mu$ exists by \cref{explicit_signed_construction}. It is easy to check that functions $\Psi$ and $\Phi$ are convex; in addition, function  $u \equiv 1$ satisfies assumptions of the minimax theorem. Thus, the following equality holds:
\[
\inf_{E}[\Psi + \Phi] = \max_{z \in E^*}[-\Phi^*(-z) - \Psi^*(z)].
\]
It is easy to check that
\[\inf_E[\Phi + \Psi] = \inf_{\sum f_\alpha \ge 0}\sum\int_{X_\alpha}f_\alpha~d\mu_\alpha.\]

Let us find $\Phi^*(-\pi)$. 
\[
\Phi^*(-\pi) = \sup_{u \ge 0}\left[-\int_X u~d\pi\right] = -\inf_{u \ge 0}\int_X u~d\pi
\]
If  $\pi$ is nonnegative, then $\int_X u~d\pi \ge 0$ for all $u \ge 0$. Otherwise $\int_X u~d\pi$ can take arbitrary small values. Hence
\begin{align*}
    \Phi^*(-\pi) = \begin{cases}
    0 \text{, if } \pi \ge 0,\\
    +\infty \text{, otherwise}.
    \end{cases}
\end{align*}
In the same way we check   that 
\begin{align*}
    \Psi^*(\pi) = \begin{cases}
    0 \text{, if } \prj_\alpha\pi = \mu_\alpha,\\
    +\infty \text{, otherwise}.
    \end{cases}
\end{align*}
Thus the maximum $\max_{\pi \in E^*}[-\Phi^*(-\pi) - \Psi^*(\pi)]$ equals $0$, if there exists a nonnegative uniting measure, otherwise it equals $-\infty$. 
In particular, if a uniting measure does not exist, then $\inf_{\sum f_\alpha \ge 0}\sum\int_{X_\alpha}f_\alpha~d\mu_\alpha = -\infty$. Hence there exist  continuous functions $f_\alpha$ satisfying $\sum\int_{X_\alpha}f_\alpha~d\mu_\alpha < 0$.
\end{proof}

\subsection{Sufficient condition for existence of a uniting measure}

Let us mention the following trivial sufficient condition for existence of uniting measure.
\begin{proposition}
Assume that there exists a family of measures $\nu_i \in \mathcal{P}(X_i)$, $1 \le i \le n$, such that  $\mu_\alpha = \prod_{i \in \alpha}\nu_i$  $\alpha \in \mathcal{I}_{nk}$ . Then the set  $\Pi(\mu_\alpha)$ is non-empty and  $\prod_{i = 1}^n \nu_i$ is a uniting measure.
\end{proposition}

We generalize this sufficient condition using \cref{explicit_signed_construction}.

\begin{theorem}[Density condition] \label{thm:density_condition} For given natural numbers $1 \le k <n$
there exists a constant $\lambda_{nk} > 1$
which admits the following property.

Assume we are given a consistent family of probability measures $\mu_\alpha \in \mathcal{P}(X_\alpha)$,  $\alpha \in \mathcal{I}_{nk}$, and another family of probability measures  $\nu_i \in \mathcal{P}(X_i)$, $1 \le i \le n$.  
Assume that  every measure  $\mu_\alpha$,  $\alpha \in \mathcal{I}_{nk}$, is absolutely continuous with respect to  $\nu_\alpha =  \prod_{i \in \alpha}\nu_i$:
\[
\mu_\alpha = \rho_{\alpha} \cdot \nu_{\alpha}.
\]
Finally, assume that  there exist constants $0 < m \le M$ such that every density $\rho_{\alpha}$ satisfies $m \le \rho_\alpha \le M$ $\nu_\alpha$-almost everywhere for all $\alpha \in \mathcal{I}_{nk}$.

Then  $\Pi(\mu_\alpha)$ is not empty provided  $\frac{M}{m} \le \lambda_{nk}$.
\end{theorem}
\begin{proof}
The definition of $m$ implies that  $\mu_\alpha - m \cdot \nu_\alpha$  is a nonnegative measure for all  $\alpha \in \mathcal{I}_{nk}$, hence $m \le 1$, because  both $\mu_\alpha$ and $\nu_\alpha$ are probability measures.
In addition, if $m = 1$, the  $\mu_\alpha - \nu_\alpha = 0$  for all $\alpha \in \mathcal{I}_{nk}$. In this case  the measure $\nu = \prod_{i = 1}^n\nu_i$ is uniting.

Consider the case $m < 1$. Note that   $\mu'_\alpha = (\mu_\alpha - m\cdot\nu_\alpha) / (1 - m)$ is a probability measure for all $\alpha \in \mathcal{I}_{nk}$, which is absolutely continuous 
with respect to  $\nu_\alpha$ and its density is bounded from above by  $\frac{m}{1 - m}(\lambda_{nk} - 1) > 0$. In addition, the family of measures   $\mu'_\alpha$ satisfies consistency condition.
\cref{explicit_signed_construction} implies that given measures  $\nu_i$ and $\mu'_\alpha$ one can construct a family of measures $\widetilde{\mu}'_t$ and find numbers $\lambda_t$ such that the signed measure $\sum_{t = 0}^k \lambda_t \widetilde{\mu}'_t$ is uniting.
Note that  $\mu'_\alpha$ is absolutely continuous with respect to $\nu_\alpha$  for all  $\alpha \in \mathcal{I}_{nt}$,  $1 \le t \le k$,
moreover, its density is bounded from above by   $\frac{m}{1 - m}(\lambda_{nk} - 1)$. This means that the same condition holds for  $\widetilde{\mu}'_\alpha$, where we consider the corresponding density with respect to 
 $\nu = \prod_{i = 1}^n\nu_i$. Hence $\widetilde{\mu}'_t$ is absolutely continuous with respect to  $\nu$ and its density is bounded almost everywhere by  $\binom{n}{t}\cdot \frac{m}{1 - m}(\lambda_{nk} - 1)$.

We infer from this  that the density of the signed uniting measure  $\mu' = \sum_{t = 0}^k\lambda_k\widetilde{\mu}'_t$ is bounded from below by  $-\sum_{t = 0}^k|\lambda_t|\binom{n}{t}\frac{m}{1 - m}(\lambda_{nk} - 1) = -C \cdot \frac{m}{1 - m}(\lambda_{nk} - 1)$, where $C$ depends on $(n, k)$ only.

Let us prove that the assertion of the theorem holds for  $\lambda_{nk} = 1 + \frac{1}{C}$.  For the set of measures  $\mu'_\alpha$ we constructed a uniting signed measure  $\mu'$ which density with respect to  $\nu$ is almost everywhere bounded from below by number  $-C \cdot \frac{m}{1 - m}(\lambda_{nk} - 1) = -\frac{m}{1 - m}$. 
Then  $\mu = (1 - m)\mu' + m\nu$ is a uniting measure for the family  $\mu_\alpha$, and its density is nonnegative  $\nu$-almost everywhere, hence $\mu$ is nonnegative.
\end{proof}
Thus we obtained a sufficient condition for existence of uniting measure for a wide class of functions. Moreover, the uniting measure obtained in \cref{thm:density_condition} admits a bounded density. 
However, it is often helpful to require density to be bounded away from zero.

\begin{definition}
We say that measures $\mu$ and $\nu$ on the same measurable space  $(X, \mathcal{F})$ are  \textit{uniformly equivalent}, if  there exists a Radon--Nicodym density  $\rho$ of $\mu$ with respect to  $\nu$,  which is bounded from above and from below by positive constants:
 $0 < m \le \rho(x) \le M$ for all  $x \in X$.
\end{definition}

In particular, uniformly equivalent measures are absolutely continuous with respect to each other. Following the proof of  \cref{thm:density_condition} one can easily check

\begin{theorem}[Uniformly equivalent density condition] \label{thm:density_condition_eqv}
Under assumption of \cref{thm:density_condition} there exists constant  $\widehat{\lambda}_{nk} > 1$ with the following property. If all  $\alpha \in \mathcal{I}_{nk}$ satisfy $m \le \rho_\alpha \le M$ $\nu_\alpha-$almost everywhere and $\frac{M}{m} \le \widehat{\lambda}_{nk}$, 
then the set  $\Pi(\mu_\alpha)$ contains at least one measure which is uniformly equivalent to  $\prod_{i = 1}^n\nu_i$.
\end{theorem}

\subsection{Estimates for \texorpdfstring{$(3,2)$}{(3, 2)}-case}
In the $(3,2)$-case one can obtain explicit estimates on the optimal value of  $\lambda_{32}$ from \cref{thm:density_condition}.

\begin{proposition}
For $\lambda_{32} > 2$ the conclusion of \cref{thm:density_condition} does not hold.
\end{proposition}
\begin{proof}
Let $X_1 = X_2 = X_3 = \{0, 1\}$ and let every $\nu_i$ be the uniform probability measure on $X_i$. Let us construct measures  $\mu_{12}, \mu_{13}, \mu_{23}$ on spaces  $X_1 \times X_2$, $X_1 \times X_3$ and $X_2 \times X_3$ respectively. Set $\mu_{ij}(x_i, x_j) = M$, if $x_i + x_j = 1$;  and $\mu_{ij}(x_i, x_j) = m$ otherwise. Here $m$ and $M$ 
are nonegative constants  such that $\frac{M}{m} = \lambda_{32}$ and $\mu_{ij}(X_i \times X_j) = 1$.

Assume that a  uniting measure $\mu$ exists. Consider the following sums:
\begin{align*}
A = 6m = &~\mu_{12}(0, 0) + \mu_{12}(1, 1) + \mu_{13}(0, 0) + \mu_{13}(1, 1) + \mu_{23}(0, 0) + \mu_{23}(1, 1) \\
= &~3 \mu(0, 0, 0) + \mu(1, 0, 0) + \mu(0, 1, 0) + \mu(0, 0, 1) \\
& + \mu(1, 1, 0) + \mu(1, 0, 1) + \mu(0, 1, 1) + 3 \mu(1, 1, 1),\\
B= 6M =&~\mu_{12}(0, 1) + \mu_{12}(1, 0) + \mu_{13}(0, 1) + \mu_{13}(1, 0) + \mu_{23}(0, 1) + \mu_{23}(1, 0) \\
=&~2\mu(1, 0, 0) + 2\mu(0, 1, 0) + 2\mu(0, 0, 1) + 2\mu(1, 1, 0) + 2\mu(1, 0, 1) + 2\mu(0, 1, 1). 
\end{align*}

On one hand $2A < B$, because $2m < M$. On the other hand, analyzing expressions on the right-hand sides we see that  $2A \ge B$. We get a contradiction.
\end{proof}

\begin{proposition}\label{thm:density_condition1}
The conclusion of \cref{thm:density_condition_eqv} holds for  $\widehat\lambda_{32} = \frac{3}{2}$. In particular, there exists a uniting measure  $\mu$, which is uniformly equivalent to $\nu = \nu_1 \times \nu_2 \times \nu_3$.
\end{proposition}
\begin{proof}
Let  $0 < m \le M$ be constants from  \cref{thm:density_condition}: $m \le \rho_{ij} \le M$ for all $1 \le i < j < M$ $\nu_{ij}$-almost everywhere.
Clearly,  $m \le 1 \le M$. If  $m = 1$ or $M = 1$, then $\mu_{ij} = \nu_{ij}$, this means that  $\nu$ is a uniting measure itself.

For  $m < 1 < M$, the following measure is uniting:
\begin{align*}
\mu =&~4\mu_1 \times \mu_2 \times \mu_3 - 2\left(\nu_1 \times \mu_2 \times \mu_3 + \mu_1 \times \nu_2 \times \mu_3 + \mu_1 \times \mu_2 \times \nu_3\right)\\
&+ 2\left(\mu_{12} \times \nu_3 + \mu_{13} \times \nu_2 + \mu_{23} \times \nu_1\right) - \left(\mu_{12} \times \mu_3 + \mu_{13} \times \mu_2 + \mu_{23} \times \mu_1\right).
\end{align*}

Let us check that  $\mu$ is nonnegative. To this end we prove that its density with respect to  $\nu = \nu_1 \times \nu_2 \times \nu_3$ is nonnegative almost everywhere. The density of  $\mu$ with respect to  $\nu$ has the form
\begin{align*}
    \frac{d\mu}{d\nu}(x_1, x_2, x_3) =&~ 4\rho_1(x_1)\rho_2(x_2)\rho_3(x_3) - 2\left(\rho_1(x_1)\rho_2(x_2) + \rho_1(x_1)\rho_3(x_3) + \rho_2(x_2)\rho_3(x_3)\right) \\
    &+2\left(\rho_{12}(x_1, x_2) + \rho_{13}(x_1, x_3) + \rho_{23}(x_2, x_3)\right)\\ &-\left(\rho_{12}(x_1, x_2)\rho_3(x_3) + \rho_{13}(x_1, x_3)\rho_2(x_2) + \rho_{23}(x_2, x_3)\rho_1(x_1)\right).
\end{align*}

Assumption  $m \le \rho_{ij}(x_i, x_j) \le M$ implies that, for  $\nu_i$-almost all $x_i$ the inequality  $m \le \rho_i(x_i) \le M$ holds, where $\rho_i = \frac{d\mu_i}{d\nu_i}$. 
The assumption of the theorem implies $1 < M \le \widehat{\lambda}_{32}m = \frac{3}{2}m$.  Thus, it is sufficient to check inequality
\[
4p_1p_2p_3 - 2(p_1p_2 + p_1p_3 + p_2p_3) + 2(p_{12} + p_{13} + p_{23}) - (p_1p_{23} + p_2p_{13} + p_3p_{12}) \ge 0
\]
for all $m \le p_i, p_{ij} \le \frac{3}{2}m$, $\frac{2}{3} < m < 1$, 
and for the proof of uniform boundedness it is sufficient to prove that there exists constant  $\eps(m) > 0$ such that  
\[
4p_1p_2p_3 - 2(p_1p_2 + p_1p_3 + p_2p_3) + 2(p_{12} + p_{13} + p_{23}) - (p_1p_{23} + p_2p_{13} + p_3p_{12}) \ge \eps(m).
\]

This expression is linear in every variable $p_i, p_{ij}$, thus for every fixed  $m$ every variable equals  $m$ or $\frac{3}{2}m$ at the minimum point. The coefficient of $p_{ij}$ equals $2 - p_k > 0$ provided $p_k \le \frac{3}{2}m < \frac{3}{2}$,
hence this function is increasing in  $p_{ij}$. Then at the minimum point one has $p_{ij} = m$ for all $1 \le i, j \le 3$. 
Finally, we reduce the proof to the following inequality we have to check:
\[
4p_1p_2p_3 - 2(p_1p_2 + p_1p_3 + p_2p_3) - m(p_1 + p_2 + p_3) + 6m \ge \eps(m)
\]
for all $m \le p_i \le \frac{3}{2}m$, $\frac{2}{3} < m < 1$.

Since the function is symmetric we have to check the following inequalities:
\begin{enumerate}
    \item $p_1 = p_2 = p_3 = m$: $4m^3 - 9m^2 + 6m > 0$ if $\frac{2}{3} < m < 1$;
    \item $p_1 = \frac{3}{2}m$, $p_2 = p_3 = m$: $6m^3 - \frac{23}{2}m^2 + 6m > 0$ if $\frac{2}{3} < m < 1$;
    \item $p_1 = p_2 = \frac{3}{2}m$, $p_3 = m$: $9m^3 - \frac{29}{2}m^2 + 6m > 0$ if $\frac{2}{3} < m < 1$;
    \item $p_1 = p_2 = p_3 = \frac{3}{2}m$: $\frac{27}{2}m^3 - 18m^2 + 6m > 0$ if $\frac{2}{3} < m < 1$.
\end{enumerate}
Every inequality can be easily checked and we complete the proof of nonnegativity of  $\mu$ and its uniform equivalence to  $\nu$.

It remains to check that  $\mu$ is uniting for  $\mu_{ij}$: 
\begin{align*}
\prj_{12}(\mu) =&~ 4\mu_1 \times \mu_2 - 2 \nu_1 \times \mu_2 - 2 \mu_1 \times \nu_2 - 2\mu_1 \times \mu_2 \\
&+ 2 \mu_{12} + 2 \mu_1 \times \nu_2 + 2 \nu_1 \times \mu_2 - \mu_{12} - \mu_1 \times \mu_2 - \mu_1 \times \mu_2 = \mu_{12}.
\end{align*}
In the same way we check that the desired identities hold for other projections. \end{proof}

One can prove another estimate for $\lambda_{32} = 2$. Unfortunately, the arguments in our proof can not be used to prove uniform equivalence of  $\mu$ and $\nu$.
\begin{proposition}\label{prop:density_condition2}
For the value $\lambda_{32} = 2$ the conclusion of \cref{thm:density_condition} holds.
\end{proposition}
\begin{proof}

Let  $0 < m \le M$ be constants from \cref{thm:density_condition}. Consider the following set:  
\[\Delta = \left\{\xi \text{ --  nonnegative measure on  $X_1 \times X_2 \times X_3$}\colon m \le \frac{d(\mu_{ij} - \prj_{ij}(\xi))}{d\nu_{ij}}\right\}.
\]
This set is not empty because it contains the trivial (zero) measure. In addition, $\Delta$ is weakly closed. From assumption $m \le \frac{d(\mu_{ij} - \prj_{ij}(\xi))}{d\nu_{ij}}$ we infer that  $\mu_i \ge \prj_i(\xi)$, 
hence  $\Delta$ is uniformly tight and the variations of measures from $\Delta$ are uniformly bounded. Then the Prokhorov theorem implies that  $\Delta$ is weakly compact. 
Hence there exists an extreme measure  $\xi_{\max}$, where functional $\xi_{\max}(X)$ attains its maximum.

\begin{lemma}\label{lem:extreme_measures_property}
For $\nu$-almost all $x \in X$ at least one of the numbers \[
\frac{d(\mu_{ij} - \prj_{ij}(\xi_{\max}))}{d\nu_{ij}}(x_i, x_j),\;1 \le i, j \le 3
\] 
equals $m$.
\end{lemma}

\begin{proof}
Assume the converse. Then there exists a positive number $\eps$, such that the set 
\[E_\eps = \left\{x \in X \colon \frac{d(\mu_{ij} - \prj_{ij}(\xi_m))}{d\nu_{ij}}(x_i, x_j) \ge m + \eps,\;1 \le i, j \le 3\right\}
\] 
satisfies $\nu(E_{\eps})>0$. Let $\xi_\Delta$ be the measure which density (with respect to  $\nu$) equals  $\eps$ on  $E_\eps$ and $0$  otherwise.
 It is easy to check that $\xi_{\max} + \xi_\Delta \in \Delta$, $(\xi_{\max} + \xi_\Delta)(X) > \xi_{\max}(X)$ and this contradicts to definition of $\xi_{\max}$.
\end{proof}

Consider the family of probability measures
\[
\mu_{ij}' = \frac{\mu_{ij} - \prj_{ij}(\xi_{\max})}{1 - \xi_{\max}(X)}, 1 \le i, j \le 3.
\]
Since  $\{\mu_{ij}\}$ is consistent, the family of measures $\{\mu'_{ij}\}$ is consistent too. Since $\xi_{\max} \in \Delta$, we have $m / \alpha \le d\mu'_{ij} / d\nu_{ij} \le M / \alpha$ almost everywhere, where $\alpha = 1 - \xi_{\max}(X)$. Hence, the family  $\{\mu'_{ij}\}$ satisfies assumptions of \cref{prop:density_condition2}. Moreover, if a measure $\mu'$ is uniting for $\mu'_{ij}$, then the measure $\mu = \alpha \mu' + \xi_{\max}$ is uniting for $\mu_{ij}$. Thus, it is sufficient to solve the problem only for  $\mu'_{ij}$.

Now, we replace $\mu_{ij}$ with $\mu'_{ij}$, $m$ and $M$ with $m / \alpha$ and $M / \alpha$ respectively. We may assume that densities $\rho_i = \frac{d\mu_i}{d\nu_i}$, $\rho_{ij} = \frac{d\mu_{ij}}{d\nu_{ij}}$ satisfying the following assumptions:
\begin{enumerate}
    \item \label{cond:A1}  $m \le \rho_{ij}(x_i, x_j) \le M$, $1 \le i, j \le 3$ for all $x \in X$.
    \item\label{cond:A2}  $\int_{X_j}\rho_{ij}(x_i, x_j)\,\nu_j(dx_j) = \rho_i(x_i)$ for all  $x_i \in X_i$.
    \item\label{cond:A3} For $\nu$-almost all $x \in X$ at least one of the numbers $\rho_{ij}(x_i, x_j)$, $1 \le i, j \le 3$, equals $m$.
\end{enumerate}

Assumptions  \ref{cond:A1} and \ref{cond:A2} are always fulfilled after changing  $\rho_i$ and $\rho_{ij}$  on a set of zero measure, and the last one follows from \cref{lem:extreme_measures_property}.
Under these assumptions one can prove the following lemma:

\begin{lemma}\label{lem:pre_A4_cond}
Assume that  $\rho_i, \rho_{ij}$ satisfy assumptions \ref{cond:A1}-\ref{cond:A3}. Then for  $\nu_{ij}$-almost all $(x_i, x_j) \in X_{ij}$ 
one of the following conditions holds: $\rho_{ij}(x_i, x_j) = m$ or $\rho_i(x_i) + \rho_j(x_j) \le m + M$.
\end{lemma}
\begin{proof}
Let $k \in \{1,2,3\}\setminus \{i, j\}$. Let us denote by $X^r_{ij}$ the set of couples $(x_i, x_j) \in X_{ij}$ such that for $\nu_k$-almost all $x_k \in X_k$ one of the numbers $\rho_{ij}(x_i, x_j)$, $\rho_{ik}(x_i, x_k)$ and $\rho_{jk}(x_j, x_k)$ equals $m$. 
Assumption \ref{cond:A3} implies that  $X_{ij}^r$ has  full measure with respect to $\nu_{ij}$.

Let $(x_i, x_j) \in X_{ij}^r$. Assume that $\rho_{ij}(x_i, x_j) > m$. The for  $\nu_k$-almost all $x_k \in X_k$ at least one of the numbers $\rho_{ik}(x_i, x_k)$ and $\rho_{jk}(x_j, x_k)$  equals $m$. 
In particular, $\rho_{ik}(x_i, x_k) + \rho_{jk}(x_j, x_k) \le m + M$ for $\nu_k$-almost all $x_k \in X_k$. Then we infer from  \ref{cond:A1}, \ref{cond:A2} 
\[
\rho_i(x_i) + \rho_j(x_j) = \int_{X_k} \rho_{ik}(x_i, x_k)~d\nu_k + \int_{X_k} \rho_{jk}(x_j, x_k)~d\nu_k \le m + M.
\]
\end{proof}

Changing, if necessary, density functions  $\rho_i$, $\rho_{ij}$ on a set of zero measure, we can assume, in addition, that the following holds:
\begin{enumerate}[resume]
    \item \label{cond:A4} For all $(x_i, x_j) \in X_{ij}$ one has $\rho_{ij}(x_i, x_j) = m$ or $\rho_i(x_i) + \rho_j(x_j) \le m + M$, $1 \le i, j \le 3$.
\end{enumerate}
\begin{lemma}
Let the density functions $\rho_i, \rho_{ij}$ satisfy assumptions \ref{cond:A1}-\ref{cond:A4}. Then for all $i \ne j$ and all $x_i \in X_i$ the following inequality holds:
\[
\nu_j\left(x_j \in X_j \colon \rho_j(x_j) \le m + M - \rho_i(x_i)\right) \ge \frac{\rho_i(x_i) - m}{M - m}.
\]
\end{lemma}
\begin{proof}
Fix a point $x_i \in X_i$, and denote by $A$ be the set of points $x_j \in X_j$ satisfying $\rho_{ij}(x_i, x_j) = m$. Then $\rho_i(x_i) = \int_{X_j}\rho_{ij}(x_i, x_j)~dx_j \le m\nu_j(A) + M(1 - \nu_j(A))$, which implies $\nu_j(A)~\le~\frac{M - \rho_i(x_i)}{M - m}$.

On the other hand assumption \ref{cond:A4} implies that for all $x_j \in X_j \backslash A$ the inequality $\rho_i(x_i) + \rho_j(x_j) \le m + M$ holds. Hence
 \[\nu_j\left(x_j \in X_j \colon \rho_j(X_j) \le m + M - \rho_i(x_i)\right) \ge \nu_j(X_j \backslash A) = 1 - \nu_j(A) \ge \frac{\rho_i(x_i) - m}{M - m}.\]
\end{proof}

Choosing a sequence $x_{i}^{(n)}$ such that $\rho_i(x^{(n)}_i) \to M_i = \sup_{x_i \in X_i} \rho_i(x_i)$ and passing to the limit one gets the following corollary:
\begin{corollary}\label{cons:measure_of_small_density}
Let $M_i = \sup_{x_i \in X_i} \rho_i(x_i)$. Then for all  $j \ne i$ the following inequality holds:
\[
\nu_j(x_j \in X_j \colon \rho_j(X_j) \le m + M - M_i) \ge \frac{M_i - m}{M - m}.
\]
\end{corollary}
\begin{lemma}\label{lem:strong_ineq}
Let $\rho_i, \rho_{ij}$ satisfy  assumptions \ref{cond:A1}-\ref{cond:A4} and $\frac{M}{m} \le 2$. Then inequalities  
\[
\frac{2}{3} \le m \le 1, \ \ \ p_i(x_i) \le \frac{m}{2}\left(3 + \sqrt{3 - \frac{2}{m}}\right)
\] hold for all  $x_i \in X_i$, $1 \le i \le 3$.
\end{lemma}
\begin{proof}
Let $M_i = \sup_{x_i \in X_i} \rho_i(x_i)$. Assume that  $M_1 \ge M_2$ and $M_1 \ge M_3$. It is sufficient to check that $\frac{3}{2}m \ge 1$ and $M_1 \le \frac{m}{2}\left(3 + \sqrt{3 - \frac{2}{m}}\right)$.

Assume that $\frac{3}{2}m \ge M_1$. Then, since $M_1 = \sup_{x_1 \in X_1}\rho_1(x_1)$, one has $M_1 \ge 1$. This implies $\frac{3}{2}m \ge M_1 \ge 1$. Moreover, $M_1 \le \frac{3}{2}m \le \frac{m}{2}\left(3 + \sqrt{3 - \frac{2}{m}}\right)$.

Consider the case $M_1 \ge \frac{3}{2}m$. Set $A = \{x_2 \in X_2 \colon \rho_2(x_2) \le m + M - M_1\}$. 
Then the following  holds:
\begin{align*}
1 &= \int_{X_2}\rho_2(X_2)~d\nu_2 \le (m + M - M_1)\nu_2(A) + M_2\left(1 - \nu_2(A)\right) \\
&\le (m + M - M_1)\nu_2(A) + M_1\left(1 - \nu_2(A)\right)
= (m + M - 2M_1)\nu_2(A) + M_1.
\end{align*}
\Cref{cons:measure_of_small_density} implies $\nu_2(A) \ge \frac{M_1 - m}{M - m} \ge \frac{M_1}{m} - 1$ (here we use  $M \le 2m$).
Applying this inequality and the inequality $M_1 \ge \frac{3}{2}m$ one gets
\begin{align*}
1 &\le (m + M - 2M_1)\nu_2(A) + M_1 \le (3m - 2M_1)\nu_2(A) + M_1 \\
&\le (3m - 2M_1)\left(\frac{M_1}{m} - 1\right) + M_1
=m\left(-2\left(\frac{M_1}{m}\right)^2 + 6\frac{M_1}{m} - 3\right).
\end{align*}
The function $-2x^2 + 6x - 3$ is decreasing on  $x \ge \frac{3}{2}$, hence
\[
1 \le m\left(-2\left(\frac{M_1}{m}\right)^2 + 6\frac{M_1}{m} - 3\right) \le m\left(-2\left(\frac{3}{2}\right)^2 + 6\cdot\frac{3}{2} - 3\right) = \frac{3}{2}m.
\]
Moreover, $-2\left(\frac{M_1}{m}\right)^2 + 6\frac{M_1}{m} - 3 \ge \frac{1}{m}$ , thus $\frac{M_1}{m} \le \frac{1}{2}\left(3 + \sqrt{3 - \frac{2}{m}}\right)$.
\end{proof}

Let us describe explicit constructions of uniting measures for  $m = \frac{2}{3}$ and $\frac{2}{3} < m \le 1$. If $m = \frac{2}{3}$, then $\rho_i(x_i) \le \frac{m}{2}\left(3 + \sqrt{3 - \frac{2}{m}}\right) = 1$ for all $x_i \in X_i$. 
Measures $\mu_i$ and $\nu_i$ are probability measures, $\frac{d\mu_i}{d\nu_i} \le 1$. Hence $\mu_i = \nu_i$. The desired measure is given by
\[
\mu = \mu_1\times\mu_{23} + \mu_2\times\mu_{13} + \mu_3\times\mu_{12} - 2\mu_1\times\mu_2\times\mu_3.
\]
This measure is nonnegative: $\frac{d\mu}{d\nu}(x_1, x_2, x_3) = \rho_{12}(x_1, x_2) + \rho_{13}(x_1, x_3) + \rho_{23}(x_2, x_3) - 2 \ge 0$ since $\rho_{ij}(x_i, x_j) \ge m = \frac{2}{3}$. 
In addition, it is uniting:
\[
\prj_{12}(\mu) = \mu_1 \times \mu_2 + \mu_2 \times \mu_1 + \mu_{12} - 2\mu_1 \times \mu_2 = \mu_{12},
\]
and the same for other projections.

Let us consider the case $\frac{2}{3} < m \le 1$. Set: $u = \sqrt{3 - \frac{2}{m}}$. Then $\frac{1}{m} = \frac{1}{2}(3 - u^2)$;  $u$ satisfies $0 < u \le 1$ under assumption $\frac{2}{3} < m \le 1$.
 The desired measure is given by
\begin{align*}
    \mu =&-\frac{8}{m^2u(u + 1)^3}\mu_1 \times \mu_2 \times \mu_3 + 2\frac{5u + 9}{u(u + 1)^3}\nu_1 \times \nu_2 \times \nu_3 \\
    &+ 4\frac{u + 3}{mu(u + 1)^3}\left(\nu_1 \times \mu_2 \times \mu_3 + \mu_1 \times \nu_2 \times \mu_3 + \mu_1 \times \mu_2 \times \nu_3\right)\\
    &- 2\frac{5u + 9}{u(u + 1)^3}\left(\mu_1 \times \nu_2 \times \nu_3 + \nu_1 \times \mu_2 \times \nu_3 + \nu_1 \times \nu_2 \times \mu_3 \right)\\
    &+ 2\frac{u + 2}{(u + 1)^2}\left(\mu_{23} \times \nu_1 + \mu_{13} \times \nu_2 + \mu_{12} \times \nu_3\right)\\
    &-\frac{2}{m(u + 1)^2}\left(\mu_{23} \times \mu_1 + \mu_{13} \times \mu_2 + \mu_{12} \times \mu_3\right).
\end{align*}
This measure is uniting for $\mu_{ij}$:
\begin{align*}
    \prj_{12}(\mu) =&~ \left(4\frac{u + 3}{mu(u + 1)^3} - 2\frac{5u + 9}{u(u + 1)^3} + 2\frac{u + 2}{(u + 1)^2}\right)\left(\nu_1\times\mu_2 + \mu_1\times\nu_2\right)\\
    &+\left(-\frac{8}{m^2u(u + 1)^3} + 4\frac{u + 3}{mu(u + 1)^3} - \frac{4}{m(u + 1)^2}\right)\mu_1 \times \mu_2\\
    &+\left(2\frac{5u + 9}{u(u + 1)^3} - 2\frac{5u + 9}{u(u + 1)^3}\right)\nu_1 \times \nu_2 +\left(2\frac{u + 2}{(u + 1)^2} - \frac{2}{m(u + 1)^2}\right)\mu_{12}\\
    =&~\mu_{12}.
\end{align*}
To prove the desired equality we substitute $\frac{1}{m} = \frac{1}{2}(3 - u^2)$ 
and check that all the terms are zero except the last one. In addition, the coefficient of $\mu_{12}$ equals $1$.
We do the same for the other projections.

To check nonnegativity of  $\mu$ it is sufficient to check that the following expression is nonnegative:
\begin{align*}
&-8p_1p_2p_3 + 4m(u + 3)(p_1p_2 + p_1p_3 + p_2p_3) - 2m^2(5u + 9)(p_1 + p_2 + p_3)\\
&+ 2m^2u(u + 1)(u + 2)(p_{12} + p_{13} + p_{23}) - 2mu(u + 1)(p_1p_{23} + p_2p_{13} + p_3p_{12})\\&+ 2m^2(5u + 9),
\end{align*}
where  $p_i = \rho_i(x_i)$, $p_{ij} = \rho_{ij}(x_i, x_j)$. One has  $m \le p_{ij} \le 2m$ by our assumption, $m \le p_i \le \frac{m}{2}\left(3 + \sqrt{3 - \frac{2}{m}}\right) = \frac{m}{2}(u + 3)$ and $\frac{2}{3} < m \le 1$ by \cref{lem:strong_ineq}.

This function is linear in  $p_{ij}$ with the coefficient 
\[
2m^2u(u + 1)(u + 2) - 2mu(u + 1)p_k \ge 2m^2u(u + 1)(u + 2) - m^2u(u + 1)(u + 3) \ge 0
\]
(here we use that $u \le 1$), hence  one can set $p_{ij} = m$ for all  $1 \le i, j \le 3$. In this case the expression is equal to
\begin{align*}
    &-8p_1p_2p_3 + 4m(u + 3)(p_1p_2 + p_1p_3 + p_2p_3) - 2m^2(5u + 9)(p_1 + p_2 + p_3)\\
&+ 6m^3u(u + 1)(u + 2) - 2m^2u(u + 1)(p_1 + p_2 + p_3)+ 2m^2(5u + 9)\\
=&-8p_1p_2p_3 + 4m(u + 3)(p_1p_2 + p_1p_3 + p_2p_3)-2m^2(u + 3)^2(p_1 + p_2 + p_3)\\
&+6m^3u(u + 1)(u + 2) -m^3(u^2 - 3)(5u + 9)
\\
=&~(m(u + 3) - 2p_1)(m(u + 3) - 2p_2)(m(u + 3) - 2p_3) \ge 0,
\end{align*}
this completes the proof of the well-posedness and the proof of \cref{prop:density_condition2}.
\end{proof}

One can prove many other sufficient conditions of existence of uniting measures.
One of the examples is given in the next theorem.

\begin{theorem}
Assume that a consistent family of measures  $\mu_{ij}$ sastisfies  $\mu_{ij} \ge \frac{2}{3} \mu_i \times \mu_j$, $1 \le i, j \le 3$. Then there exists a uniting measure.
\end{theorem}
\begin{proof}
The desired measure is given by
\[\mu = \left(\mu_{12} - \frac{2}{3} \mu_1 \times \mu_2\right) \times \mu_3 + 
\left(\mu_{13} - \frac{2}{3} \mu_1 \times \mu_3\right) \times \mu_2 +
\left(\mu_{23} - \frac{2}{3} \mu_2 \times \mu_3\right) \times \mu_1.\]

Indeed, one has
\[\prj_{12}(\mu) = \mu_{12} - \frac{2}{3} \mu_1 \times \mu_2 + \mu_1 \times \mu_2  - \frac{2}{3} \mu_1 \times \mu_2 + \mu_2 \times \mu_1 - \frac{2}{3} \mu_1 \times \mu_2 =  \mu_{12},\]
analogously for other projections. Thus  $\mu$ is uniting.
\end{proof}

Note that this construction does not allow to prove existence of a measure which is uniformly equivalent to something else. 

\section{Connection to the Monge--Kantorovich problem with linear constraints.}

\subsection{Monge--Kantorovich problem with linear constraints: definitions and basic facts}

D.~Zaev considered in  \cite{Zaev}  the multimarginal transportation problem with additional linear constraints. 
In this subsection we formulate basic definitions and theorems of his paper. 

Let $X_1, X_2, \dots, X_n$ be Polish spaces equipped with Borel $\sigma$-algebras, $X := X_1 \times \dots \times X_n$, $\mu_1, \dots, \mu_n$ are probability measures on  $X_1, \dots, X_n$ respectively.

Let $W$ be an arbitrary linear subspace in  $C_L(X, \mu_i)$. 
Let us consider the following subspace in the set of measures:
\[
\Pi_W(\mu_i) = \left\{\pi \in \Pi(\mu_i) \colon \int \omega~d\pi = 0 \text{ for all} \ \omega \in W \right\}.
\]
Finally, we are ready to formulate our constrained problem:

\begin{problem}[Monge--Kantorovich problem with linear constraints]
Given Polish spaces  $X = X_1 \times \dots X_n$, Borel probability measures  $\mu_i \in \mathcal{P}(X_i)$, a cost function $c \in C_L(X, \mu_i)$, and a linear subspace $W \subset C_L(X, \mu_i)$ find 
\[
\inf_{\pi \in \Pi_W(\mu)}\left\{\int_X c(x)~d\pi\right\}.
\]
\end{problem}

The following theorems  are main results of  \cite{Zaev}:

\begin{theorem}
Problem with additional linear constraints has a solution if the set $\Pi_W(\mu_i)$ is not empty.
\end{theorem}

\begin{theorem}[Kantorovich duality]\label{zaev_duality}
Let $X_1, \dots, X_n, X = X_1 \times \dots \times X_n$ be Polish spaces, let $\mu_k \in \mathcal{P}(X_k)$, $k = 1, \dots, n$ and let $W$ be a linear subspace of $C_L(X, \mu_k)$ (or $C_b(X)$), $c \in C_L(X, \mu_i)$ (or $C_b(X)$). 
Then 
\[
\inf_{\pi \in \pi_W(\mu)}\int_X c~d\pi = \sup_{f + \omega \le c}\sum_{k = 1}^n\int_{X_k}f_k(x_k)~d\mu_k,
\]
where $f(x_1, \dots, x_n) = \sum_{k = 1}^n f_k(x_k)$, $f_k \in C_L(X_k, \mu_k)$ (or $C_b(X_k)$), $\omega \in W$.
\end{theorem}

\subsection{
A problem with linear constraints which is equivalent to the multistochastic problem}

Let us consider again the multistochastic Monge--Kantorovich problem on Polish spaces  $X_1, \dots, X_n$.  We are given  $\binom{n}{k}$ probability measures   $\mu_\alpha$ on  $X_\alpha$, where  $\alpha \in \mathcal{I}_{nk}$, and a cost function  $c: X \to \mathbb{R}$,
 $X = X_1 \times \dots \times X_n$.
Our aim is to construct an equivalent Monge--Kantorovich problem with linear constraints.
Then we can apply duality \cref{zaev_duality}.

In what follows we denote
\[
\widetilde{X} = \prod_{\alpha \in \mathcal{I}_{nk}} X_{\alpha}. 
\]
For every $\alpha \in \mathcal{I}_{nk}$
we define the corresponding natural projection ${\rm Pr}_{\alpha} \colon \widetilde{X} \to X_{\alpha}$.

\begin{definition}
For all  $\alpha \in \mathcal{I}_{nk}$ and  $i \in \alpha$ let us consider projection
$\widetilde{x}_\alpha^i := \prj_{X_i} \circ \prj_{X_\alpha}$. In what follows  $\widetilde{x}_\alpha^i$ denotes the projection operator
and, at the same time, the image of  $\widetilde{x} \in \widetilde{X}$ under action of this operator. 
The set  $\{\widetilde{x}\}_\alpha^i$ can be viewed as a set of coordinates of  $\widetilde{x}$ in $\widetilde{X}$.
\end{definition}

\begin{definition}
The subspace $P \subset \widetilde{X}$ will be defined as follows:
\[P = \left\{\widetilde{x} \in \widetilde{X} \colon \widetilde{x}_\alpha^i = \widetilde{x}_\beta^i \text{ for all } \alpha, \beta \in \mathcal{I}_{nk}, i \in \alpha \cap \beta\right\}.
\]
\end{definition}

The subspace  $P$ can be characterized in terms of a diagonal operator. The space  $\widetilde{X}$ is isomorphic to  $(X_1 \times \dots \times X_n)^{\binom{n - 1}{k - 1}} = X^{\binom{n-1}{k-1}}$: 
to verify this it is sufficient to interchange factors in the product of spaces $X_\alpha = \prod_{i \in \alpha}X_i$. Let  $\Delta$ be the diagonal mapping from
 $X$ onto $\widetilde{X} = X^{\binom{n-1}{k-1}}$. 
 It is easy to see that this mapping is well--defined, because it does not depend on permutation of spaces in the isomorphism $\widetilde{X} \cong (X_1 \times \dots \times X_n)^{\binom{n - 1}{k - 1}}$.
 Hence  $P$ is the image of $X$ under action $\Delta$ and  restriction of $\Delta$ on  $P$ acts bijectively.

The following properties of $\Delta$ are direct consequences of its definition:
\begin{proposition}
Operator $\Delta$ generates an operator $\Delta_*: \mathcal{P}(X) \to \mathcal{P}(\widetilde{X})$ acting on measures, which has the following properties:
\begin{enumerate}
    \item For every measure $\mu \in \mathcal{P}(X)$ the support of  $\Delta_*(\mu)$ is a subset of  $P$.
    \item Operator $\Delta_*$ is a bijection between  $\mathcal{P}(X)$  and the set of measures  $\mu \in \mathcal{P}(\widetilde{X})$ with the property $\mathrm{supp}(\mu) \subset P$.
    \item Every $\mu \in \mathcal{P}(X)$ and every $\alpha \in \mathcal{I}_{nk}$ satisfy $\prj_\alpha(\mu) = \prj_{X_\alpha}(\Delta_*(\mu))$.
    \item Let $\mu$ be an arbitrary probability measure on $X$ and let $c \in L^1(X, \mu)$. Let $\widetilde{c}$ be a measurable function on $\widetilde{X}$ such that $\widetilde{c}(\widetilde{x}) = c(\Delta^{-1}(\widetilde{x}))$ 
    for all $\widetilde{x} \in P$. Then $\widetilde{c} \in L^1(\widetilde{X}, \Delta_*(\mu))$ and $\int_X c~d\mu = \int_{\widetilde{X}}\widetilde{c}~d\Delta_*(\mu) = \int_{P}\widetilde{c}~d\Delta_*(\mu)$.
\end{enumerate}
\end{proposition}

The following theorem is an immediate corollary of these properties
\begin{theorem} \label{weak_connection}
Let $c \in C_L(X, \mu_\alpha)$ be a function on $X$ and $\widetilde{c} \in C_L(\widetilde{X}, \mu_\alpha)$ be an extension of $c \circ \Delta^{-1}: P \subset \widetilde{X} \to \mathbb{R}$ 
onto the whole space $\widetilde{X}$. Then
\[
\inf_{\pi \in \Pi(X, \mu_\alpha)}\int_X c~d\pi = \inf_{\substack{\xi \in \Pi(\widetilde{X}, \mu_\alpha) \\ \mathrm{supp}(\xi) \subset P}}\int \widetilde{c}~d\xi.
\]
The minimum on the left-hand side is attained if and only if the minimum on the right-hand side is attained.
\end{theorem}

Consider the distance function  $d_i$ on  $X_i$ and the family of functions  $\omega^i_{\alpha\beta}: \widetilde{X} \to \mathbb{R},$ 
\[\omega^i_{\alpha\beta}(\widetilde{x}) := \min(d_i(\widetilde{x}^i_\alpha, \widetilde{x}^i_\beta), 1)
\] 
for all $\alpha, \beta \in \mathcal{I}_{nk}$, $i \in \alpha \cap \beta$. Note that  every $\omega^i_{\alpha\beta}$ is a nonnegative, continuous, bounded from above function, hence 
$\omega^i_{\alpha\beta} \in C_b(\widetilde{X}) \subset C_L(\widetilde{X}, \mu_\alpha)$. In addition, if some measure  $\mu \in \mathcal{P}(\widetilde{X})$ satisfies  
$\int \omega^i_{\alpha\beta}~d\mu = 0$, then $\mathrm{supp}(\mu) \subset (\omega^i_{\alpha\beta})^{-1}(0) = \{\widetilde{x} \in \widetilde{X}: \widetilde{x}^i_\alpha = \widetilde{x}^i_\beta\}$.

Let us define  the space of linear restrictions:
\[W := \mathrm{span}\{\omega^i_{\alpha\beta}\} \subset C_b(\widetilde{X}) \subset C_L(\widetilde{X}, \mu_\alpha).
\] 
It  follows from the observations collected above that  for every $\pi \in \mathcal{P}(\widetilde{X})$ the equality $\int\omega~d\pi = 0$ holds for all 
$\omega \in W$ if and only if  $\mathrm{supp}(\pi) \subset P$. Hence  
\[\Pi_W(\widetilde{X}, \mu_\alpha) = \{\pi \in \Pi(\widetilde{X}, \mu_\alpha): \mathrm{supp}(\pi) \subset P\}.\]

Having this in mind, we can give another formulation of \cref{weak_connection}:
\begin{theorem}
Let  $c \in C_L(X, \mu_\alpha)$ be a function on  $X$ and $\widetilde{c} \in C_L(\widetilde{X}, \mu_\alpha)$ be an  extension of  $c \circ \Delta^{-1}: P \subset \widetilde{X} \to \mathbb{R}$ onto the entire space $\widetilde{X}$. 
Then 
\[
\inf_{\pi \in \Pi(X, \mu_\alpha)}\int_X c~d\pi = \inf_{\xi \in \Pi_W(\widetilde{X}, \mu_\alpha)}\int \widetilde{c}~d\xi,
\]
and the minimum on the left-hand side is attained if and only if it is attained on the right-hand side.
\end{theorem}

This theorem gives another formulation of the transportation problem with linear constraints which is equivalent to our multistochastic problem.
It remains to prove that there exists a function  $\widetilde{c}$ which satisfies our requirement.

\begin{lemma}
    a) Let $c \in C_b(X)$. There exists a function  $\widetilde{c} \in C_b(\widetilde{X})$ which is an extension of  $c \circ \Delta^{-1}$ onto $\widetilde{X}$. 
    b) Let  $c \in C_L(X, \mu_\alpha)$. There exists a function $\widetilde{c} \in C_L(\widetilde{X}, \mu_\alpha)$  which is an extension of  $c \circ \Delta^{-1}$ onto $\widetilde{X}$.
\end{lemma}

\begin{proof}
Let $\mathrm{pr}$ be the projection of $\widetilde{X} \cong X^{\binom{n - 1}{k - 1}}$ onto a fixed factor. It is easy to see that  $\mathrm{pr}$ is continuous and $\mathrm{pr} \circ \Delta = \mathrm{id}$ on $X$.

a) Assume that $c \in C_b(X)$ and $|c| \le M$ for some number $M$. Set $\widetilde{c}(\widetilde{x}) := c(\mathrm{pr}(\widetilde{x}))$.
Function $\widetilde{c}$ is continuous, $|\widetilde{c}| \le M$ and $\widetilde{c}(\widetilde{x}) = c(\Delta^{-1}(\widetilde{x}))$ for all $\widetilde{x} \in P$. 
Thus, $\widetilde{c}$ is an extension of $c \circ \Delta^{-1}$ onto $\widetilde{X}$ and $\widetilde{c} \in C_b(\widetilde{X})$.

b) Assume that $c \in C_L(X, \mu_\alpha)$. Then  $|c(x)| \le \sum_{\alpha \in \mathcal{I}_{nk}}f_\alpha(x_\alpha)$. Set 
\begin{align*}
    \widetilde{c}(\widetilde{x}) := \begin{cases}
    -\sum_{\alpha \in \mathcal{I}_{nk}}f_\alpha(\widetilde{x}_\alpha), &\text{ if } c(\mathrm{pr}(\widetilde{x})) < -\sum_{\alpha \in \mathcal{I}_{nk}}f_\alpha(\widetilde{x}_\alpha),\\
    \sum_{\alpha \in \mathcal{I}_{nk}}f_\alpha(\widetilde{x}_\alpha), &\text{ if } c(\mathrm{pr}(\widetilde{x})) > \sum_{\alpha \in \mathcal{I}_{nk}}f_\alpha(\widetilde{x}_\alpha), \\
    c(\mathrm{pr}(\widetilde{x})), &\text{ otherwise}.
    \end{cases}
\end{align*}
The function $\widetilde{c}$ constructed in this way is continuous, $|\widetilde{c}(\widetilde{x})| \le \sum_{\alpha \in \mathcal{I}_{nk}}f_\alpha(\widetilde{x}_\alpha)$ and
 $\widetilde{c}(\widetilde{x}) = c(\Delta^{-1}(\widetilde{x}))$ for all $\widetilde{x} \in P$. Thus, $\widetilde{c}$ is an extension of  $c \circ \Delta^{-1}$ onto
  $\widetilde{X}$ and $\widetilde{c} \in C_L(\widetilde{X}, \mu_\alpha)$.
\end{proof}

\Cref{zaev_duality} implies the following duality relation:

\begin{proposition}
Under assumptions of the previous theorem
\[
\inf_{\pi \in \Pi(X, \mu_\alpha)}\int_X c~d\pi = \sup_{f + \omega \le \widetilde{c}}\sum_{\alpha \in \mathcal{I}_{nk}}\int_{X_\alpha}f_\alpha(x_\alpha)~d\mu_\alpha,
\]
where $f(\widetilde{x}) = \sum_{\alpha \in \mathcal{I}_{nk}} f_\alpha(\widetilde{x}_\alpha)$, $f_\alpha \in C_L(X_\alpha, \mu_\alpha)$ (or $C_b(X_\alpha)$), $\omega \in W$.
\end{proposition}

Assume that for the family of functions  $f_\alpha$ there exists  $\omega \in W$ such that $\sum_{\alpha \in \mathcal{I}_{nk}}f_\alpha(\widetilde{x}_\alpha) + \omega(\widetilde{x}) \le \widetilde{c}(\widetilde{x})$ 
for all $\widetilde{x} \in \widetilde{X}$. 
In particular, this equality holds for all  $\widetilde{x} \in P$. Then for all  $x \in X$  \[\sum_{\alpha \in \mathcal{I}_{nk}}f_\alpha(\Delta(x)_\alpha) + \omega(\Delta(x)) \le \widetilde{c}(\Delta(x)).\] 
Moreover, $\widetilde{c}(\Delta(x)) = c(x)$, $\omega(\Delta(x)) = 0$, $\Delta(x)_\alpha = x_\alpha$, hence $\sum_{\alpha \in \mathcal{I}_{nk}}f_\alpha(x_\alpha) \le c(x)$ for all $x \in X$. One gets
\[
\sup_{f + \omega \le \widetilde{c}}\sum_{\alpha \in \mathcal{I}_{nk}}\int_{X_\alpha}f_\alpha(x_\alpha)~d\mu_\alpha \le \sup_{f \le c}\sum_{\alpha \in \mathcal{I}_{nk}}\int_{X_\alpha}f_\alpha(x_\alpha)~d\mu_\alpha.
\]
In addition, the following inequality holds:
\[
\inf_{\pi \in \Pi(X, \mu_\alpha)}\int_X c~d\pi \ge \sup_{f \le c}\sum_{\alpha \in \mathcal{I}_{nk}}\int_{X_\alpha}f_\alpha(x_\alpha)~d\mu_\alpha.
\]

Summarizing these results we get the following final version of our duality theorem:

\begin{theorem}[Kantorovich duality for non-compact spaces] \label{thm:noncompact_duality}
Assume we are given Polish spaces  $X_1, \dots, X_n$ and a family of measures  $\mu_\alpha \in \mathcal{P}(X_\alpha)$, where $\alpha \in \mathcal{I}_{nk}$. 
Let  $c \in C_L(X, \mu_\alpha)$ (or $C_b(X)$) be a cost function on $X$. Then
\[
\inf_{\pi \in \Pi(\mu_\alpha)}\int_X c~d\pi = \sup_{f \le c}\sum_{\alpha \in \mathcal{I}_{nk}}\int_{X_\alpha}f_\alpha~d\mu_\alpha,
\]
where the supremum is taken on the set of all $f_\alpha \in C_L(X, \mu_\alpha)$ (or $C_b(X_\alpha)$), $f(x) = \sum_{\alpha \in \mathcal{I}_{nk}}f(x_\alpha)$. If the set $\Pi(\mu_\alpha)$ is non-empty, the infimum on the left-hand side is attained. 
\end{theorem}

\section{Sufficient conditions for existence of a dual solution}
\subsection{Definition and properties of \texorpdfstring{$(n, k)$}{(n, k)}-functions}
\begin{definition}
Assume we are given Polish spaces $X_1, \dots, X_n$ and a positive integer $1 \le k < n$. A function $F \colon X \to [-\infty, + \infty)$ is called an $(n, k)$-function if there exists a tuple of functions $\{f_\alpha\}_{\alpha \in \mathcal{I}_{nk}}$, $f_\alpha \colon X_\alpha \to [-\infty, + \infty)$ satisfying
\[
F(x) = \sum_{\alpha \in \mathcal{I}_{nk}}f_\alpha(x_\alpha)
\]
for all $x \in X$. If $F(x) > -\infty$ for each $x$ (and therefore $f_\alpha(x_\alpha) > -\infty$ for all $x_\alpha \in X_\alpha$), $F$ is called a finite $(n, k)$-function.
\end{definition}
This definition is given without any additional assumptions on the functions $f_\alpha$ and the function $F$. We prove that for every $(n, k)$-function $F$ there exists a "regular" tuple of functions $\{f_\alpha\}$ such that $F(x) = \sum_{\alpha \in \mathcal{I}_{nk}}f_\alpha(x_\alpha)$ for all $x \in X$.

Let us introduce more notations. For $x_\alpha \in X_\alpha$, $x_\beta \in X_\beta$, such that  $\alpha \cap \beta = \varnothing$, we denote by $x_\alpha x_\beta$  a point from the space $X_{\alpha \sqcup \beta}$, whose coordinates will be the union of the coordinates $x_\alpha$ and $x_\beta$. In addition, we write $\textbf{n} = \{1, 2, \dots, n\}$.

\begin{proposition} \label{prop:nk_function_decomposition}
Let $F$ be a finite $(n, k)$-function defined on the space $X$. Fix $y \in X$. For each $\alpha \in \mathcal{I}_n$ we define a function $F_\alpha \colon x_\alpha \mapsto F(x_\alpha y_{\boldsymbol{n} \backslash \alpha})$ on the space $X_\alpha$. 

Then there exists a tuple of real numbers $\{\lambda_i\}_{i = 0}^k$ depending only on $n$ and $k$ such that
$F(x) = \sum_{\alpha \in \mathcal{I}_{nk}}\widehat{f}_\alpha(x_\alpha)$ for each $x \in X$, where 
\[
\widehat{f}_\alpha(x_\alpha) = \sum_{\beta \subseteq \alpha}\lambda_{|\beta|}F_\beta(x_\beta),
\ 
\alpha \in \mathcal{I}_{nk}.
\]
\end{proposition}

This representation of $F$ is regular in the following sense: if $F$ is a measurable / continuous / bounded function, then for all $\alpha \in \mathcal{I}_{nk}$ the function $\widehat{f}_\alpha$ is measurable / continuous / bounded too.

\begin{example}\label{ex:n1-function_decompostion}
Let $F$ be a finite $(n, 1)$-function. Fix $y = (y_1, y_2, \dots, y_n) \in X$. Let $\lambda_0 = \frac{1}{n} - 1$ and $\lambda_1 = 1$. Then
\[
\widehat{f}_i(x_i) = F_i(x_i) - \frac{n - 1}{n}F_{\varnothing} = F_i(y_1, \dots, y_{i - 1}, x_i, y_{i + 1}, \dots, y_n) - \frac{n - 1}{n}F(y_1, \dots, y_n).
\]
Since $F$ is a finite $(n, 1)$-function, there exists a tuple of functions $f_i \colon X_i \to \mathbb{R}$ such that $F(x_1, \dots, x_n) = f_1(x_1) + \dots + f_n(x_n)$ for all $x \in X$. One can easily verify that
\[
\widehat{f}_i(x_i) = f_i(x_i) - f_i(y_i) + \frac{1}{n}(f_1(y_1) + \dots + f_n(y_n)),
\]
and therefore $F(x) = \sum_{i = 1}^n\widehat{f}_i(x_i)$ for all $x \in X$.
\end{example}

\begin{example}\label{ex:32_function_decomposition}
Let $F$ be a finite $(3, 2)$-function. Fix $(y_1, y_2, y_3) \in X$. Let $\lambda_0 = 1/3$, $\lambda_1 = -1/2$ and $\lambda_2 = 1$. Then by construction
\begin{align*}
    &\widehat{f}_{12}(x_1, x_2) = F(x_1, x_2, y_3) - \frac{1}{2}F(x_1, y_2, y_3) - \frac{1}{2}F(y_1, x_2, y_3) + \frac{1}{3}F(y_1, y_2, y_3),\\
    &\widehat{f}_{13}(x_1, x_3) = F(x_1, y_2, x_3) - \frac{1}{2}F(x_1, y_2, y_3) - \frac{1}{2}F(y_1, y_2, x_3) + \frac{1}{3}F(y_1, y_2, y_3),\\
    &\widehat{f}_{23}(x_2, x_3) = F(y_1, x_2, x_3) - \frac{1}{2}F(y_1, x_2, y_3) - \frac{1}{2}F(y_1, y_2, x_3) + \frac{1}{3}F(y_1, y_2, y_3).
\end{align*}
Similarly to \cref{ex:n1-function_decompostion} we can verify that 
\[
F(x_1, x_2, x_3) = \widehat{f}_{12}(x_1, x_2) + \widehat{f}_{13}(x_1, x_3) + \widehat{f}_{23}(x_2, x_3)
\]
for all $x \in X$.
\end{example}
\begin{proof}[Proof of \cref{prop:nk_function_decomposition}]
Consider a function $\widehat{F} \colon X \to \mathbb{R}$ defined as follows:
\[
\widehat{F}(x) = \sum_{\alpha \in \mathcal{I}_{nk}}\widehat{f}_{\alpha}(x_\alpha).
\]
Since by construction $\widehat{f}(x_\alpha) = \sum_{\beta \subseteq \alpha}\lambda_{|\beta|}F_\beta(x_\beta)$, one has
\[
\widehat{F}(x) = \sum_{\beta \in \mathcal{I}_n}\sum_{\alpha \in \mathcal{I}_{nk}\colon \beta \subseteq \alpha}\lambda_{|\beta|}F_\beta(x_\beta).
\]

For every $\beta \in \mathcal{I}_n$, let us find the amount $A_{\beta}$ of numbers $\alpha \in \mathcal{I}_{nk}$ satisfying $\beta \subseteq \alpha$. If $|\beta| > k$, then there is no such $\alpha$. Otherwise, it can be easily verified that $A_{\beta}  = \binom{n - |\beta|}{k - |\beta|}$. Hence,
\[
\widehat{F}(x) = \sum_{\beta \in \mathcal{I}_n \colon |\beta| \le k}\binom{n - |\beta|}{k - |\beta|}\lambda_{|\beta|}F_\beta(x_\beta) = \sum_{t = 0}^k\lambda_t\binom{n - t}{k - t}\sum_{\beta \in \mathcal{I}_{nt}}F_\beta(x_\beta).
\]

Since $F$ is a finite $(n, k)$-function, there exists a tuple of functions $\{f_\alpha\}_{\alpha \in \mathcal{I}_{nk}}$, $f_\alpha\colon X_\alpha \to \mathbb{R}$, such that for all $x \in X$ we have \[
\sum_{\alpha \in \mathcal{I}_{nk}}f_\alpha(x_\alpha) = F(x).\] 
For each $\beta \in \mathcal{I}_n$ the function $F_\beta(x_\beta)$ can be represented as follows:
\[
F_\beta(x_\beta) = \sum_{\gamma, \delta \in \mathcal{I}_n}f_{\gamma \sqcup \delta}(x_\gamma y_\delta),
\]
where the sum is taken for all pairs of disjoint sets of indices $\gamma, \delta \in \mathcal{I}_n$ satisfying $\gamma \subseteq \beta$, $\delta \subseteq \boldsymbol{n} \backslash \beta$ and $|\gamma| + |\delta| = k$. Hence, the function $\widehat{F}(x)$ can be represented as follows:
\begin{equation}\label{eq:c_gamma_delta_decomposition}
\widehat{F}(x) = \sum_{t = 0}^k\lambda_t\binom{n - t}{k - t}\sum_{\beta \in \mathcal{I}_{nt}}F_\beta(x_\beta) = \sum_{\gamma, \delta \in \mathcal{I}_n}c_{\gamma, \delta}f_{\gamma, \delta}(x_\gamma y_\delta),
\end{equation}
where the last sum is taken for all pairs of disjoint sets of indices $\gamma$ and $\delta$ such that $|\gamma| + |\delta| = k$, and $c_{\gamma, \delta}$ is a linear combination of $\{\lambda_i\}_{i = 0}^k$ with constant coefficients.

Let us find the coefficient $c_{\gamma, \delta}$. To this end,  let us find  for each $0 \le t \le k$
the amount  of indices $\beta \in \mathcal{I}_{nt}$ 
satisfying $\gamma \subseteq \beta$ and $\delta \subseteq \boldsymbol{n} \backslash \beta$. If $t < |\gamma|$, then this quantity is trivially zero. Similarly, it is zero if $t > n - |\delta| = n - k + |\gamma|$. Otherwise, exactly $|\gamma|$ indices of $\beta$ are fixed, and we need to choose $t - |\gamma|$  indices from $n - |\gamma| - |\delta| = n - k$ available items. Hence, the amount of such $\beta$ is $\binom{n - k}{t - |\gamma|}$. Substituting this into equation \cref{eq:c_gamma_delta_decomposition} we get
\[
c_{\gamma, \delta} = \sum_{t = |\gamma|}^{\min(k, n - k + |\gamma|)}\lambda_t\binom{n - t}{k - t}\binom{n - k}{t - |\gamma|}.\]
In particular, the coefficient $c_{\gamma, \delta}$ depends only on $|\gamma|$. 

In order for the equality $F(x) = \widehat{F}(x)$ to hold, it is sufficient to require that the coefficients $c_{\gamma, \delta}$ satisfy the following equalities:
\begin{align*}
    c_{\gamma, \delta} = 
    \begin{cases}
    1, &\text{if $|\gamma| = k$,}\\
    0, &\text{otherwise}.
    \end{cases}
\end{align*}
We obtain the system of linear equations on $\lambda$
\begin{align*}
    \begin{dcases}
    &\sum_{t = a}^{\min(k, n - k + a)}\lambda_t\binom{n - t}{k - t}\binom{n - k}{t - a} = 0 \;\text{ for $0 \le a < k$},\\
    &\sum_{t = k}^{\min(k, n - k + k)}\lambda_t\binom{n - t}{k - t}\binom{n - k}{t - k} = \lambda_k = 1.
    \end{dcases}
\end{align*}

The matrix of this linear system is upper-triangular and all diagonal elements are not equal to $0$. Hence, this system admits a unique solution $\{\widehat{\lambda}_i\}_{i = 0}^k$. Thus, if $\widehat{f}_\alpha(x_\alpha) = \sum_{\beta \subseteq \alpha}\widehat{\lambda}_{|\beta|}F_\beta(x_\beta)$, then $F(x) = \sum_{\alpha \in \mathcal{I}_{nk}}\widehat{f}_\alpha(x_\alpha)$ for all $x \in X$.
\end{proof}

For $1 \le i \le n$, we fix a probability measure $\mu_i$ on the space $X_i$. For each $\alpha \in \mathcal{I}_n$ we denote by $\mu_\alpha$ the probability measure $\prod_{i \in \alpha}\mu_i$ on the space $X_\alpha$, and we denote by $\mu$ the probability measure $\prod_{1 \le i \le n}\mu_i$ on the space $X$. If a finite $(n, k)$-function $F$ is integrable (with respect to $\mu$), we expect that there exists a tuple of integrable functions $\{f_\alpha\}_{\alpha \in \mathcal{I}_{nk}}$ (with respect to $\mu_\alpha$) such that $F(x) = \sum_{\alpha \in \mathcal{I}_{nk}}f_\alpha(x_\alpha)$. Using \cref{prop:nk_function_decomposition}, we construct a tuple of integrable functions $\{f_\alpha\}$ such that $\norm{f_\alpha}_1$ differs from $\norm{F}_1$ by no more than a constant factor depending on $n$ and $k$.

To achieve this let us verify the following lemma:
\begin{lemma}\label{lem:point_with_good_sections}
Let $X_i$, $1 \le i \le n$, be  Polish spaces equipped with the Borel $\sigma$-algebras, and for every $i$ let $\mu_i$ be a probability measure one $X_i$. Let $c \colon X \to \mathbb{R}$ be an integrable function  on  $X$. Fix a point $y \in X$, and for each $\alpha \in \mathcal{I}_n$ let us denote by $c_\alpha$ the function $x_\alpha \mapsto c(x_\alpha y_{\boldsymbol{n} \backslash \alpha})$ defined on $X_\alpha$.

Then there exists a point $y \in X$ such that $\norm{c_\alpha}_1 \le 2^{n + 1}\norm{c}_1$ for all $\alpha \in \mathcal{I}_n$. For $\alpha = \varnothing$ the function $c_\varnothing$ is a constant function on the one-point space $X_\varnothing$ which is equal to $c(y)$, and $\norm{c_\varnothing}_1$ is just the absolute value of $c(y)$.
\end{lemma}

\begin{proof}
For each $\alpha \in \mathcal{I}_n$ the spaces $X_\alpha \times X_{\boldsymbol{n} \backslash \alpha}$ and $X$ are canonically isomorphic, and therefore the function $c$ can be viewed as a function of two arguments $c(x_\alpha, y_{\boldsymbol{n} \backslash \alpha})$, where $x_\alpha \in X_\alpha$ and $y_{\boldsymbol{n} \backslash \alpha} \in X_{\boldsymbol{n} \backslash \alpha}$.

By the Fubini-Tonelli theorem, the function $|c(\cdot, y_{\boldsymbol{n} \backslash \alpha})|$ is integrable for $\mu_{\boldsymbol{n} \backslash \alpha}$-almost all $y_{\boldsymbol{n} \backslash \alpha}$ and
\begin{equation}\label{eq:Fubini-Tonelli_c}
\norm{c}_1 = \int_{X_{\boldsymbol{n} \backslash \alpha}}\left(\int_{X_\alpha}|c(x_\alpha, y_{\boldsymbol{n}\backslash \alpha})|\,\mu_{\alpha}(dx_\alpha)\right)\,\mu_{\boldsymbol{n} \backslash \alpha}(dy_{\boldsymbol{n} \backslash \alpha}).
\end{equation}

Consider the internal function from this expression: \[C_{\boldsymbol{n} \backslash \alpha}(y_{\boldsymbol{n} \backslash \alpha}) = \int_{X_\alpha}|c(x_\alpha, y_{\boldsymbol{n} \backslash \alpha})|\,\mu_\alpha(dx_\alpha).\]
This function is non-negative. In addition, it follows from \cref{eq:Fubini-Tonelli_c}, $C_{\boldsymbol{n} \backslash \alpha} \in L^1(X_{\boldsymbol{n} \backslash \alpha}, \mu_{\boldsymbol{n} \backslash \alpha})$ and $\norm{C_{\boldsymbol{n} \backslash \alpha}}_1 = \norm{c}_1$. Let
\[
A_{\boldsymbol{n} \backslash \alpha} = \left\{y_{\boldsymbol{n} \backslash \alpha} \in X_{\boldsymbol{n} \backslash \alpha} \colon C_{\boldsymbol{n} \backslash \alpha}(y_{\boldsymbol{n} \backslash \alpha}) > 2^{n + 1}\norm{c}_1\right\}.
\]
If $\norm{c}_1 = 0$, then $C_{\boldsymbol{n} \backslash \alpha}(y_{\boldsymbol{n} \backslash \alpha})$ is equal to 0 for $\mu_{\boldsymbol{n} \backslash \alpha}$-almost all points $y_{\boldsymbol{n} \backslash \alpha}$, and therefore $\mu_{\boldsymbol{n} \backslash \alpha}(A_{\boldsymbol{n} \backslash \alpha}) = 0$. Otherwise, it follows from Markov's inequality that
\[
\mu_{\boldsymbol{n} \backslash \alpha}(A_{\boldsymbol{n} \backslash \alpha})  \le \frac{1}{2^{n + 1}\norm{c}_1}\int_{X_{\boldsymbol{n} \backslash \alpha}}C_{\boldsymbol{n} \backslash \alpha}(y_{\boldsymbol{n} \backslash \alpha})\,\mu_{\boldsymbol{n} \backslash \alpha}(d_{\boldsymbol{n} \backslash \alpha}) = \frac{\norm{C_{\boldsymbol{n} \backslash \alpha}}_1}{2^{n + 1}\norm{c}_1} = \frac{1}{2^{n + 1}}.
\]
In both cases we conclude that $\mu_{\boldsymbol{n} \backslash \alpha}(A_{\boldsymbol{n} \backslash \alpha}) \le 2^{-n - 1}$. 

If $y \in \prj_{\boldsymbol{n} \backslash \alpha}^{-1}(X_{\boldsymbol{n} \backslash \alpha} \backslash A_{\boldsymbol{n} \backslash \alpha})$, then \[C_{\boldsymbol{n} \backslash \alpha}(y_{\boldsymbol{n} \backslash \alpha}) \le 2^{n + 1}\norm{c},\] and therefore the function $c_\alpha \colon x_\alpha \mapsto c(x_\alpha y_{\boldsymbol{n} \backslash \alpha})$ is integrable with respect to $\mu_\alpha$ and
\[
\norm{c_\alpha}_1 = C_{\boldsymbol{n} \backslash \alpha}(y_{\boldsymbol{n} \backslash \alpha}) \le 2^{n + 1}\norm{c}_1.
\]
Thus, if \[
y \in A = \bigcap_{\alpha \in \mathcal{I}_n} \prj_{\boldsymbol{n} \backslash \alpha}^{-1}(X_{\boldsymbol{n} \backslash \alpha} \backslash A_{\boldsymbol{n} \backslash \alpha}),
\]
then for all $\alpha \in \mathcal{I}_n$ the function $c_\alpha \colon x_\alpha \mapsto c(x_\alpha y_{\boldsymbol{n} \backslash \alpha})$ is integrable and $\norm{c_\alpha}_1 \le 2^{n + 1}\norm{c}_1$.

We only need to verify that $A$ is non-empty. We have
\[
\mu\left(\prj_{\boldsymbol{n} \backslash \alpha}^{-1}(X_{\boldsymbol{n} \backslash \alpha} \backslash A_{\boldsymbol{n} \backslash \alpha})\right) = \mu_{\boldsymbol{n} \backslash \alpha}\left(X_{\boldsymbol{n} \backslash \alpha} \backslash A_{\boldsymbol{n} \backslash \alpha}\right) = 1 - \mu_{\boldsymbol{n} \backslash \alpha}(A_{\boldsymbol{n} \backslash \alpha}) \ge 1 - \frac{1}{2^{n + 1}},
\]
and therefore
\[
\mu(A) \ge 1 - \frac{|\mathcal{I}_n|}{2^{n + 1}} \ge 1 - \frac{2^n}{2^{n + 1}} = \frac{1}{2}.
\]
Thus,  $A$ is a set of positive measure, and therefore $A \ne \varnothing$.
\end{proof}
\begin{theorem}\label{thm:existence_integrable_f_alpha_for_product}
For every $1 \le i \le n$, let $X_i$ be a Polish space equipped with the Borel $\sigma$-algebra, and let $\mu_i$ be a probability measure on  $X_i$. There exists a constant $C$ depending only on $n$ and $k$ such that for any finite $(n, k)$-function $F \in L^1(X, \mu)$ there exists a tuple of integrable functions $\{\widehat{f}_\alpha\}_{\alpha \in \mathcal{I}_{nk}}$, $\widehat{f}_\alpha \in L^1(X_\alpha, \mu_\alpha)$, such that
\[
F(x) = \sum_{\alpha \in \mathcal{I}_{nk}}\widehat{f}_\alpha(x_\alpha)
\]
for all $x \in X$ and $\lVert\widehat{f}_\alpha\rVert_1 \le C \cdot \norm{F}_1$ for all $\alpha \in \mathcal{I}_{nk}$.
\end{theorem}
\begin{proof}
Consider a finite $(n, k)$-function $F$ defined on the space $X$. By \cref{lem:point_with_good_sections} there exists a point $y \in X$ such that the function $F_\alpha \colon x_\alpha \mapsto F(x_\alpha y_{\boldsymbol{n} \backslash \alpha})$ is integrable and $\norm{F_\alpha}_1 \le 2^{n + 1}\norm{F}_1$ for all $\alpha \in \mathcal{I}_n$.

By \cref{prop:nk_function_decomposition} there exists a tuple of real numbers $\{\lambda_i\}_{i = 0}^k$ such that  $F(x) = \sum_{\alpha \in \mathcal{I}_{nk}}\widehat{f}_\alpha(x_\alpha)$ for all $x \in X$, where
\[
\widehat{f}_\alpha(x_\alpha) = \sum_{\beta \subseteq \alpha}\lambda_{|\beta|}F_\beta(x_\beta),
\ 
\alpha \in \mathcal{I}_{nk}.
\]
 Since $F_\beta \in L^1(X_\beta, \mu_\beta)$ for all $\beta \in \mathcal{I}_n$, we conclude that $\widehat{f}_\alpha \in L^1(X_\alpha, \mu_\alpha)$. In addition,
\[
\norm{\widehat{f}_\alpha}_1 \le \sum_{\beta \subseteq \alpha}\left|\lambda_{|\beta|}\right| \cdot \norm{F_\beta}_1 \le 2^{n + 1}\norm{F}_1\sum_{\beta \subseteq \alpha}\left|\lambda_{|\beta|}\right| = 2^{n + 1}\norm{F}_1 \sum_{t = 0}^k\binom{k}{t}\left|\lambda_t\right|.
\]
Thus, we conclude that $\norm{\widehat{f}_\alpha}_1 \le C \cdot \norm{F}_1$, where
\[
C = 2^{n + 1}\sum_{t = 0}^k\binom{k}{t}|\lambda_t|,
\]
and this constant depends only on $n$ and $k$.
\end{proof}
\begin{example}
Let us find a constant $C$ explicitly for the case of the $(3, 2)$-problem. Consider a finite integrable $(3, 2)$-function $F$. There exists a point $y \in X = X_1 \times X_2 \times X_3$ such that $\norm{F_\alpha}_1 \le 16\norm{F}_1$ for all $\alpha \in \mathcal{I}_3$. By \cref{ex:32_function_decomposition} the functions
\begin{align*}
    &\widehat{f}_{12}(x_1, x_2) = F(x_1, x_2, y_3) - \frac{1}{2}F(x_1, y_2, y_3) - \frac{1}{2}F(y_1, x_2, y_3) + \frac{1}{3}F(y_1, y_2, y_3),\\
    &\widehat{f}_{13}(x_1, x_3) = F(x_1, y_2, x_3) - \frac{1}{2}F(x_1, y_2, y_3) - \frac{1}{2}F(y_1, y_2, x_3) + \frac{1}{3}F(y_1, y_2, y_3),\\
    &\widehat{f}_{23}(x_2, x_3) = F(y_1, x_2, x_3) - \frac{1}{2}F(y_1, x_2, y_3) - \frac{1}{2}F(y_1, y_2, x_3) + \frac{1}{3}F(y_1, y_2, y_3).
\end{align*}
satisfy the equation $F(x_1, x_2, x_3) = \widehat{f}_{12}(x_1, x_2) + \widehat{f}_{13}(x_1, x_3) + \widehat{f}_{23}(x_2, x_3)$ for all $(x_1, x_2, x_3) \in X$. All functions $\{\widehat{f}_{ij}\}$ are integrable with respect to $\mu_i \otimes \mu_j$. In addition,
\begin{align*}
\norm{\widetilde{f}_{12}}_1 &\le \norm{F(\cdot, \cdot, y_3)}_1 + \frac{1}{2}\norm{F(\cdot, y_2, y_3)}_1 + \frac{1}{2}\norm{F(y_1, \cdot, y_3)}_1 + \frac{1}{3}|F(y_1, y_2, y_3)|\\
& \le 16\left(1 + \frac{1}{2} + \frac{1}{2} + \frac{1}{3}\right)\norm{F}_1 < 38 \norm{F}_1.
\end{align*}
Similarly, $\norm{\widehat{f}_{13}}_1 < 38 \norm{F}_1$ and $\norm{\widehat{f}_{23}}_1 < 38 \norm{F}_1$, and therefore we can put $C = 38$. This constant estimate is crude, but we do not need to know the optimal value.
\end{example}

We want to generalize this property to a wider class of measures that are uniformly equivalent to the product of their projections to one-dimensional spaces.

\begin{definition}
We call the probability measure $\mu$ on the space $X$ \textit{reducible} if for $1 \le i \le n$ there exists a probability measure $\nu_i$ on spaces $X_i$ such that $\mu$ is uniformly equivalent to $\prod_{1 \le i \le n}\nu_i$.

We call the consistent set of probability measures $\{\mu_\alpha\}_{\alpha \in \mathcal{I}_{nk}}$ \textit{reducible} if there exists a uniting reducible measure $\mu \in \Pi(\mu_\alpha)$.
\end{definition}

If the probability measures $\mu$ and $\nu$ on the space $X$ are uniformly equivalent, then their projections are also uniformly equivalent: $\prj_\alpha \mu$ is uniformly equivalent to $\prj_\alpha \nu$ for all $\alpha \in \mathcal{I}_n$. In particular, if the set of measures $\mu_\alpha$ is reducible, then $\mu_i = \prj_i(\mu)$ is uniformly equivalent to $\nu_i$. Then the measure $\prod_{1 \le i \le n}\mu_i$ is uniformly equivalent to the measure $\prod_{1 \le n}\nu_i$. Hence, the following is true:

\begin{proposition}
A tuple of probability measures $\{\mu_\alpha\}_{\alpha \in \mathcal{I}_{nk}}$ is reducible if and only if there exists a uniting measure $\mu \in \Pi(\mu_\alpha)$, which is uniformly equivalent to $\prod_{1 \le i \le n}\mu_i$.
\end{proposition}

If the set of measures $\mu_\alpha$ is reducible, then for all $\beta \in \mathcal{I}_{nt}$, $t \le k$, the measure $\mu_\beta$ is uniformly equivalent to $\prod_{i \in \beta}\mu_i$. It is easy to see that this condition is not sufficient.

\begin{example}\label{ex:discrete_nonuniform}
Let $X_1$, $X_2$ and $X_3$ be discrete spaces, each consisting of two elements $\{0, 1\}$. Define a probability measure $\mu_{ij}$ on the space $X_i \times X_j$ as follows:
\begin{align*}
    \mu_{ij}(x_i, x_j) = \begin{cases}
    \frac{1}{3}, &\text{if $x_i \ne x_j$},\\
    \frac{1}{6}, &\text{otherwise}.
    \end{cases}
\end{align*}

The tuple of measures $\{\mu_{ij}\}$ is consistent. In addition, every measure $\mu_{ij}$, $\{i, j\} \in \mathcal{I}_{3, 2}$, is uniformly equivalent to $\mu_i \otimes \mu_j$. The set $\Pi(\mu_{ij})$ is non-empty: consider the following measure $\mu$ on the space $X_1 \times X_2 \times X_3$: $\mu(x_1, x_2, x_3) = 0$ if $x_1 = x_2 = x_3$, otherwise $\mu(x_1, x_2, x_3) = 1/6$. It is easy to check that $\mu \in \Pi(\mu_{ij})$.

Let $\nu \in \Pi(\mu_{ij})$. Then the following equations hold:
\begin{align*}
    &\nu(0, 0, 0) + \nu(0, 0, 1) = \mu_{12}(0, 0) = \frac{1}{6}, \\
    &\nu(0, 0, 1) + \nu(0, 1, 1) = \mu_{13}(0, 1) = \frac{1}{3}, \\
    &\nu(0, 1, 1) + \nu(1, 1, 1) = \mu_{23}(1, 1) = \frac{1}{6}.
\end{align*}
From these equations we get $\nu(0, 0, 0) + \nu(1, 1, 1) = 0$. From the non-negativity of the measure we get $\nu(0, 0, 0) = \nu(1, 1, 1) = 0$, and then we easily verify that $\nu(x_1, x_2, x_3) = 1/6$ for the remaining points. Thus $\Pi(\mu_{i j})$ consists of a single measure that is not uniformly equivalent to $\mu_1 \otimes \mu_2 \otimes \mu_3$.
\end{example}

The following theorem generalizes \cref{thm:existence_integrable_f_alpha_for_product} to reducible tuples of measures.
\begin{theorem}\label{thm:existence_integrable_f_alpha_for_reducible}
For $1 \le i \le n$, let $X_i$ be a Polish space equipped with the Borel $\sigma$-algebra, and let $\mu$ be a reducible probability measure on $X$. Denote $\mu_\alpha = \prj_\alpha(\mu)$. Then there exists a constant $C_\mu$ such that for any finite $(n, k)$-function $F \in L^1(X, \mu)$ there exists a tuple of integrable functions $\{\widehat{f}_\alpha\}_{\alpha \in \mathcal{I}_{nk}}$, $\widehat{f}_\alpha \in L^1(X_\alpha, \mu_\alpha)$, such that
\[
F(x) = \sum_{\alpha \in \mathcal{I}_{nk}}\widehat{f}_\alpha(x_\alpha)
\]
for all $x \in X$ and \[
\norm{\widehat{f}_\alpha}_{L^1(\mu_\alpha)} \le C_\mu \cdot \norm{F}_{L^1(\mu)}\] 
for all $\alpha \in \mathcal{I}_{nk}$.
\end{theorem}
\begin{proof}
Since $\mu$ is reducible, there exist probability measures $\nu_i \in \mathcal{P}(X_i)$ and positive reals $m$ and $M$ such that $m \cdot \nu \le \mu \le M \cdot \nu$, where $\nu = \prod_{1 \le i \le n}\nu_i$. 

Consider a finite $(n, k)$-function $F \in L^1(X, \mu)$. Since $\mu \ge m \cdot \nu$, the function $F$ is integrable with respect to $\nu$ and
\[
\norm{F}_{L^1(\nu)} \le \frac{1}{m}\norm{F}_{L^1(\mu)}.
\]

Denote $\nu_\alpha = \prod_{i \in \alpha}\nu_i$. It follows from \cref{thm:existence_integrable_f_alpha_for_product} that there exists a tuple of integrable functions $\{\widehat{f}_\alpha\}_{\alpha \in \mathcal{I}_{nk}}$, $\widehat{f}_\alpha \in L^1(X_\alpha, \nu_\alpha)$ such that
\[
F(x) = \sum_{\alpha \in \mathcal{I}_{nk}}\widehat{f}_\alpha(x_\alpha)
\]
for all $x \in X$ and \[
\norm{\widehat{f}_\alpha}_{L^1(\nu_\alpha)} \le C \cdot \norm{F}_{L^1(\nu)} \le \frac{C}{m}\norm{F}_{L^1(\mu)}\] 
for all $\alpha \in \mathcal{I}_{nk}$, where $C$ is a constant depending only on $n$ and $k$

Since $M \cdot \nu \ge \mu$, we have $M \cdot \nu_\alpha \ge \mu_\alpha$ for all $\alpha \in \mathcal{I}_{nk}$. Hence, the function $\widehat{f}_\alpha$ is integrable with respect to $\mu_\alpha$ and
\[
\norm{\widehat{f}_\alpha}_{L^1(\mu_\alpha)} \le M\norm{\widehat{f}_\alpha}_{L^1(\nu_{\alpha})} \le \frac{M}{m}C\norm{F}_{L^1(\mu)}
\]
for all $\alpha \in \mathcal{I}_{nk}$. Thus, we can put $C_\mu = \frac{M}{m}C$.
\end{proof}
\subsection{Existence of a dual solution for reducible tuples of measures}

First, we generalize the notion of the proper thickness of the set introduced in \cite{VerZatPet14}. 

\begin{definition}
Let $X_1, \dots, X_n$ be Polish spaces, and for each $\alpha \in \mathcal{I}_{nk}$ let $\mu_\alpha$ be a probability measure on the space $X_\alpha$. For a measurable set $A \subset X$ define its proper $(n, k)$-thickness as
\begin{equation}\label{eq:thickness_definition}
\mathrm{sth}(A) = \inf\left\{\sum_{\alpha \in \mathcal{I}_{nk}}\mu_\alpha(Y_\alpha) \colon Y_\alpha \subseteq X_\alpha, A \subseteq \bigcup_{\alpha \in \mathcal{I}_{nk}}\prj_{\alpha}^{-1}(Y_\alpha)\right\}.
\end{equation}
\end{definition}
We are going to use this notion in the particular case of sets with zero proper thickness.
\begin{proposition}
If $\mathrm{sth}(A) = 0$, then the infimum in \cref{eq:thickness_definition} is attained: there exist measurable subsets $Y_\alpha \subseteq X_\alpha$, $\alpha \in \mathcal{I}_{nk}$, such that $\mu_\alpha(Y_\alpha) = 0$ and $A \subseteq \bigcup_{\alpha \in \mathcal{I}_{nk}}\prj_\alpha^{-1}(Y_\alpha)$.
\end{proposition}
\begin{proof}
The proof follows the proof of \cite[Lemma 2.5.4]{VerZatPet14}. If for a tuple of measurable subsets $\{Y_\alpha\}_{\alpha \in \mathcal{I}_{nk}}$ we have $A \subseteq \bigcup_{\alpha \in \mathcal{I}_{nk}}\prj_{\alpha}^{-1}(Y_\alpha)$, then $f_\alpha = \mathbbm{1}[Y_\alpha]$ satisfy the inequality
\[
\sum_{\alpha \in \mathcal{I}_{nk}}f_\alpha(x_\alpha) \ge \mathbbm{1}[A](x)
\]
for all $x \in X$, where $\mathbbm{1}[A]$ is the characteristic function of the set $A$. Moreover, it is clear that \[
\sum_{\alpha \in \mathcal{I}_{nk}}\int_{X_\alpha}f_\alpha(x_\alpha)\,\mu_\alpha(dx_\alpha) = \sum_{\alpha \in \mathcal{I}_{nk}}\mu_\alpha(Y_\alpha).
\]

Since $\mathrm{sth}(A) = 0$, we can consider a minimizing sequence of tuples of functions $\{f_\alpha^{(t)}\}_\alpha$, $f^{(t)}_\alpha \colon X_\alpha \to [0, 1]$, such that
\[
\sum_{\alpha \in \mathcal{I}_{nk}}f^{(t)}_\alpha(x_\alpha) \ge \mathbbm{1}[A](x)
\]
for all $x \in X$ and
\[
\sum_{\alpha \in \mathcal{I}_{nk}}\int_{X_\alpha}f_\alpha^{(t)}(x_\alpha)\,\mu_\alpha(dx_\alpha) \xrightarrow[t \to \infty]{} 0.
\]
Since $f_\alpha^{(t)}$ is non-negative for all $\alpha \in \mathcal{I}_{nk}$ and for all $t$, we conclude that
\[
\int_{X_\alpha}f_\alpha^{(t)}(x_\alpha)\,\mu_\alpha(dx_\alpha) \xrightarrow[t \to \infty]{} 0 \;\text{for all $\alpha \in \mathcal{I}_{nk}$}.
\]
Let us recall the formulation of the Koml{\'o}s theorem.
\begin{theorem}[{{\cite[Theorem 4.7.24]{Bogachev2007}}}]
Let $\mu$ be a finite nonnegative measure on a space $X$, let $\{f_n\} \subset L^1(\mu)$, and let
\[
\sup_n\norm{f_n}_{L^1(\mu)} < \infty.
\]
Then, one can find a subsequence $\{g_n\} \subseteq \{f_n\}$ and a function $g \in L^1(\mu)$ such
that, for every sequence $\{h_n\} \subseteq \{g_n\}$, the arithmetic means $(h_1 + \dots + h_n)/n$ converge almost everywhere to $g$.
\end{theorem}

Using this theorem and passing, if necessary, to subsequences, we may assume that the sequence
\[
g^{(t)}_\alpha = \frac{1}{t}\left(f^{(1)}_\alpha + \dots + f^{(t)}_\alpha\right)
\]
converges to some integrable function $g_\alpha$ $\mu_\alpha$-almost everythere in $X_\alpha$ for all $\alpha \in \mathcal{I}_{nk}$. Thus, we can suppose that
\[
g_\alpha(x_\alpha) = \limsup_{t \to \infty}g_\alpha^{(t)}(x_\alpha)\; \text{ for all $x_\alpha \in X_\alpha$.}
\]

By construction we obtain  $0 \le g_\alpha(x_\alpha) \le 1$ for all $x_\alpha \in X_\alpha$. Also, since $\sum_{\alpha \in \mathcal{I}_{nk}}g_\alpha^{(t)}(x_\alpha) \ge \mathbbm{1}[A](x)$ for all $x \in X$ and for all $t$, we conclude that 
\begin{equation}\label{eq:g_alpha_ge_char_func}
\sum_{\alpha \in \mathcal{I}_{nk}}g_\alpha(x_\alpha) \ge \mathbbm{1}[A](x)\;\text{ for all $x \in X$.}
\end{equation}
In addition, since $|g_\alpha^{(t)}(x_\alpha)| \le 1$ it follows from the Lebesgue's dominated convergence theorem that
\[
\int_{X_\alpha}g_\alpha(x_\alpha)\,\mu_\alpha(dx_\alpha) = \lim_{t \to \infty}\int_{X_\alpha}g_\alpha^{(t)}(x_\alpha)\,\mu_\alpha(dx_\alpha) = \lim_{t \to \infty}\int_{X_\alpha}f_\alpha^{(t)}(x_\alpha)\,\mu_\alpha(dx_\alpha) = 0.
\]
Thus, since the function $g_\alpha$ is non-negative, we conclude that $g_\alpha(x_\alpha) = 0$ for $\mu_\alpha$-almost all $x_\alpha \in X_\alpha$.

Consider the tuple of sets $\{Y_\alpha\}_{\alpha \in \mathcal{I}_{nk}}$:
\[
Y_\alpha = \left\{x_\alpha \in X_\alpha \colon g_\alpha(x_\alpha) > 0\right\}.
\]
Since $g_\alpha$ is equal to $0$ almost everywhere on $X_\alpha$, we have $\mu_\alpha(Y_\alpha) = 0$. In addition, if $x \in A$, then it follows from inequality \cref{eq:g_alpha_ge_char_func} that $\sum_{\alpha \in \mathcal{I}_{nk}}g_\alpha(x_\alpha) \ge 1$, and therefore there exists at least one $\alpha \in \mathcal{I}_{nk}$ such that
$g_\alpha(x_\alpha) > 0$ or equivalently $x_\alpha \in Y_\alpha$.
Thus, $A \subseteq \bigcup_{\alpha \in \mathcal{I}_{nk}}\prj_\alpha^{-1}(Y_\alpha)$.
\end{proof}

\begin{definition}
We say that a measurable set $A \subset X$ is a zero $(n, k)$-thickness set if $\mathrm{sth}(A) = 0$, or equivalently if there exist a tuple of measurable subsets $Y_\alpha \subset X_\alpha$, $\alpha \in \mathcal{I}_{nk}$ such that $\mu_\alpha(Y_\alpha) = 0$ for all $\alpha$ and $A \subseteq \cup_{\alpha \in \mathcal{I}_{nk}}\prj_{\alpha}^{-1}(Y_\alpha)$.
\end{definition}

In addition to the standard dual multistochastic problem, we consider a more convenient relaxed dual problem. Let $c$ be a measurable cost function on the space $X$. Denote by \[\Psi_c(\mu_\alpha)\] the set of tuples of integrable functions $\{f_\alpha\}_{\alpha \in \mathcal{I}_{nk}}$, $f_\alpha \colon X_\alpha \to \mathbb{R}$ such that  inequality 
\[
\sum_{\alpha \in \mathcal{I}_{nk}}f_\alpha(x_\alpha) \le c(x_\alpha)
\]
holds at all points $x \in X$ except a zero $(n, k)$-thickness set. Then, in the relaxed dual problem we are looking for
\[
J = \sup\left\{\sum_{\alpha \in \mathcal{I}_{nk}}\int_{X_\alpha}f_\alpha(x_\alpha)\,\mu_\alpha(dx_\alpha) \colon \{f_\alpha\} \in \Psi_c(\mu_\alpha)\right\}
\]

If $\{f_\alpha\} \in \Psi_c(\mu_\alpha)$, then there exists a tuple of measurable subsets $Y_\alpha \subset X_\alpha$ such that $\mu_\alpha(Y_\alpha) = 0$ and
\[
\sum_{\alpha \in \mathcal{I}_{nk}}f_\alpha(x_\alpha) \le c(x) \text{ for all }x \not\in \bigcup_{\alpha \in \mathcal{I}_{nk}}\prj_{\alpha}^{-1}(Y_\alpha).
\]
Consider the tuple of functions $\{\widehat{f}_\alpha\}$ defined as follows: $\widehat{f}_\alpha(x_\alpha) = f_\alpha(x_\alpha)$ if $x_\alpha \not\in Y_\alpha$ and $\widehat{f}_\alpha(x_\alpha) = -\infty$ otherwise. For all $\alpha \in \mathcal{I}_{nk}$ the function $\widehat{f}_\alpha$ coincides with $f_\alpha$ almost everywhere with respect to $\mu_\alpha$, and therefore 
\[
\sum_{\alpha \in \mathcal{I}_{nk}}\int \widehat{f}_\alpha(x_\alpha)\,\mu_\alpha(dx_\alpha) = \sum_{\alpha \in \mathcal{I}_{nk}}\int f_\alpha(x_\alpha)\,\mu_\alpha(dx_\alpha)
\]
In addition, the inequality $\sum_{\alpha \in \mathcal{I}_{nk}}\widehat{f}_\alpha(x_\alpha) \le c(x)$ holds for all $x \in X$. Thus, having a  tuple of functions $\{f_\alpha\} \in \Psi_c(\mu_\alpha)$ one can construct another tuple of (not necessary real-valued) functions $\{\widehat{f}_\alpha\}$ satisfying the conditions of the standard dual problem with the same value of the dual functional. Therefore the supremum is the same for both standard and relaxed dual problems. 

In \cite{Kell} the following theorem was proved, establishing the existence of a dual solution in the multi-marginal case.
\begin{theorem}[Kellerer]\label{existence_multimarginal_dual_solution}
For every $1 \le i \le n$, let $X_i$ be a Polish space equipped with a Borel probability measure $\mu_i$. Let $c \colon X_1 \times \dots \times X_n \to [-\infty, +\infty]$ be a measurable cost function on the space $X_1 \times \dots \times X_n$. Suppose that there exists a tuple of integrable functions $\{c_i\}_{i = 1}^n$, $c_i: X_i \to (-\infty, +\infty]$ such that  inequality 
\[
|c(x_1, \dots, x_n)| \le \sum_{i = 1}^n c_i(x_i)
\]
holds for all $(x_1, \dots, x_n) \in X$.

Then the supremum in the relaxed dual Monge-Kantorovich problem
\[
\sup\left\{\sum_{i = 1}^n\int_{X_i}\varphi_i(x_i)\,\mu_i(dx_i) \colon \{\varphi_i\}_{i = 1}^n \in \Psi_c(\mu_i)\right\}
\]
is finite and attained.
\end{theorem}
We prove the multistochastic generalization of this theorem for the case of reducible tuple of projections.

\begin{theorem}\label{thm:existence_multistochastic_dual_solution}
For every $1 \le i \le n$, let $X_i$ be a Polish space, let $\{\mu_\alpha\}_{\alpha \in \mathcal{I}_{nk}}$, $\mu_\alpha \in \mathcal{P}(X_\alpha)$ be a reducible tuple of probability measures, and let $c \colon X \to [-\infty, +\infty]$ be a measurable cost function on the space $X$. Suppose that there exists a tuple of integrable functions $\{c_\alpha\}_{\alpha \in \mathcal{I}_{nk}}$, $c_\alpha \colon X_\alpha \to (-\infty, +\infty]$ such that the inequality
\[
|c(x)| \le \sum_{\alpha \in \mathcal{I}_{nk}}c_\alpha(x_\alpha)
\]
holds for all $x \in X$.

Then the supremum in the relaxed dual multistochastic Monge-Kantorovich problem
\begin{equation}\label{eq:relaxed_dual_sup}
J = \sup\left\{\sum_{\alpha \in \mathcal{I}_{nk}}\int_{X_\alpha}f_\alpha(x_\alpha)\,\mu_\alpha(dx_\alpha) \colon \{f_\alpha\}_{\alpha \in \mathcal{I}_{nk}} \in \Psi_c(\mu_\alpha)\right\}
\end{equation}
is finite and attained. 
\end{theorem}

\begin{proof}
Replacing $c_\alpha$ with $|c_\alpha|$ we may assume that the function $c_\alpha$ is non-negative for all $\alpha \in \mathcal{I}_{nk}$. Let $c^*_\alpha \colon X_\alpha \to [0, +\infty)$ be an arbitrary finite integrable function such that $c^*_\alpha(x_\alpha) = c_\alpha(x_\alpha)$ for $\mu_\alpha$-almost all $x_\alpha \in X_\alpha$. Consider a function $c^*$ on the space $X$:
\begin{align*}
    c^*(x) = \begin{cases}
    c(x), &\text{if $c^*_\alpha(x_\alpha) = c_\alpha(x_\alpha)$ for all $x_\alpha \in X_\alpha$},\\
    0, &\text{otherwise}.
    \end{cases}
\end{align*}

It trivially follows from the construction that $c^*(x) = c(x)$ for all $x \in X$ except a zero $(n, k)$-thickness set. Hence, \[
\Psi_{c}(\mu_\alpha) = \Psi_{c^*}(\mu_\alpha).
\]
In addition, $|c^*(x)| \le \sum_{\alpha \in \mathcal{I}_{nk}}c^*_\alpha(x_\alpha)$ for all $x \in X$. In particular, since $c^*_\alpha(x_\alpha) < +\infty$ for all $x_\alpha \in X_\alpha$ and for all $\alpha \in \mathcal{I}_{nk}$, we conclude that $|c^*(x)| < +\infty$ for all $x \in X$. Thus, replacing $c$ with $c^*$ and replacing $c_\alpha$ with $c_\alpha^*$ for all $\alpha \in \mathcal{I}_{nk}$, we may assume that $|c(x)| < +\infty$ for all $x \in X$ and $0 \le c_\alpha(x_\alpha) < +\infty$ for all $x_\alpha \in X_\alpha$ and for all $\alpha \in \mathcal{I}_{nk}$.

Denote
\[
\widehat{J} = \sum_{\alpha \in \mathcal{I}_{nk}}\int_{X_\alpha}c_\alpha(x_\alpha)\,\mu_\alpha(dx_\alpha).
\]
The function $c_\alpha \colon X_\alpha \to [0, +\infty)$ is finite and integrable with respect to $\mu_\alpha$ for all $\alpha \in \mathcal{I}_{nk}$; in addition,
\[
\sum_{\alpha \in \mathcal{I}_{nk}}(-c_\alpha(x_\alpha)) \le c(x) \text{ for all $x \in X$.}
\]
Thus, $\{-c_\alpha\}_{\alpha \in \mathcal{I}_{nk}} \in \Psi_c(\mu_\alpha)$, and therefore the set $\Psi_c(\mu_\alpha)$ is non-empty and
\[
J \ge \sum_{\alpha \in \mathcal{I}_{nk}}\int_{X_\alpha}(-c_\alpha(x_\alpha))\,\mu_\alpha(dx_\alpha) = -\widehat{J}.
\]

Since the tuple of measures $\{\mu_\alpha\}$ is reducible, there exists a reducible measure $\mu \in \Pi(\mu_\alpha)$. Since $c_\alpha \in L^1(X_\alpha, \mu_\alpha)$, the extension of $c_\alpha$ to the space $X$ is integrable with respect to $\mu$. Thus, since $|c(x)| \le \sum_{\alpha \in \mathcal{I}_{nk}}c_\alpha(x_\alpha) \in L^1(X, \mu)$, we conclude that $c \in L^1(X, \mu)$.

Let $\{f_\alpha\}_{\alpha \in \mathcal{I}_{nk}} \in \Psi_c(\mu_\alpha)$. Since $f_\alpha \in L^1(X_\alpha, \mu_\alpha)$, the extension of $f_\alpha$ to the space $X$ is integrable with respect to $\mu$.  Hence,
\[
\sum_{\alpha \in \mathcal{I}_{nk}}\int_{X_\alpha}f_\alpha(x_\alpha)\,\mu_\alpha(dx_\alpha) = \int_X\sum_{\alpha \in \mathcal{I}_{nk}}f_\alpha(x_\alpha)\,\mu(dx)
\]
We have $\sum_{\alpha \in \mathcal{I}_{nk}}f_\alpha(x_\alpha) \le c(x)$ at all points except a zero $(n, k)$-thickness set. Since $\mu$ is a uniting measure, every set of zero $(n, k)$-thickness has zero measure with respect to $\mu$. Hence, $\sum_{\alpha \in \mathcal{I}_{nk}}f_\alpha(x_\alpha) \le c(x)$ for $\mu$-almost all $x \in X$, and therefore
\[
\int_X\sum_{\alpha \in \mathcal{I}_{nk}}f_\alpha(x_\alpha)\,\mu(dx) \le \int_X c(x)\,\mu(dx) \le \sum_{\alpha \in \mathcal{I}_{nk}}\int_{X_\alpha}c_\alpha(x_\alpha)\,\mu_\alpha(dx_\alpha) = \widehat{J}.
\]
Thus, we conclude that
\[
\sum_{\alpha \in \mathcal{I}_{nk}}\int_{X_\alpha}f_\alpha(x_\alpha)\,\mu_\alpha(dx_\alpha) \le \widehat{J} \;\text{ for all $\{f_\alpha\} \in \Psi_c(\mu_\alpha)$},
\]
and therefore $J \le \widehat{J}$. In particular, the supremum in \cref{eq:relaxed_dual_sup} is finite.

Consider the maximizing sequence of tuples of functions $\{f_\alpha^{(t)}\}_{\alpha \in \mathcal{I}_{nk}} \in \Psi_c(\mu_\alpha)$ such that
\[
\sum_{\alpha \in \mathcal{I}_{nk}}\int_{X_\alpha}f_\alpha^{(t)}(x_\alpha)\,\mu_\alpha(dx_\alpha) \xrightarrow[n \to \infty]{} J.
\]
We may assume that
\begin{equation}\label{eq:int_Fn_bounded from_below}
\sum_{\alpha \in \mathcal{I}_{nk}}\int_{X_\alpha}f_\alpha^{(t)}(x_\alpha)\,\mu_\alpha(dx_\alpha) \ge - \widehat{J}\;\text{ for all $t$.}
\end{equation}

For each $t$ consider a finite $(n, k)$-function $F^{(t)}(x) = \sum_{\alpha \in \mathcal{I}_{nk}}f_\alpha^{(t)}(x_\alpha)$. Let us bound the norm of the function $F^{(t)}$ from above. Since $F^{(t)}(x) \le c(x)$ for all points except a zero $(n, k)$-thickness set, and $c(x) \le \sum_{\alpha \in \mathcal{I}_{nk}}c_\alpha(x_\alpha)$ for all $x \in X$, we conclude that $F^{(t)}(x) \le \sum_{\alpha \in \mathcal{I}_{nk}}c_\alpha(x_\alpha)$ for $\mu$-almost all $x \in X$. Finally, since $\sum_{\alpha \in \mathcal{I}_{nk}}c_\alpha(x_\alpha) \ge 0$, we have
\[
F^{(t)}(x) + |F^{(t)}(x)| = \max(0, 2F^{(t)}(x)) \le 2\sum_{\alpha \in \mathcal{I}_{nk}}c_\alpha(x_\alpha)
\]
for $\mu$-almost all $x \in X$. Combining this with inequality \cref{eq:int_Fn_bounded from_below} we get
\[
\norm{F^{(t)}}_{L^1(\mu)} = \int_X |F^{(t)}(x)|\,\mu(dx) \le 2\sum_{\alpha \in \mathcal{I}_{nk}}\int_{X_\alpha}c_\alpha(x_\alpha)\mu_\alpha(dx_\alpha) - \int_X F^{(t)}(x)\,\mu(dx) \le 3\widehat{J}.
\]

Since $\mu$ is reducible, for each $t$ by \cref{thm:existence_integrable_f_alpha_for_reducible} there exists a tuple of finite integrable functions $\{\widehat{f}_\alpha^{(t)}\}_{\alpha \in \mathcal{I}_{nk}}$ such that the equation
\[
F^{(t)}(x) = \sum_{\alpha \in \mathcal{I}_{nk}}\widehat{f}^{(t)}_\alpha(x_\alpha)
\]
holds for all $x \in X$ and
\[
\norm{\widehat{f}^{(t)}_\alpha}_{L^1(\mu_\alpha)} \le C_\mu\norm{F^{(t)}}_{L^1(\mu)} \le 3C_\mu \widehat{J} = C
\]
for all $\alpha \in \mathcal{I}_{nk}$. In particular, $\{\widehat{f}^{(t)}_\alpha\} \in \Psi_c(\mu_\alpha)$ for all $t$, and this sequence of tuples is also maximizing. Thus, replacing $\{f_\alpha^{(t)}\}$ with $\{\widehat{f}_\alpha^{(t)}\}$, we may assume that the inequality
\[
\norm{f_\alpha^{(t)}}_{L^1(\mu_\alpha)} \le C
\]
holds for all $\alpha \in \mathcal{I}_{nk}$ and for all $t$.

In particular,
\[
\sup_t \norm{f_\alpha^{(t)}}_{L^1(\mu_\alpha)} < +\infty
\]
for all $\alpha \in \mathcal{I}_{nk}$. Hence, using the Komlo{\'s} theorem and passing, if necessary, to subsequences, we may assume that the sequence of functions
\[
g_\alpha^{(t)}(x_\alpha) = \frac{1}{t}\left(f^{(1)}_\alpha + \dots + f_\alpha^{(t)}\right), \ \ t \in \mathbb{N},
\]
converges to some function $g_\alpha \in L^1(X_\alpha, \mu_\alpha)$ $\mu_\alpha$-almost everywhere in $X_\alpha$ for all $\alpha \in \mathcal{I}_{nk}$.

For each $t$ consider the finite $(n, k)$-function
\[
G^{(t)}(x) = \sum_{\alpha \in \mathcal{I}_{nk}}g_\alpha^{(t)}(x_\alpha) = \frac{1}{t}\left(F^{(1)}(x) + \dots + F^{(t)}(x)\right).
\]
We have $G^{(t)}(x) \le c(x)$ for all $x \in X$ except a zero $(n, k)$-thickness set, and therefore $\{g^{(t)}_\alpha\}_{\alpha \in \mathcal{I}_{nk}} \in \Psi_c(\mu_\alpha)$ for all $t$. In addition, it follows from the properties of the Ce{\'s}aro mean that the sequence of tuples $\{g_\alpha^{(t)}\}_{\alpha \in \mathcal{I}_{nk}}$ is maximizing as well as $\{f_\alpha^{(t)}\}_{\alpha \in \mathcal{I}_{nk}}$.

Let us verify that $\{g_\alpha\}_{\alpha \in \mathcal{I}_{nk}} \in \Psi_c(\mu_\alpha)$. For every $t$ there exists a tuple of measurable subsets $\{A^{(t)}_\alpha\}_{\alpha \in \mathcal{I}_{nk}}$, $A_\alpha^{(t)} \subseteq X_\alpha$ such that $\mu_\alpha(A_\alpha^{(t)}) = 1$ and $G^{(t)}(x) \le c(x)$ for all $x \in X$ such that $x_\alpha \in A_\alpha^{(t)}$ for all $\alpha \in \mathcal{I}_{nk}$. In addition, for each $\alpha \in \mathcal{I}_{nk}$ there exists a measurable subset $A_\alpha' \subseteq X_\alpha$ such that $\mu_\alpha(A_\alpha') = 1$ and if $x_\alpha \in A_\alpha'$, then $g^{(t)}(x_\alpha) \to g_\alpha(x_\alpha)$ as $t \to \infty$.

For $\alpha \in \mathcal{I}_{nk}$, let
\[
A_\alpha = A_\alpha' \cap \left(\bigcap_{t = 1}^{\infty}A_\alpha^{(t)}\right).
\]
For any $x \in \bigcap_{\alpha \in \mathcal{I}_{nk}}\prj_\alpha^{-1}(A_\alpha)$ we have $\sum_{\alpha \in \mathcal{I}_{nk}}g^{(t)}_\alpha(x_\alpha) \le c(x)$ for all $t$ and
\[
\sum_{\alpha \in \mathcal{I}_{nk}}g^{(t)}_\alpha(x_\alpha) \xrightarrow[t \to \infty]{} \sum_{\alpha \in \mathcal{I}_{nk}}g_\alpha(x_\alpha).
\]
Thus, if $x \in \bigcap_{\alpha \in \mathcal{I}_{nk}}\prj_\alpha^{-1}(A_\alpha)$, then $\sum_{\alpha \in \mathcal{I}_{nk}}g_\alpha(x_\alpha) \le c(x)$, and therefore, since $\mu_\alpha(A_\alpha) = 1$, we conclude that $\{g_\alpha\}_{\alpha \in \mathcal{I}_{nk}} \in \Psi_c(\mu_\alpha)$.

Consider the finite $(n, k)$-function $G(x) = \sum_{\alpha \in \mathcal{I}_{nk}}g_\alpha(x_\alpha)$. We have
\[
G^{(t)}(x) \xrightarrow[t \to \infty]{} G(x)
\]
for all $x \in X$ except a zero $(n, k)$-thickness set, and therefore the sequence of functions $\{G^{(t)}\}$ converges pointwise to $G$ $\mu$-almost everywhere. In addition, 
\[
G^{(t)}(x) \le \sum_{\alpha \in \mathcal{I}_{nk}}c_\alpha(x_\alpha) \in L^1(X, \mu)
\]
for $\mu$-almost all $x$, and therefore it follows from the reverse Fatou lemma that
\[
J = \lim_{t \to \infty}\int_X G^{(t)}(x)\,\mu(dx) \le \int_X G(x)\,\mu(dx) = \sum_{\alpha \in \mathcal{I}_{nk}}\int_{X_\alpha}g_\alpha(x_\alpha)\,\mu_\alpha(x_\alpha). 
\]
Thus, the supremum in \cref{eq:relaxed_dual_sup} is attained on the tuple of functions $\{g_\alpha\}_{\alpha \in \mathcal{I}_{nk}}$.
\end{proof}

Combining this result with \cref{thm:noncompact_duality}, we get the following general duality theorem for the case of reducible projections.

\begin{theorem}[General duality theorem]
For every $1 \le i \le n$, let $X_i$ be a Polish space, let $\{\mu_\alpha\}_{\alpha \in \mathcal{I}_{nk}}$, $\mu_\alpha \in \mathcal{P}(X_\alpha)$ be a reducible tuple of probability measures, and let $c \in C_L(X, \mu_\alpha)$ be a continuous cost function on the space $X$. Then there exists a uniting measure $\pi \in \Pi(\mu_\alpha)$ and a tuple of integrable functions $\{f_\alpha\}_{\alpha \in \mathcal{I}_{nk}}$, $f_\alpha \colon X_\alpha \to [-\infty, +\infty)$, such that \[
\sum_{\alpha \in \mathcal{I}_{nk}}f_\alpha(x_\alpha) \le c(x) \text{ for all $x \in X$}
\]
and
\[
\int_X c(x)\,\pi(dx) = \sum_{\alpha \in \mathcal{I}_{nk}}\int_{X_\alpha}f_\alpha(x_\alpha)\,\mu_\alpha(dx_\alpha).
\]
In particular, $\pi$ is a solution to the related primal $(n, k)$-problem, and $\{f_\alpha\}_{\alpha \in \mathcal{I}_{nk}}$ is a solution to the related dual $(n, k)$-problem.
\end{theorem}

\subsection{Unreachability of the supremum in the dual problem in the irreducible case}
In contrast to the multi-marginal case, in the theorem proved above, the essential requirement is the irreducibility of the set of measures $\mu_\alpha$. In the following paragraph we construct a multistochastic $(3, 2)$-problem with a bounded continuous cost function such that the supremum in the corresponding dual problem can not be attained.

Let $X_1 = X_2 = X_3 = \mathbb{N}$. For $1 \le i \le 3$, the space $X_i$ is a Polish space equipped with the discrete topology. For each $n$ denote
\[
    A_n = \{(n + 1, n, n), (n, n + 1, n), (n, n, n + 1)\}.
\]
One can easily verify that these sets are pairwise disjoint.

Consider the measure $\mu_p$ on the space $X = X_1 \times X_2 \times X_3$ defined as follows:
\begin{align*}
\mu_p(n_1, n_2, n_3) = \begin{dcases}
\frac{2}{(\pi n)^2}, &\text{if $(n_1, n_2, n_3) \in A_n$ for some $n$},\\
0, &\text{otherwise}.
\end{dcases}
\end{align*}
We have
\[
\mu_p(X) = \sum_{n = 1}^\infty |A_n| \cdot \frac{2}{(\pi n)^2} = \frac{6}{\pi^2}\sum_{n = 1}^\infty\frac{1}{n^2} = 1,
\]
and therefore the measure $\mu_p$ is a probability measure.

Consider another measure $\mu_\eps$ on the space $X$: let $\mu_\eps(n_1, n_2, n_3) = 2^{-n_1-n_2-n_3}$ for all $(n_1, n_2, n_3) \in X$. We have
\[
\mu_\eps(X) = \sum_{(n_1, n_2, n_3) \in X}\frac{1}{2^{n_1 + n_2 + n_3}} = \left(\sum_{n_1 = 1}^\infty \frac{1}{2^{n_1}}\right)\cdot\left(\sum_{n_2 = 1}^\infty \frac{1}{2^{n_2}}\right) \cdot \left(\sum_{n_3 = 1}^\infty \frac{1}{2^{n_3}}\right) = 1,
\]
and therefore $\mu_\eps$ is a probability measure too.

\begin{lemma}\label{lem:pi_x_ge_0_on_An}
Consider the probability measure $\mu = (1 - \alpha)\mu_p + \alpha\mu_\eps$, where $0 \le \alpha \le 1$. For $\{i, j\} \in \mathcal{I}_{3, 2}$, denote $\mu_{ij} = \prj_{ij}(\mu)$. If $\gamma \in \Pi(\mu_{ij})$ is a uniting measure for the tuple of projections $\{\mu_{ij}\}$, which means it has the same projections as $\mu$, then \[
\gamma(x) \ge \frac{2(1 - \alpha)}{\pi^2}\left(\frac{1}{n^2} - \frac{1}{(n + 1)^2}\right) - \frac{\alpha}{2^n}
\]
for all $x \in A_n$ for all $n$.
\end{lemma}
\begin{proof}
First, let us find $\mu_{ij}$ explicitly. We have
\[
\prj_{ij}(\mu_\eps)(n_i, n_j) = \sum_{n = 1}^\infty\frac{1}{2^{n_i + n_j + n}} = \frac{1}{2^{n_i + n_j}} \text{ for all $(n_i, n_j) \in \mathbb{N}^2$}.
\]
In addition, one can easily verify that
\begin{align*}
    \prj_{ij}(\mu_p)(n_i, n_j) = \begin{dcases}
    0, &\text{if $|n_i - n_j| \ge 2$},\\
    \frac{2}{(\pi n)^2}, &\text{if $|
    n_i - n_j| \le 1$ and $\min(n_i, n_j) = n$}.
    \end{dcases}
\end{align*}
In particular, since $\mu_{ij} = (1 - \alpha)\prj_{ij}(\mu_p) + \alpha \prj_{ij}(\mu_\eps)$, we obtain the following equations:
\begin{align}\label{eq:mu_ij_value_in_natural_points}
\mu_{ij}(n_i, n_j) = \begin{dcases}
\frac{\alpha}{2^{n_i + n_j}} &\text{if $|n_i - n_j| \ge 2$},\\
\frac{2(1 - \alpha)}{(\pi n)^2} + \frac{\alpha}{2^{n_i + n_j}} &\text{if $|n_i - n_j| \le 1$ and $\min(n_i, n_j) = n$}.
\end{dcases}
\end{align}

Fix a positive integer $m$. Consider the following functions $f_{ij} \colon \mathbb{N}^2 \to \mathbb{R}$:
\begin{align*}
    f_{12}(n_1, n_2) &= \begin{cases}
    1,\hphantom{-} &\text{if $(n_1, n_2) = (m + 1, m)$},\\
    0, &\text{otherwise};
    \end{cases}\\
    f_{13}(n_1, n_3) &= \begin{cases}
    -1, &\text{if $n_1 = m + 1$ and $n_3 \in \{m - 1, m + 1\}$},\\
    0, &\text{otherwise};
    \end{cases}\\
    f_{23}(n_2, n_3) &= \begin{cases}
    -1, &\text{if $n_2 = m$ and $n_3 \not\in \{m - 1, m, m + 1\}$},\\
    0, &\text{otherwise}.
    \end{cases}
\end{align*}
The function $f_{ij}$ is bounded, and therefore is integrable with respect to $\mu_{ij}$. Using equations \cref{eq:mu_ij_value_in_natural_points} we get
\begin{align*}
    \int f_{12}\,d\mu_{12} &= \mu_{12}(m + 1, m) = \frac{2(1 - \alpha)}{\pi^2m^2} + \frac{\alpha}{2^{2m + 1}},\\
    \int f_{13}\,d\mu_{13} &= -\mu_{13}(m + 1, m - 1) - \mu_{13}(m + 1, m + 1) = -\frac{\alpha}{2^{2m}} - \frac{2(1 - \alpha)}{\pi^2(m + 1)^2} - \frac{\alpha}{2^{2m + 2}},\\
    \int f_{23}\,d\mu_{23} &= -\smashoperator{\sum_{n \not\in \{m - 1, m, m + 1\}}}\quad\mu_{23}(m, n)  = -\sum_{n = 1}^\infty\frac{\alpha}{2^{m + n}} +  \frac{\alpha}{2^{2m - 1}} + \frac{\alpha}{2^{2m}} + \frac{\alpha}{2^{2m + 1}}
    > -\frac{\alpha}{2^m} + \frac{\alpha}{2^{2m}}.
\end{align*}
Summarizing this, we obtain
\begin{align}\label{eq:sum_f_ij_integrals_sup_unreach}
\begin{split}
\int f_{12}\,d\mu_{12} &+ \int f_{13}\,d\mu_{13} + \int f_{23}\,d\mu_{23} \\
&> \frac{2(1 - \alpha)}{\pi^2}\left(\frac{1}{m^2} - \frac{1}{(m + 1)^2}\right) + \left(\frac{\alpha}{2^{2m + 1}} - \frac{\alpha}{2^{2m + 2}}\right) - \frac{\alpha}{2^m}\\
&>\frac{2(1 - \alpha)}{\pi^2}\left(\frac{1}{m^2} - \frac{1}{(m + 1)^2}\right) - \frac{\alpha}{2^m}.
\end{split}
\end{align}

Consider the $(3, 2)$-function \[F(n_1, n_2, n_3) = f_{12}(n_1, n_2) + f_{13}(n_1, n_3) + f_{23}(n_2, n_3).\]
Let us verify that $F(n_1, n_2, n_3) \le 0$ if $(n_1, n_2, n_3) \ne (m + 1, m, m)$. Indeed, since $f_{13} \le 0$ and $f_{23} \le 0$, we conclude that if $F(n_1, n_2, n_3) > 0$, then $f_{12}(n_1, n_2) > 0$, and therefore $(n_1, n_2) = (m + 1, m)$. If $n_3 \not\in \{m - 1, m, m + 1\}$, then by construction $f_{23}(m, n_3) = -1$, and $f_{13}(m + 1, n_3) = 0$, and therefore $F(m + 1, m, n_3) = 0$. Otherwise, if $n_3 \in \{m - 1, m + 1\}$, then $f_{13}(m + 1, n_3) = -1$ and $f_{23}(m, n_3) = 0$, and therefore $F(m + 1, m, n_3) = 0$ too.

In addition, $F(m + 1, m, m) = 1$, and therefore if $\gamma$ is a probability measure on the space $X$, then
\[
\int_X F(n_1, n_2, n_3)\,\gamma(dn_1, dn_2, dn_3) \le \gamma(m + 1, m, m).
\]
Combining this with inequality \cref{eq:sum_f_ij_integrals_sup_unreach}, we conclude that if $\gamma \in \Pi(\mu_{ij})$, then
\begin{align*}
    \gamma(m + 1, m, m) &\ge \int_X F(n_1, n_2, n_3)\,\gamma(dn_1, dn_2, dn_3)\\
    &=\int f_{12}\,d\mu_{12} + \int f_{13}\,d\mu_{13} + \int f_{23}\,d\mu_{23} \\
    &> \frac{2(1 - \alpha)}{\pi^2}\left(\frac{1}{m^2} - \frac{1}{(m + 1)^2}\right) - \frac{\alpha}{2^m}.
\end{align*}
For the remaining points of $A_m$ the inequality is proved in the same manner. 
\end{proof}

\begin{corollary}\label{cons:pi_x_ge_0_on_An}
There exists a real $\alpha_0 \in (0, 1)$ such that if $\gamma \in \Pi(\mu_{ij})$, then $\gamma(x) > 0$ for all $x \in A_n$ for all $n$, where $\mu_{ij} = \prj_{ij}((1 - \alpha_0)\mu_p + \alpha_0\mu_\eps)$.
\end{corollary}
\begin{proof}
By \cref{lem:pi_x_ge_0_on_An} we only need to prove that there exists $\alpha_0 \in (0, 1)$ such that
the inequality
\[
\frac{2(1 - \alpha_0)}{\pi^2}\left(\frac{1}{n^2} - \frac{1}{(n + 1)^2}\right) - \frac{\alpha_0}{2^n} > 0
\]
holds for all $n \in \mathbb{N}$, or equivalently
\begin{equation}\label{eq:alpha_0_ineq}
\frac{2(1 - \alpha_0)}{\pi^2\alpha_0} > \frac{2^{-n}}{\frac{1}{n^2} - \frac{1}{(n + 1)^2}}.
\end{equation}

One can easily verify that the function in the right hand-side of the inequality converges to $0$, and therefore there exists a constant $M$ such that the inequality
\[
M \ge \frac{2^{-n}}{\frac{1}{n^2} - \frac{1}{(n + 1)^2}}
\]
holds for all positive integer $n$. Thus, the inequality \cref{eq:alpha_0_ineq} follows from
\[
\frac{2(1 - \alpha_0)}{\pi^2\alpha_0} > M
\]
and therefore every $\alpha_0$ such that $0 < \alpha_0 < 2 / (M\pi^2 + 2)$ is suitable.
\end{proof}

\begin{theorem}\label{thm:unreachibility_of_supremum}
Let $\alpha_0$ be the constant constructed in \cref{cons:pi_x_ge_0_on_An}. Let $\mu = (1 - \alpha_0)\mu_p + \alpha_0\mu_\eps$, and for $\{i, j\} \in \mathcal{I}_{3, 2}$ let $\mu_{ij} = \prj_{ij}(\mu)$. Consider the cost function $c \colon X \to \{0, 1\}$: $c(x) = 1$ if $x \in A_n$ for some $n$, and $c(x) = 0$ otherwise. Then the supremum in the corresponding dual $(3, 2)$-problem can not be attained.
\end{theorem}
\begin{proof}
The cost function $c$ is a bounded continuous function on the space $X$ equipped with the discrete topology. In addition, the set $\Pi(\mu_{ij})$ is non-empty, and therefore it follows from \cref{thm:noncompact_duality} that
\[
\min_{\gamma \in \Pi(\mu_{ij})}\int_X c\,d\gamma = \sup\left\{\sum \int_{X_{ij}} f_{ij}\,d\mu_{ij} \colon \sum f_{ij}(x_i, x_j) \le c(x_1, x_2, x_3)\right\}.
\]
Assume that the supremum in the dual problem is attained. Then there exists a uniting measure $\gamma \in \Pi(\mu_{ij})$ and a tuple of integrable functions $\{f_{ij}\}_{\{i, j\} \in \mathcal{I}_{3, 2}}$, $f_{ij} \colon X_{ij} \to [-\infty, +\infty)$ such that
\[
f_{12}(n_1, n_2) + f_{13}(n_1, n_3) + f_{23}(n_2, n_3) \le c(n_1, n_2, n_3)
\]
for all $(n_1, n_2, n_3) \in X$ and
\[
\int_{X_{12}} f_{12}\,d\mu_{12} + \int_{X_{13}} f_{13}\,d\mu_{13} + \int_{X_{23}} f_{23}\,d\mu_{23} = \int_X c\,d\gamma.
\]

It follows from equation \cref{eq:mu_ij_value_in_natural_points} that $\mu_{ij}(n_i, n_j) > 0$ for all pairs of positive integers $(n_i, n_j)$. Hence, since $f_{ij}$ is integrable with respect to $\mu_{ij}$, we conclude that $f_{ij}$ can not take value $-\infty$.

Consider the finite $(3, 2)$-function
\begin{equation}\label{eq:integer_32_func_F_definition}
F(n_1, n_2, n_3) = f_{12}(n_1, n_2) + f_{13}(n_1, n_3) + f_{23}(n_2, n_3).
\end{equation}
Since $f_{ij}$ is integrable with repsect to $\mu_{ij}$ and the measure $\gamma$ is uniting, the function $F$ is integrable with respect to $\gamma$ and
\[
\int_X F\,d\gamma = \int_{X_{12}} f_{12}\,d\mu_{12} + \int_{X_{13}} f_{13}\,d\mu_{13} + \int_{X_{23}} f_{23}\,d\mu_{23} = \int_X c\,d\gamma.
\]
Since in addition $F(n_1, n_2, n_3) \le c(n_1, n_2, n_3)$ for all $(n_1, n_2, n_3) \in X$, we conclude that $F(n_1, n_2, n_3) = c(n_1, n_2, n_3)$ $\gamma$-almost everywhere. It follows from \cref{cons:pi_x_ge_0_on_An} that $\gamma(x) > 0$ if $x \in A_n$ for some $n$, and therefore
\begin{equation}\label{eq:F_value_on_An}
F(n + 1, n, n) = F(n, n + 1, n) = F(n, n, n + 1) = 1
\end{equation}
for all $n \in \mathbb{N}$.

One can easily verify using equation \cref{eq:integer_32_func_F_definition} that for all $n \in \mathbb{N}$ we have
\begin{align*}
    F(n + 1, &n + 1, n + 1) - F(n, n, n) \\
    =&\quad F(n, n + 1, n + 1) + F(n + 1, n, n + 1) + F(n + 1, n + 1, n)\\
    &-F(n + 1, n, n) - F(n, n + 1, n) - F(n, n, n + 1).
\end{align*}
Since, $F(n_1, n_2, n_3) \le c(n_1, n_2, n_3)$ and $c(n_1, n_2, n_3) = 0$ if the point $(n_1, n_2, n_3)$ is not contained in the set $\sqcup_{n = 1}^\infty A_n$, the inequality
\[
F(n, n + 1, n + 1) + F(n + 1, n, n + 1) + F(n + 1, n + 1, n) \le 0
\]
holds for every positive integer $n$. In addition, it follows from equation \cref{eq:F_value_on_An} that
\[
F(n + 1, n, n) + F(n, n + 1, n) + F(n, n, n + 1) = 3.
\]
Summarizing this, we conclude that $F(n + 1, n + 1, n + 1) \le F(n, n, n) -3$, and therefore \[
F(n, n, n) \le F(1, 1, 1) - 3(n - 1) \le c(1, 1, 1) - 3(n - 1) = -3(n - 1).
\]
for all $n \in \mathbb{N}$.

In particular, we conclude that for all $n \in \mathbb{N}$ the following inequality holds:
\[
|f_{12}(n, n)| + |f_{13}(n, n)| + |f_{23}(n, n)| \ge 3(n - 1).
\]
Using this inequality and equation \cref{eq:mu_ij_value_in_natural_points}, we can bound from below the $\sum \norm{f_{ij}}_{L^1(\mu_{ij})}$:
\begin{align*}
    \norm{f_{12}}_{L^1(\mu_{12})} &+ \norm{f_{13}}_{L^1(\mu_{13})} + \norm{f_{23}}_{L^1(\mu_{23})} \\
    &\ge \sum_{n = 1}^\infty \left(|f_{12}(n, n)|\cdot\mu_{12}(n, n) + |f_{13}(n, n)|\cdot\mu_{13}(n, n) + |f_{23}(n, n)|\cdot\mu_{23}(n, n) \right)\\
    &>\frac{2(1 - \alpha_0)}{\pi^2}\sum_{n = 1}^\infty\frac{1}{n^2}\left(|f_{12}(n, n)| + |f_{13}(n, n)| + |f_{23}(n, n)|\right)\\
    &\ge \frac{2(1 - \alpha_0)}{\pi^2}\sum_{n = 1}^\infty\frac{3(n - 1)}{n^2} = +\infty.
\end{align*}
Thus, at least one the functions $f_{ij}$ is not integrable, and this contradiction proves \cref{thm:unreachibility_of_supremum}.
\end{proof}

The measure $\mu$ constructed in \cref{thm:unreachibility_of_supremum} is strictly positive at every point of the space $X$. In particular, this means that $\mu$ is equivalent to $\prj_{1}(\mu) \otimes \prj_{2}(\mu) \otimes \prj_{3}(\mu)$. Thus, we obtain the following proposition, which demonstrates that we cannot replace “uniform equivalence” with simple equivalence.

\begin{proposition}
\label{X123}
Let $X_1 = X_2 = X_3 = \mathbb{N}$. There exists a probability measure $\mu$ on the space $X = X_1 \times X_2 \times X_3$ and a cost function $c: X \to \{0, 1\}$ such that the following conditions hold:
\begin{enumerate}[label=\upshape{(\roman*)}]
\item measure $\mu$ is equivalent (but not uniformly equivalent) to $\mu_1 \otimes \mu_2 \otimes \mu_3$, where $\mu_i = \prj_i\mu$;
\item there is no optimal solution to the dual problem for the cost function $c$ and projections $\mu_{ij}$, where $\mu_i = \prj_i\mu$.
\end{enumerate}
\end{proposition}

In the classical Monge-Kantorovich problem the dual solution may not exist provided $c$ is unbounded. In \cite{Schachermayer2009,Beiglbock2009} authors introduce the concept of strong $c$-monotonicity, which generalizes the $c$-monotonicity and allows us to find a generalized dual solution.

\begin{definition}
A Borel set $\Gamma \subseteq X \times Y$ is \textit{strongly $c$-monotone} if there exist Borel measurable functions $\varphi \colon X \to [-\infty, +\infty)$, $\psi \colon Y \to [-\infty, +\infty)$ such that $\varphi(x) + \psi(y) \le c(x, y)$ for all $(x, y) \in X \times Y$ and $\varphi(x) + \psi(y) = c(x, y)$ holds if $(x, y) \in \Gamma$. A transport plan $\pi \in \Pi(\mu, \nu)$ is \textit{strongly $c$-monotone} if $\pi$ is concentrated on a strongly $c$-monotone Borel set.
\end{definition}

One can easily verify that strong $c$-monotonicity implies $c$-monotonicity, and if there exists a solution to the dual problem, then every optimal transport plan is strongly $c$-monotone. In \cite{Beiglbock2009} authors prove that under general assumptions on the cost function the transport plan $\pi$ is optimal if and only if $\pi$ is strongly $c$-monotone.
\begin{theorem}[{{\cite[Theorem 3]{Beiglbock2009}}}]
Let $X$, $Y$ be Polish spaces equipped with Borel probability measures $\mu$, $\nu$, and let $c \colon X \times Y \to [0, \infty]$ be Borel measurable and $\mu \otimes \nu$-a.e. finite. Then a finite transport plan $\pi \in \Pi(\mu, \nu)$ is optimal if and only if it is strongly $c$-monotone.
\end{theorem}
In particular, for every finite optimal transport plan $\pi$ there exist (not necessary integrable) functions $\varphi$, $\psi$ such that $\varphi(x) + \psi(x) \le c(x, y)$ and the equaility holds $\pi$-a.e. We can naturally generalize the concept of strong $c$-monotonicity to the multistochastic Monge-Kantorovich problem as follows.
\begin{definition}
A Borel set $\Gamma \subset X$ is \textit{strongly $c$-monotone} if there exist Borel measurable functions $\{f_\alpha\}_{\alpha \in \mathcal{I}_{nk}}$, $f_\alpha \colon X_\alpha \to [-\infty, +\infty)$ such that the inequality
\[
\sum_{\alpha \in \mathcal{I}_{nk}}f_\alpha(x_\alpha) \le c(x)
\]
holds for all $x \in X$ and the equality is achieved if $x \in \Gamma$. A transport plan $\pi \in \Pi(\mu_\alpha)$ is \textit{strongly $c$-monotone} if $\pi$ is concentrated on a strongly $c$-monotone Borel set $\Gamma$.
\end{definition}

We do not know whether exists a strongly $c$-monotone transport plan in the problem considered in \cref{thm:unreachibility_of_supremum}. In what follows, we construct another example of the $(3, 2)$-problem and prove that in this example there is no strongly $c$-monotone optimal transport plan.

As in the previous example, let $X_1 = X_2 = X_3 = \mathbb{N}$. For each $n$ denote 
\[
B_n = \{(n, n + 1, n + 1), (n + 1, n, n + 1), (n + 1, n + 1, n)\}.
\]
Consider the following measure $\mu$ defined on the space $X_1 \times X_2 \times X_3$ as follows:
\begin{align}\label{eq:mu_nonstrong_definition}
\mu(n_1, n_2, n_3) = \begin{dcases}
\frac{1}{(\pi n)^2} &\text{if $(n_1, n_2, n_3) \in A_n \sqcup B_n$ for some $n$,}\\
0 &\text{otherwise.}
\end{dcases}
\end{align}
One can check that $\mu$ is a probability measure. Finally, for $\{i, j\} \in \mathcal{I}_{3, 2}$ denote $\mu_{ij} = \prj_{ij}(\mu)$.
\begin{lemma}
The measure $\mu$ is the only uniting measure for the tuple of projections $\{\mu_{ij}\}$.
\end{lemma}
\begin{proof}
Let $\gamma \in \Pi(\mu_{ij})$. For $\{i, j\} \in \mathcal{I}_{3,2}$, the projection $\mu_{ij}$ is concentrated on the set $\{(n_i, n_j) \in \mathbb{N}^2 \colon |n_i - n_j| \le 1\}$, and therefore the transport plan $\gamma$ is concentrated on the set \[
S = \{(n_1, n_2, n_3) \in \mathbb{N}^3 \colon \max\{n_1, n_2, n_3\} - \min\{n_1, n_2, n_3\} \le 1\} = \bigsqcup_{k = 1}^\infty\left(\{(k, k, k)\} \sqcup A_k \sqcup B_k\right).\]

One can easily verify that $\gamma$ is uniquely defined by its values on the diagonal, and if we denote $a_k = \gamma(k, k, k)$, then we have
\begin{align}\label{eq:gamma_nonstrong_formula}
\begin{split}
\gamma(n_1, n_2, n_3) = \begin{cases}
\mu(n_1, n_2, n_3) - (a_1 + \dots + a_n) &\text{if $(n_1, n_2, n_3) \in A_n$ for some $n$,}\\
\mu(n_1, n_2, n_3) + (a_1 + \dots + a_n) &\text{if $(n_1, n_2, n_3) \in B_n$ for some $n$.}
\end{cases}
\end{split}
\end{align}
We have $\mu(n_1, n_2, n_3) = (\pi n)^{-2}$ for all $(n_1, n_2, n_3) \in A_n$, and therefore $a_1 + \dots + a_n \le (\pi n)^{-2}$ for all $n$. Thus, since all $a_n$ are nonnegative, we conclude that $\gamma(k, k, k) = a_k = 0$ for all $k$, and therefore $\gamma = \mu$ by equation \cref{eq:gamma_nonstrong_formula}.
\end{proof}

It follows from the previous lemma that $\mu$ is the unique optimal solution to the multistochastic problem with arbitrary bounded cost function. Next, we construct the cost function $c$ such that $\mu$ is not strongly $c$-monotone. The existence of this example demonstrates that we can not generalize the equivalence of optimality and strongly $c$-monotonicity to the multistochastic case.
\begin{theorem}
Let $\mu$ be the measure on $\mathbb{N}^3$ defined in equation \cref{eq:mu_nonstrong_definition}, and let $\mu_{ij} = \prj_{ij}(\mu)$. Consider the cost function $c \colon \mathbb{N}^3 \to \{0, 1\}$: $c(x) = 1$ if $x \in B_n$ for some $n$, and $c(x) = 0$ otherwise. Then there are no functions $\{f_{ij}\}$, $f_{ij} \colon \mathbb{N}^2 \to [-\infty, +\infty)$ such that \[f_{12}(n_1, n_2) + f_{13}(n_1, n_3) + f_{23}(n_2, n_3) \le c(n_1, n_2, n_3)\] for all $(n_1, n_2, n_3) \in \mathbb{N}^3$ and the equality holds $\mu$-a.e.
\end{theorem}
\begin{proof}
Assuming the opposite, consider the following $(3, 2)$-function:
\begin{equation}
\label{f12f13f32}
F(n_1, n_2, n_3) = f_{12}(n_1, n_2) + f_{13}(n_1, n_3) + f_{23}(n_2, n_3).
\end{equation}

Since $\{f_{ij}\}$ satisfy the assumptions of the theorem, we have $F(n_1, n_2, n_3) = c(n_1, n_2, n_3)$ $\mu$-a.e. Hence, since $\mu(n_1, n_2, n_3) > 0$ for all $(n_1, n_2, n_3) \in A_n \sqcup B_n$, we get
\begin{align}\label{eq:F_nonstrong_on_AB}
\begin{split}
    &F(n + 1, n, n) = F(n, n + 1, n) = F(n, n, n + 1) = 0,\\
    &F(n, n + 1, n + 1) = F(n + 1, n, n + 1) = F(n + 1, n + 1, n) = 1
\end{split}
\end{align}
for all $n \in \mathbb{N}$.

Applying (\ref{f12f13f32}) one can easily verify the following equation:
\begin{align*}
    F(n, n, n) &+ F(n, n + 1, n + 1) + F(n + 1, n, n + 1) + F(n + 1, n + 1, n) \\
    &= F(n + 1, n + 1, n + 1) + F(n + 1, n, n) + F(n, n + 1, n) + F(n, n, n + 1).
\end{align*}
Combining this with equation \cref{eq:F_nonstrong_on_AB}, we get $F(n, n, n) + 3 = F(n + 1, n + 1, n + 1)$ for all $n$, and therefore the inequality \[
F(n, n, n) = F(n + k, n + k, n + k) - 3k \le c(n + k, n + k, n + k) - 3k \le 1 - 3k
\]
holds for all $n, k \in \mathbb{N}$. Thus, $F(n, n, n) = -\infty$ for all $n$. In particular, $F(1, 1, 1) = -\infty$, and therefore $f_{ij}(1, 1) = -\infty$ for some $\{i, j\} \in \mathcal{I}_{3, 2}$. Without loss of generality we may assume that $f_{12}(1, 1) = -\infty$. Then $F(1, 1, 2)$ is also equal to $-\infty$, and this contradicts equation \cref{eq:F_nonstrong_on_AB}.
\end{proof}
\section{Properties of the dual solution in \texorpdfstring{$(3,2)$}{(3,2)}-problem}
\subsection{Boundedness of the dual solution} In the classical Monge-Kantorovich problem for the bounded cost function $c(x, y)$ we can transform every solution to the dual problem to the bounded one, using Legendre transformation.
\begin{proposition}
Let $X$ and $Y$ be Polish spaces, let $\mu \in \mathcal{P}(X)$ and $\nu \in \mathcal{P}(Y)$, and let $c\colon X \times Y \to \mathbb{R}_+$ be a cost function. If $c$ is a bounded continuous cost function, then there exists a solution $(\varphi, \psi)$ to the related dual problem such that both $\varphi(x), \psi(y)$ lie between $-\norm{c}_{\infty}$ and $\norm{c}_{\infty}$ for all $x \in X$ and $y \in Y$.  
\end{proposition}
\begin{proof}
The proof is an adaptation of the argument from the proof of \cite[Theorem 1.3]{Villani}. Let $(\varphi, \psi)$ be a solution to the dual problem provided by \cite[Theorem 2.4.3]{RR}. If $\pi$ is a solution to the related primal problem, then $\varphi(x) + \psi(y) = c(x, y)$ $\pi$-a.e. In particular, there exists a point $(x_0, y_0) \in X \times Y$ such that $\varphi(x_0) + \psi(y_0) = c(x_0, y_0) \ge 0$. For any real number $s$ the pair of functions $(\varphi - s, \psi + s)$ is also a solution to the dual problem. By a proper choice of $s$, we can ensure 
\[
\varphi(x_0) \ge 0, \;\psi(y_0) \ge 0.
\]

Since $\varphi(x) + \psi(y) \le c(x, y)$, we have $\varphi(x) \le c(x, y_0) - \phi(y_0) \le c(x, y_0)$ for all $x$, and $\psi(y) \le c(x_0, y) - \varphi(x_0) \le c(x_0, y)$ for all $y$. Consider the Legendre transformation of the function $\varphi$:
\[
\overline{\varphi}(x) = \inf_{y \in Y}(c(x, y) - \psi(y)).
\]
By construction, $\overline{\varphi}(x) + \psi(y) \le c(x, y)$ for all $x \in X$ and for all $y \in Y$. From the inequality $\varphi(x) \le c(x, y) - \psi(y)$ we see that $\overline{\varphi}(x) \ge \varphi(x)$ for all $x$. Since $\overline{\varphi}(x) \le c(x, y) - \psi(y)$ for all $y$, we have \begin{align*}
\overline{\varphi}(x) &\le c(x, y_0) - \psi(y_0) \le \norm{c}_{\infty},
\intertext{and it follows from the inequality $\psi(y) \le c(x_0, y)$ that}
\overline{\varphi}(x) &\ge \inf_{y \in Y}(c(x, y) - c(x_0, y)) \ge -\norm{c}_{\infty}.
\end{align*}
Hence, $\overline{\varphi}$ is an integrable function; since $\overline{\varphi}(x) \ge \varphi(x)$ for all $x$, we have \[
\int_X \overline{\varphi}(x)\,\mu(dx) + \int_Y\psi(y)\,\nu(dy) \ge \int_X \varphi(x)\,\mu(dx) + \int_Y\psi(y)\,\nu(dy),\]
and therefore $(\overline{\varphi}, \psi)$ is a solution to the dual problem.

Finally, define
\[
\overline{\psi}(y) = \inf_{x \in X}(c(x, y) - \overline{\varphi}(x)).
\]
By the same arguments we conclude that $(\overline{\varphi}, \overline{\psi})$ is a solution to the dual problem and $-\norm{c}_{\infty} \le \overline{\psi}(y) \le \norm{c}_{\infty}$ for all $y \in Y$.
\end{proof}
We want to generalize this observation to the multistochastic case.
\begin{definition}
Given finite measures $\mu$ and $\nu$ on the space $X$, we say that $\mu \ll_B \nu$ if there exists a positive real $M$ such that $\mu \le M \cdot \nu$.
\end{definition}
The following properties trivially follow from the
definition.
\begin{proposition}
Let $\mu$ and $\nu$ be finite measures on the space $X$. Suppose that $\mu \ll_B \nu$. Then
\begin{assertions}
    \item  $\mu$ is absolutely continuous with respect to $\nu$; 
    \item \label{st:l1_inclusion} $L^1(X, \mu) \supseteq L^1(X, \nu)$;
    \item \label{st:llB_prj} if $X = X_1 \times \dots \times X_n$, then $\prj_{\alpha}\mu \ll_B \prj_{\alpha}\nu$ for all $\alpha \in \mathcal{I}_n$.
\end{assertions}
\end{proposition}
\begin{definition}
Let $X_1, \dots, X_n$ be Polish spaces, let $\pi \in \mathcal{P}(X)$, and let $\nu_\alpha$ be a probability measure on $X_\alpha$ for some $\alpha \in \mathcal{I}_n$ such that $\nu_\alpha \ll_B \pi_\alpha$. Let $\rho$ be a density function of $\nu_\alpha$ with respect to $\pi_\alpha$. Then denote by $\Up_\alpha(\nu_\alpha, \pi)$ the measure $\rho^*(x) \cdot \pi$, where $\rho^*(x) = \rho(x_\alpha)$ for all $x \in X$.
\end{definition}
\begin{proposition}
Let $X_1, \dots, X_n$ be Polish spaces, let $\pi \in \mathcal{P}(X)$, and let $\nu_\alpha$ be a probability measure on $X_\alpha$ for some $\alpha \in \mathcal{I}_n$ such that $\nu_\alpha \ll_B \pi_\alpha$. Then
\begin{assertions}
    \item \label{st:up_well_defined} the measure $\Up_\alpha(\nu_\alpha, \pi)$ is well-defined;
    \item \label{st:up_bounded} $\Up_\alpha(\nu_\alpha, \pi) \ll_B \pi$;
    \item \label{st:up_sup_prj} if $\beta \supseteq \alpha$, then $\prj_\beta(\Up_\alpha(\nu_\alpha, \pi)) = \Up_\alpha(\nu_\alpha, \pi_\beta)$;
    \item \label{st:up_sub_prj} if $\beta \subseteq \alpha$, then $\prj_\beta(\Up_\alpha(\nu_\alpha, \pi)) = \prj_\beta(\nu_\alpha)$;
    \item \label{st:up_decomposition} if $\pi = \mu_\alpha \otimes \mu_\beta$ for $\mu_\alpha \in \mathcal{P}(X_\alpha)$ and $\mu_\beta \in \mathcal{P}(X_\beta)$, then $\Up_\alpha(\nu_\alpha, \pi) = \nu_\alpha \otimes \mu_\beta$.
\end{assertions}
\end{proposition}
\begin{proof}
\Cref{st:up_well_defined} is trivial: if $\nu_\alpha = \rho_1 \cdot \pi_\alpha = \rho_2 \cdot \pi_\alpha$, then $\rho_1(x_\alpha) = \rho_2(x_\alpha)$ for $\pi_\alpha$-a.e. $x_\alpha \in X_\alpha$, and therefore $\rho_1^*(x) = \rho_2^*(x)$ for $\pi$-a.e. $x \in X$. In addition, since $\nu_\alpha \ll_B \pi_\alpha$, there exists a positive real $M$ such that $\rho(x_\alpha) \le M$ for $\pi_\alpha$-a.e. $x_\alpha \in X_\alpha$, and therefore $\rho^*(x) \le M$ for $\pi$-a.e. $x \in X$. Hence, $\rho^* \in L^1(X, \pi)$ and the measure $\rho^* \cdot \pi$ is well-defined. Furthermore, since $\rho^* \le M$ $\pi$-a.e, we have $\Up(\nu_\alpha, \pi) \le M \cdot \pi$; thus, $\Up(\nu_\alpha, \pi) \ll_B \pi$ and \cref{st:up_bounded} holds.

We have $\Up(\nu_\alpha, \pi) = \rho(x_\alpha) \cdot \pi$. The function $\rho$ does not depend on coordinates $x_i$ for all $i \not\in \alpha$. Hence, if $\beta \supseteq \alpha$ and $\beta \in \mathcal{I}_n$, then $\prj_\beta(\rho(x_\alpha) \cdot \pi) = \rho(x_\alpha) \cdot \pi_\beta$. Since $\prj_\alpha(\pi_\beta) = \pi_\alpha$, we conclude that $\Up_\alpha(\nu_\alpha, \pi_\beta) = \rho(x_\alpha) \cdot \pi_\beta$. Thus, if $\beta \supseteq \alpha$, then $\prj_\beta(\Up_\alpha(\nu_\alpha, \pi)) = \Up_\alpha(\nu_\alpha, \prj)$, and this implies \cref{st:up_sup_prj}. In addition, we have $\prj_\alpha(\rho(x_\alpha) \cdot \pi) = \rho(x_\alpha) \cdot \pi_\alpha = \nu_\alpha$, and therefore $\prj_{\alpha}(\Up_\alpha(\nu_\alpha, \pi)) = \nu_\alpha$. Hence, if $\beta \subseteq \alpha$, then $\prj_\beta(\Up_\alpha(\nu_\alpha, \pi)) = \prj_\beta \circ \prj_\alpha(\Up_\alpha(\nu_\alpha, \pi)) = \prj_\beta(\nu_\alpha)$, and this implies \cref{st:up_sub_prj}.

Finally, suppose that $\pi = \mu_\alpha \otimes \mu_\beta$. Then $\pi_\alpha = \mu_\alpha$, and therefore $\nu_\alpha = \rho \cdot \pi_\alpha = \rho \cdot \mu_\alpha$. Thus, $\nu_\alpha \otimes \mu_\beta = (\rho(x_\alpha) \cdot \mu_\alpha) \otimes \mu_\beta = \rho(x_\alpha) \cdot \pi = \Up_\alpha(\nu_\alpha, \pi)$, and this implies \cref{st:up_decomposition}.
\end{proof}

Let $X_1$, $X_2$, $X_3$ be Polish spaces, let $\mu_i \in \mathcal{P}(X_i)$ for $1 \le i \le i$, and let $\mu_{ij} = \mu_i \otimes \mu_j$ for all $\{i, j\} \in \mathcal{I}_{3,2}$. Let $c \colon X \to \mathbb{R}_+$ be a nonnegative bounded continuous cost function. The space $\Pi(\mu_{ij})$ is non-empty, since $\mu_1 \otimes \mu_2 \otimes \mu_3 \in \Pi(\mu_{ij})$, and therefore by \cref{thm:noncompact_duality} there is no duality gap. In addition, since the family of measures $\{\mu_{ij}\}$ is reducible, by \cref{thm:existence_multistochastic_dual_solution} there exists a solution to the related dual problem. Thus, there exists a solution $\pi \in \Pi(\mu_{ij})$ to the primal problem and a solution $\{f_{ij}\}$, $f_{ij} \in L^1(X_{ij}, \mu_{ij})$ to the dual problem, and
\[
\int_{X}c\,d\pi = \int_{X_{12}}f_{12}\,d\mu_{12} + \int_{X_{13}}f_{13}\,d\mu_{13} + \int_{X_{23}}f_{23}\,d\mu_{23}.
\]
\begin{lemma}\label{lem:f_ij_integrability}
Let $\widetilde{\pi}$ be a probability measure on $X$. Suppose that there exists $\gamma \in \Pi(\mu_{ij})$ such that $\widetilde{\pi} \ll_B \gamma$. Then extensions of all $f_{12}$, $f_{13}$ and $f_{23}$ to the space $X$ are integrable with respect to the measure $\widetilde{\pi}$.
\end{lemma}
\begin{proof}
The extension of $f_{ij}$ is integrable with respect to $\widetilde{\pi}$ if and only if $f_{ij} \in L^1(X_{ij}, \prj_{ij}(\widetilde{\pi}))$. Since $\widetilde{\pi} \ll_B \gamma$, by \cref{st:llB_prj} we have $\prj_{ij}(\widetilde{\pi}) \ll_B \prj_{ij}(\gamma) = \mu_{ij}$, and therefore by \cref{st:l1_inclusion} we conclude that $L^1(X_{ij}, \prj_{ij}(\widetilde{\pi})) \supseteq L^1(X_{ij}, \mu_{ij}) \owns f_{ij}$.
\end{proof}
Denote $F(x_1, x_2, x_3) = f_{12}(x_1, x_2) + f_{13}(x_1, x_3) + f_{23}(x_2, x_3)$.
\begin{lemma}
Let $\widetilde{\pi}$ be a probability measure on $X$. Suppose that there exists $\gamma \in \Pi(\mu_{ij})$ such that $\widetilde{\pi} \ll_B \gamma$. Then
\begin{assertions}
    \item \label{st:f_ij_integrability} the function $F$ and the extensions of all $f_{12}$, $f_{13}$ and $f_{23}$ to the space $X$ are integrable with respect to the measure $\widetilde{\pi}$;
    \item \label{st:int_F_bounded_from_above} $\displaystyle\int_X F\,d\widetilde{\pi} \le \norm{c}_{\infty}$;
    \item \label{st:int_F_bounded_from_below} if $\widetilde{\pi} \ll_B \pi$, then $\displaystyle\int_X F\,d\widetilde{\pi} \ge 0$.
\end{assertions}
\end{lemma}
\begin{proof}
\Cref{st:f_ij_integrability} trivially follows from \cref{lem:f_ij_integrability}. Since $\{f_{ij}\}$ is a solution to the dual problem, we have $F(x_1, x_2, x_3) \le c(x_1, x_2, x_3)$ for all $x \in X$. In particular,
\[
\int_X F\,d\widetilde{\pi} \le \int_X c\,d\widetilde{\pi} \le \norm{c}_{\infty},
\]
and this implies \cref{st:int_F_bounded_from_above}.

Since $\prj_{ij}(\pi) = \mu_{ij}$, by \cref{st:f_ij_integrability} the function $F \in L^1(X, \pi)$ and
\begin{align*}
\int_X F\,d\pi &= \int_{X_{12}}f_{12}\,d\pi + \int_{X_{13}}f_{13}\,d\pi + \int_{X_{23}}f_{23}\,d\pi\\
&= \int_{X_{12}}f_{12}\,d\mu_{12} + \int_{X_{13}}f_{13}\,d\mu_{13} + \int_{X_{23}}f_{23}\,d\mu_{23} = \int_X c\,d\pi.
\end{align*}
Since in addition $F(x_1, x_2, x_3) \le c(x_1, x_2, x_3)$ for all $x \in X$, we conclude that $F(x_1, x_2, x_3) = c(x_1, x_2, x_3)$ for $\pi$-a.e. $x \in X$. Thus, if $\widetilde{\pi} \ll_B \pi$, then $F(x_1, x_2, x_3) = c(x_1, x_2, x_3)$ $\widetilde{\pi}$-a.e., and therefore
\[
\int_X F\,d\widetilde{\pi} = \int_X c\,d\widetilde{\pi} \ge 0
\]
since $c \ge 0$. This implies \cref{st:int_F_bounded_from_below}.
\end{proof}
\begin{lemma} \label{lem:int_F_bounded_on_plane}
Let $(i, j, k)$ be a permutation of indices $(1, 2, 3)$. Let $\nu_i$ be a probability measure on $X_i$ such that $\nu_i \ll_B \mu_i$. Then $F \in L^1(X, \nu_i \otimes \mu_j \otimes \mu_k)$ and
\[
\int_X F\,d(\nu_i \otimes \mu_j \otimes \mu_k) \ge \int_X F\,d\pi - \norm{c}_{\infty}.
\]
\end{lemma}
\begin{proof}
Since $\nu_i \ll_B \mu_i$, we have $\nu_i \ll_B \prj_i(\pi)$, and therefore the measure $\Up_i(\nu_i, \pi)$ is well-defined. Consider the following measure:
\begin{equation}\label{eq:plane_gamma_repr}
\gamma = \nu_i \otimes \mu_j \otimes \mu_k - \Up_i(\nu_i, \pi) + \mu_i \otimes \prj_{jk}(\Up_i(\nu_i, \pi)) - \pi.
\end{equation}
We claim that all the projections of $\gamma$ to the spaces $X_{ij}$, $X_{ik}$ and $X_{jk}$ are zero measures. First, by \cref{st:up_sup_prj,st:up_decomposition} we have
\begin{align*}
&\prj_{ij}(\Up_i(\nu_i, \pi)) = \Up_i(\nu_i, \prj_{ij}(\pi)) = \Up_i(\nu_i, \mu_i \otimes \mu_j) = \nu_i \otimes \mu_j,\\
&\prj_{ik}(\Up_i(\nu_i, \pi)) = \Up_i(\nu_i, \prj_{ik}(\pi)) = \Up_{i}(\nu_i, \mu_i \otimes \mu_k) = \nu_i \otimes \mu_k.
\end{align*}
Next, we find the projections of $\Up_i(\nu_i, \pi)$ to the spaces $X_j$ and $X_k$:
\begin{align*}
&\prj_{j}(\Up_i(\nu_i, \pi)) = \prj_j \circ \prj_{ij}(\Up_i(\nu_i, \pi)) = \prj_j(\nu_i \otimes \mu_j) = \mu_j,\\
&\prj_{k}(\Up_i(\nu_i, \pi)) = \prj_k \circ \prj_{ik}(\Up_i(\nu_i, \pi)) = \prj_k(\nu_i \otimes \mu_k) = \mu_k.
\end{align*}
Finally, we find the projections of $\gamma$ to the spaces $X_{ij}$, $X_{ik}$ and $X_{jk}$:
\begin{align*}
    \prj_{ij}(\gamma) &= \prj_{ij}(\nu_i \otimes \mu_j \otimes \mu_k) - \prj_{ij}(\Up_i(\nu_i, \pi)) + \prj_{ij}(\mu_i \otimes \prj_{jk}(\Up_i(\nu_i, \pi))) - \prj_{ij}(\pi)\\
    &= \nu_i \otimes \mu_j  - \nu_i \otimes \mu_j + \mu_i \otimes \prj_j(\Up_i(\nu_i, \pi)) - \mu_i \otimes \mu_j \\
    &= \nu_i \otimes \mu_j - \nu_i \otimes \mu_j + \mu_i \otimes \mu_j - \mu_i \otimes \mu_j = 0;\\
    \prj_{ik}(\gamma) &= \prj_{ik}(\nu_i \otimes \mu_j \otimes \mu_k) - \prj_{ik}(\Up_i(\nu_i, \pi)) + \prj_{ik}(\mu_i \otimes \prj_{jk}(\Up_i(\nu_i, \pi))) - \prj_{ik}(\pi)\\
    &= \nu_i \otimes \mu_k - \nu_i \otimes \mu_k + \mu_i \otimes \prj_k(\Up_i(\nu_i, \pi)) - \mu_i \otimes \mu_k\\
    &= \nu_i \otimes \mu_k - \nu_i \otimes \mu_k + \mu_i \otimes \mu_k - \mu_i \otimes \mu_k = 0;\\
    \prj_{jk}(\gamma) &= \prj_{jk}(\nu_i \otimes \mu_j \otimes \mu_k) - \prj_{jk}(\Up_i(\nu_i, \pi)) + \prj_{jk}(\mu_i \otimes \prj_{jk}(\Up_i(\nu_i, \pi))) - \prj_{jk}(\pi)\\
    &= \mu_j \otimes \mu_k - \prj_{jk}(\Up_i(\nu_i, \pi)) + \prj_{jk}(\Up_i(\nu_i, \pi)) - \mu_j \otimes \mu_k = 0.
\end{align*}

Since $\nu_i \ll_B \mu_i$, we have \[
\nu_i \otimes \mu_j \otimes \mu_k \ll_B \mu_i \otimes \mu_j \otimes \mu_k \in \Pi(\mu_{ij}).\]
Next, it follows from \cref{st:up_bounded} that
\[
\Up_i(\nu_i, \pi) \ll_B \pi \in \Pi(\mu_{ij}).
\]
In addition, by \cref{st:llB_prj} we have $\prj_{jk}(\Up_i(\nu_i, \pi)) \ll_B \prj_{jk}(\pi) = \mu_j \otimes \mu_k$, and therefore
\[
\mu_i \otimes \prj_{jk}(\Up_i(\nu_i, \pi)) \ll_B \mu_i \otimes \mu_j \otimes \mu_k \in \Pi(\mu_{ij}).
\]
Thus, it follows from \cref{st:f_ij_integrability} that the function $F$ and the extension of all $f_{12}$, $f_{13}$, and $f_{23}$ to the space $X$ are integrable with respect to all of the summands of equation \cref{eq:plane_gamma_repr}, and therefore that functions are integrable with respect to $\gamma$. In particular,
\[
\int_X F\,d\gamma = \int_{X_{ij}}f_{ij}\,d\prj_{ij}(\gamma) + \int_{X_{ik}}f_{ik}\,d\prj_{ik}(\gamma) + \int_{X_{jk}}f_{jk}\,d\prj_{jk}(\gamma) = 0.
\]

On the other hand, we have
\[
\int_X F\,d\gamma = \int_X F\,d(\nu_i \otimes \mu_j \otimes \mu_k) - \int_X F\,d\Up_i(\nu_i, \pi) + \int_X F\,d(\mu_i \otimes \prj_{jk}(\Up_i(\nu_i, \pi))) - \int_X F\,d\pi.
\]
Since $\Up_i(\nu_i, \pi) \ll_B \pi$, by \cref{st:int_F_bounded_from_below} we have \[\int_X F\,d\Up_i(\nu_i, \pi) \ge 0.\]
By \cref{st:int_F_bounded_from_above} we have \[\int_X F\,d(\mu_i \otimes \prj_{jk}(\Up_i(\nu_i, \pi))) \le \norm{c}_{\infty}.\]
Thus, we get
\[
0 = \int_X F\,d\gamma \le \int_X F\,d(\nu_i \otimes \mu_j \otimes \mu_k) - \int_X F\,d\pi + \norm{c}_{\infty}.
\]
\end{proof}
\begin{lemma}\label{lem:F_bounded_on_3D}
For $1 \le i \le 3$, let $\nu_i$ be a probability measure on $X_i$ such that $\nu_i \ll_B \mu_i$. Then $F \in L^1(X, \nu_1 \otimes \nu_2 \otimes \nu_3)$ and
\[
\int_X F\,d(\nu_1 \otimes \nu_2 \otimes \nu_3) \ge -12\norm{c}_{\infty}.
\]
\end{lemma}
\begin{proof}
The proof is similar to the proof of \cref{lem:int_F_bounded_on_plane}. We have $\nu_i \otimes \nu_j \ll_B \mu_i \otimes \mu_j = \prj_{ij}(\pi)$, and therefore the measure $\Up_{ij}(\nu_i \otimes \nu_j, \pi)$ is well-defined for all $\{i, j\} \in \mathcal{I}_{3, 2}$. Consider the following measures:
\begin{align*}
    \gamma^{(0)} &= \sum_{\{i, j\} \in \mathcal{I}_{3, 2}}\Up_{ij}(\nu_i \otimes \nu_j, \pi);\\
    \gamma^{(1)} &= \sum_{(i, j, k) \in S_3}\mu_i \otimes \prj_{jk}(\Up_{ij}(\nu_i \otimes \nu_j, \pi));\\
    \gamma^{(2)} &= \sum_{\substack{\{i, j\} \in \mathcal{I}_{3,2} \\ \{i, j, k\} = \{1, 2, 3\}}} \mu_i \otimes \mu_j \otimes \prj_{k}(\Up_{ij}(\nu_i \otimes \nu_j, \pi));\\
    \gamma^{(3)} &= \sum_{\substack{\{i, j\} \in \mathcal{I}_{3,2} \\ \{i, j, k\} = \{1, 2, 3\}}} \mu_i \otimes \mu_j \otimes \nu_k; \\
    \gamma &= \nu_1 \otimes \nu_2 \otimes \nu_3 - \gamma^{(0)} + \gamma^{(1)} - \gamma^{(2)} - \gamma^{(3)} + 2\pi.
\end{align*}

We claim that $\prj_{ij}(\gamma) = 0$ for all $\{i, j\} \in \mathcal{I}_{3, 2}$. Let $(i, j, k)$ be a permutation of indices $(1, 2, 3)$. By construction,
\[
\gamma^{(0)} = \Up_{ij}(\nu_i \otimes \nu_j, \pi) + \Up_{ik}(\nu_i \otimes \nu_k, \pi) + \Up_{jk}(\nu_j \otimes \nu_k, \pi).
\]
It follows from \cref{st:up_sup_prj} that
\[\prj_{ij}(\Up_{ij}(\nu_i \otimes \nu_j, \pi)) = \nu_i \otimes \nu_j,
\]
and therefore
\begin{equation}\label{eq:gamma_0_prj}
\prj_{ij}(\gamma^{(0)}) = \nu_i \otimes \nu_j + \prj_{ij}(\Up_{ik}(\nu_i \otimes \nu_k, \pi)) + \prj_{ij}(\Up_{jk}(\nu_j \otimes \nu_k, \pi)).
\end{equation}

Next, let us find the projection of $\gamma^{(1)}$ onto the space $X_{ij}$. The measure $\gamma^{(1)}$ can be written as follows:
\begin{align*}
    \gamma^{(1)} &= \mu_i \otimes \prj_{jk}(\Up_{ij}(\nu_i \otimes \nu_j, \pi)) + \mu_j \otimes \prj_{ik}(\Up_{ij}(\nu_i \otimes \nu_j, \pi))\\
    &+\mu_i \otimes \prj_{jk}(\Up_{ik}(\nu_i \otimes \nu_k, \pi)) + \mu_k \otimes \prj_{ij}(\Up_{ik}(\nu_i \otimes \nu_k, \pi))\\
    &+\mu_j \otimes \prj_{ik}(\Up_{jk}(\nu_j \otimes \nu_k, \pi)) + \mu_k \otimes \prj_{ij}(\Up_{jk}(\nu_j \otimes \nu_k, \pi)).
\end{align*}
It follows from \cref{st:up_sub_prj} that
\begin{align*}
    &\prj_{ij}(\mu_i \otimes \prj_{jk}(\Up_{ij}(\nu_i \otimes \nu_j, \pi))) = \mu_i \otimes \prj_j(\Up_{ij}(\nu_i \otimes \nu_j, \pi)) = \mu_i \otimes \nu_j,\\
    &\prj_{ij}(\mu_j \otimes \prj_{ik}(\Up_{ij}(\nu_i \otimes \nu_j, \pi))) = \prj_{i}(\Up_{ij}(\nu_i \otimes \nu_j, \pi)) \otimes \mu_j = \nu_i \otimes \mu_j,
\intertext{and we trivially have}
    &\prj_{ij}(\mu_i \otimes \prj_{jk}(\Up_{ik}(\nu_i \otimes \nu_k, \pi))) = \mu_i \otimes \prj_{j}(\Up_{ik}(\nu_i \otimes \nu_k, \pi)),\\
    &\prj_{ij}(\mu_k \otimes \prj_{ij}(\Up_{ik}(\nu_i \otimes \nu_k, \pi))) = \prj_{ij}(\Up_{ik}(\nu_i \otimes \nu_k, \pi)),\\
    &\prj_{ij}(\mu_j \otimes \prj_{ik}(\Up_{jk}(\nu_j \otimes \nu_k, \pi))) = \prj_i(\Up_{jk}(\nu_j \otimes \nu_k, \pi)) \otimes \mu_j,\\
    &\prj_{ij}(\mu_k \otimes \prj_{ij}(\Up_{jk}(\nu_j \otimes \nu_k, \pi))) = \prj_{ij}(\Up_{jk}(\nu_j \otimes \nu_k, \pi)).
\end{align*}
Thus, we get
\begin{align}\label{eq:gamma_1_prj}
\begin{split}
    \prj_{ij}(\gamma^{(1)}) &= \mu_i \otimes \nu_j + \nu_i \otimes \mu_j + \prj_{ij}(\Up_{ik}(\nu_i \otimes \nu_k, \pi)) + \prj_{ij}(\Up_{jk}(\nu_j \otimes \nu_k, \pi))\\
    &+\prj_i(\Up_{jk}(\nu_j \otimes \nu_k, \pi)) \otimes \mu_j + \mu_i \otimes \prj_j(\Up_{ik}(\nu_i \otimes \nu_k, \pi))
\end{split}
\end{align}

Finally, by construction
\[
\gamma^{(3)} = \nu_i \otimes \mu_j \otimes \mu_k + \mu_i \otimes \nu_j \otimes \mu_k + \mu_i \otimes \mu_j \otimes \nu_k,
\]
so we get
\begin{equation}\label{eq:gamma_3_prj}
\prj_{ij}(\gamma^{(3)}) = \nu_i \otimes \mu_j + \mu_i \otimes \nu_j + \mu_i \otimes \mu_j.
\end{equation}
Similarly, we conclude that
\begin{equation}\label{eq:gamma_2_prj}
\prj_{ij}(\gamma^{(2)}) = \prj_i(\Up_{jk}(\nu_j \otimes \nu_k, \pi)) \otimes \mu_j + \mu_i \otimes \prj_j(\Up_{ik}(\nu_i \otimes \nu_k, \pi)) + \mu_i \otimes \mu_j.
\end{equation}
Thus, from equations \cref{eq:gamma_0_prj,eq:gamma_1_prj,eq:gamma_2_prj,eq:gamma_3_prj} we get
\begin{align*}
    \prj_{ij}(\gamma) &= \prj_{ij}(\nu_i \otimes \nu_j \otimes \nu_k) - \prj_{ij}(\gamma^{(0)}) + \prj_{ij}(\gamma^{(1)}) - \prj_{ij}(\gamma^{(2)}) - \prj_{ij}(\gamma^{(3)}) + 2\prj_{ij}(\pi) \\
    & = \nu_i \otimes \nu_j - \nu_i \otimes \nu_j - \prj_{ij}(\Up_{ik}(\nu_i \otimes \nu_k, \pi)) - \prj_{ij}(\Up_{jk}(\nu_j \otimes \nu_k, \pi))\\
    &+ \mu_i \otimes \nu_j + \nu_i \otimes \mu_j + \prj_{ij}(\Up_{ik}(\nu_i \otimes \nu_k, \pi)) + \prj_{ij}(\Up_{jk}(\nu_j \otimes \nu_k, \pi))\\
    &+\prj_i(\Up_{jk}(\nu_j \otimes \nu_k, \pi)) \otimes \mu_j + \mu_i \otimes \prj_j(\Up_{ik}(\nu_i \otimes \nu_k, \pi))\\
    &-\prj_i(\Up_{jk}(\nu_j \otimes \nu_k, \pi)) \otimes \mu_j - \mu_i \otimes \prj_j(\Up_{ik}(\nu_i \otimes \nu_k, \pi)) - \mu_i \otimes \mu_j\\
    &-\nu_i \otimes \mu_j - \mu_i \otimes \nu_j - \mu_i \otimes \mu_j+2\mu_i \otimes \mu_j \\
    &= 0.
\end{align*}

Let us verify that the functions $F$ and the extensions of $f_{ij}$ to the space $X$ for all $\{i, j\} \in \mathcal{I}_{3,2}$ are integrable with respect to $\gamma$. First, since $\nu_t \ll_B \mu_t$ for $1 \le t \le 3$, we have
\[
\nu_1 \otimes \nu_2 \otimes \nu_3 \ll_B \mu_1 \otimes \mu_2 \otimes \mu_3 \in \Pi(\mu_{ij}).
\]

Let $(i, j, k)$ be a permutation of indices $(1, 2, 3)$. It follows from \cref{st:up_bounded} that $\Up_{ij}(\nu_i \otimes \nu_j, \pi) \ll_B \pi$, and therefore $\gamma^{(0)} \ll_B \pi$. Next, since $\Up_{ij}(\nu_i \otimes \nu_j, \pi) \ll_B \pi$, it follows from \cref{st:llB_prj} that $\prj_{jk}(\Up_{ij}(\nu_i \otimes \nu_j, \pi)) \ll_B \prj_{jk}(\pi) = \mu_j \otimes \mu_k$ and $\prj_{k}(\Up_{ij}(\nu_i \otimes \nu_j, \pi)) \ll_B \prj_{k}(\pi) = \mu_k$. Hence, $\mu_i \otimes \prj_{jk}(\Up_{ij}(\nu_i \otimes \nu_j, \pi)) \ll_B \mu_1 \otimes \mu_2 \otimes \mu_3$ and $\mu_i \otimes \mu_j \otimes \prj_{k}(\Up_{ij}(\nu_i \otimes \nu_j, \pi)) \ll_B \mu_1 \otimes \mu_2 \otimes \mu_3$, and therefore
\[
\gamma^{(1)} \ll_B \mu_1 \otimes \mu_2 \otimes \mu_3 \in \Pi(\mu_{ij})\; \text{ and }\; \gamma^{(2)} \ll_B \mu_1 \otimes \mu_2 \otimes \mu_3 \in \Pi(\mu_{ij}).
\]
Finally, since $\nu_k \ll_B \mu_k$, we have $\mu_i \otimes \mu_j \otimes \nu_k \ll_B \mu_1 \otimes \mu_2 \otimes \mu_3$, and therefore \[\gamma^{(3)}  \ll_B \mu_1 \otimes \mu_2 \otimes \mu_3 \in \Pi(\mu_{ij}).\] 

Thus, by \cref{st:f_ij_integrability} the function $F$ and the extension of $f_{ij}$ to the space $X$ for all $\{i, j\} \in \mathcal{I}_{3,2}$ are integrable with respect to all summands from the definition of $\gamma$, and therefore that functions are integrable with respect to $\gamma$. In particular,
\[
\int_X F\,d\gamma = \int_{X_{12}}f_{12}\,d\prj_{12}(\gamma) + \int_{X_{13}}f_{13}\,d\prj_{13}(\gamma) + \int_{X_{23}}f_{23}\,d\prj_{23}(\gamma) = 0.
\]

Since $\Up_{ij}(\nu_i \otimes \nu_j, \pi) \ll_B \pi$ for all $\{i, j\} \in \mathcal{I}_{3,2}$, it follows from \cref{st:int_F_bounded_from_below} that
\[
\int_X F\,d\gamma^{(0)} \ge 0.
\]
Applying \cref{st:int_F_bounded_from_above} to all terms of the definition of $\gamma^{(1)}$, we conclude that
\[
\int_X F\,d\gamma^{(1)} \le 6\norm{c}_{\infty}.
\]
Finally, applying \cref{lem:int_F_bounded_on_plane} to all terms of $\gamma^{(2)}$ and $\gamma^{(3)}$, we get
\[
\int_X F\,d\gamma^{(2)} \ge 3\int_X F\,d\pi - 3\norm{c}_{\infty}\; \text{ and }\; \int_X F\,d\gamma^{(3)} \ge 3\int_X F\,d\pi - 3\norm{c}_{\infty}.
\]

Thus, we get the following inequality:
\begin{align*}
\int_X F\,d\gamma &\le \int_X F\,d(\nu_1 \otimes \nu_2 \otimes \nu_3) + 6\norm{c}_{\infty} + 2\left(3\norm{c}_{
\infty} - 3\int_X F\,d\pi\right) + 2\int_X F\,d\pi\\
&= 12\norm{c}_{\infty} - 4\int_X F\,d\pi + \int_X F\,d(\nu_1 \otimes \nu_2 \otimes \nu_3),
\end{align*}
and therefore
\[
\int_X F\,d(\nu_1 \otimes \nu_2 \otimes \nu_3) \ge 4\int_X F\,d\pi - 12\norm{c}_{\infty}.
\]
It follows from \cref{st:int_F_bounded_from_below} that $\int_X F\,d\pi \ge 0$; hence,
\[
\int_X F\,d(\nu_1 \otimes \nu_2 \otimes \nu_3) \ge 4\int_X F\,d\pi - 12\norm{c}_{\infty} \ge -12\norm{c}_\infty.
\]
\end{proof}
\begin{theorem}\label{thm:boundedness_F_in_32}
Let $X_1$, $X_2$, $X_3$ be Polish spaces, let $\mu_i \in \mathcal{P}(X_i)$ for $1 \le i \le 3$, and let $\mu_{ij} = \mu_i \otimes \mu_j$ for all $\{i, j\} \in \mathcal{I}_{3,2}$. Let $c \colon X \to \mathbb{R}_+$ be a bounded continuous cost function. If $\{f_{ij}\}$ is a solution to the related dual problem, then \[
f_{12}(x_1, x_2) + f_{13}(x_1, x_3) + f_{23}(x_2, x_3) \ge -12\norm{c}_{\infty}
\]
for $\mu_1 \otimes \mu_2 \otimes \mu_3$-a.e. points $x \in X$.
\end{theorem}
\begin{proof}
Denote $F(x_1, x_2, x_3) = f_{12}(x_1, x_2) + f_{13}(x_1, x_3) + f_{23}(x_2, x_3)$, and denote $\mu = \mu_1 \otimes \mu_2 \otimes \mu_3$. For $1 \le i \le 3$, let $A_i \in \mathcal{B}_i$ be a measurable subset of $X_i$. If $\mu_i(A_i) = 0$ for some $1 \le i \le 3$, then $\mu(A_1 \times A_2 \times A_3) = 0$, and therefore $\int_{A_1 \times A_2 \times A_3}F\,d\mu = 0$.

Suppose otherwise that $\mu_i(A_i) > 0$ for all $1 \le i \le 3$. Denote $\nu_i = (\mathbbm{1}[A_i] / \mu_{i}(A_i)) \cdot \mu_i$, where $\mathbbm{1}[A]$ is an indicator function of the set $A$. The measure $\nu_i$ is a probability measure and $\nu_i \le (1 / \mu_i(A_i)) \cdot \mu_i$, and therefore $\nu_i \ll_B \mu_i$. By \cref{lem:int_F_bounded_on_plane} we conclude that $\int_X F\,d(\nu_1 \otimes \nu_2 \otimes \nu_3) \ge -12\norm{c}_{\infty}$. By construction,
\[
\int_X F\,d(\nu_1 \otimes \nu_2 \otimes \nu_3) = \frac{\int_{A_1 \times A_2 \times A_3}F\,d\mu}{\mu_1(A_1)\mu_2(A_2)\mu_3(A_3)}.
\]
Thus, we get
\begin{equation}\label{eq:int_F_on_product}
\int_{A_1 \times A_2 \times A_3}F\,d\mu \ge -12\norm{c}_{\infty}\cdot \mu(A_1 \times A_2 \times A_3)\; \text{ for all $A_i \in \mathcal{B}_i$.}
\end{equation}

Consider the measure $(F + 12\norm{c}_{\infty})\cdot \mu$. By equation \cref{eq:int_F_on_product} this measure is non-negative on a semialgebra $\mathcal{A}_0 = \{A_1 \times A_2 \times A_3 \colon A_i \in \mathcal{B}_i\}$, and therefore this measure is non-negative on every element of $\sigma(\mathcal{A}_0)$, and this $\sigma$-algebra coincides with the Borel $\sigma$-algebra on the space $X$. Thus, the measure $(F + 12\norm{c}_{\infty})\cdot \mu$ is non-negative, and therefore $F(x_1, x_2, x_3) + 12\norm{c}_{\infty} \ge 0$ for $\mu$-a.e. points $x \in X$.
\end{proof}

\begin{theorem}
Let $X_1$, $X_2$, $X_3$ be Polish spaces, let $\mu_i \in \mathcal{P}(X_i)$ for $1 \le i \le 3$, and let $\mu_{ij} = \mu_i \otimes \mu_j$ for all $\{i, j\} \in \mathcal{I}_{3,2}$. Let $c \colon X \to \mathbb{R}_+$ be a bounded continuous cost function. Then
\begin{assertions}
    \item\label{st:bounded_solution_of_relaxed_problem} there exists a solution $\{f_{ij}\}$ to the relaxed dual problem such that 
    \[
    -17 \norm{c}_{\infty} \le f_{ij}(x_i, x_j) \le 13\frac{1}{3}\norm{c}_{\infty};
    \]
    \item \label{st:bounded_solution_of_standard_problem} there exists a solution $\{f_{ij}\}$ to the standard dual problem such that
    \[
    -26\frac{2}{3} \norm{c}_{\infty} \le f_{ij}(x_i, x_j) \le 13\frac{1}{3}\norm{c}_{\infty}.
    \]
\end{assertions}
\end{theorem}
\begin{proof}
First, it follows from \cref{thm:existence_multistochastic_dual_solution} that there exists a (real-valued) solution $\{f_{ij}\}$ to the relaxed dual problem. By \cref{thm:boundedness_F_in_32} we conclude that the inequality
\begin{equation}\label{eq:boundedness_of_relaxed_solution}
\norm{c}_\infty  \ge f_{12}(x_1, x_2) + f_{13}(x_1, x_3) + f_{23}(x_2, x_3) \ge -12\norm{c}_\infty
\end{equation}
holds for $\mu_1 \otimes \mu_2 \otimes \mu_3$-almost all points.

Consider a finite $(3, 2)$-function $F(x_1, x_2, x_3) = f_{12}(x_1, x_2) + f_{13}(x_1, x_3) + f_{23}(x_2, x_3)$. Let $A$ be the set of points $(x_1, x_2, x_3) \in X$ such that either $F(x_1, x_2, x_3) < -12\norm{c}_{\infty}$ or $F(x_1, x_2, x_3) > \norm{c}_\infty$. By inequality \cref{eq:boundedness_of_relaxed_solution} we have $\mu_1 \otimes \mu_2 \otimes \mu_3(A) = 0$. Applying \cref{lem:point_with_good_sections} to the indicator function of the set $A$, we conclude that there exists a point $(y_1, y_2, y_3) \in X$ such that for each $\alpha \in \mathcal{I}_3$ the set $A_\alpha = \{x_\alpha \in X_\alpha \colon (x_\alpha, y_{\{1, 2, 3\} \backslash \alpha}) \in A\}$ have a zero measure with respect to $\mu_\alpha$.

For each $\alpha \in \mathcal{I}_3$ consider the function $F_\alpha \colon x_\alpha \mapsto F(x_\alpha y_{\{1, 2, 3\} \backslash \alpha})$. If $x_\alpha \not\in A_\alpha$, then $\norm{c}_\infty \ge F_\alpha(x_\alpha) \ge -12\norm{c}_\infty$, and therefore this inequality holds for $\mu_\alpha$-almost all $x_\alpha \in X_\alpha$. Consider the functions
\begin{align*}
    &\widehat{f}_{12}(x_1, x_2) = F(x_1, x_2, y_3) - \frac{1}{2}F(x_1, y_2, y_3) - \frac{1}{2}F(y_1, x_2, y_3) + \frac{1}{3}F(y_1, y_2, y_3),\\
    &\widehat{f}_{13}(x_1, x_3) = F(x_1, y_2, x_3) - \frac{1}{2}F(x_1, y_2, y_3) - \frac{1}{2}F(y_1, y_2, x_3) + \frac{1}{3}F(y_1, y_2, y_3),\\
    &\widehat{f}_{23}(x_2, x_3) = F(y_1, x_2, x_3) - \frac{1}{2}F(y_1, x_2, y_3) - \frac{1}{2}F(y_1, y_2, x_3) + \frac{1}{3}F(y_1, y_2, y_3).
\end{align*}
By \cref{ex:32_function_decomposition} the equation $F(x_1, x_2, x_3) = \widehat{f}_{12}(x_1, x_2) + \widehat{f}_{13}(x_1, x_3) + \widehat{f}_{23}(x_2, x_3)$ holds for all $(x_1, x_2, x_3) \in X$. In addition, one can easily verify that the inequality \[
-17\norm{c}_\infty \le \widehat{f}_{ij}(x_i, x_j) \le 13\frac{1}{3}\norm{c}_\infty 
\]
holds for $\mu_{ij}$-almost all $(x_i, x_j) \in X_{ij}$.

Thus, there exists a tuple of bounded measurable functions $\{g_{ij}\}$ such that $g_{ij} = \widehat{f}_{ij}$ almost everywhere and
\[
-17\norm{c}_\infty \le g_{ij}(x_i, x_j) \le 13\frac{1}{3}\norm{c}_\infty 
\]
for all $(x_i, x_j) \in X_{ij}$. The inequality
\[
g_{12}(x_1, x_2) + g_{13}(x_1, x_3) + g_{23}(x_2, x_3) = F(x_1, x_2, x_3) \le c(x_1, x_2, x_3)
\]
holds at all points except a zero $(3, 2)$-thickness set, and therefore $\{\widehat{g}_{ij}\} \in \Psi_c(\mu_{ij})$. Finally, we have
\begin{align*}
\int_{X_{12}}g_{12}\,d\mu_{12} + \int_{X_{13}}g_{13}\,d\mu_{13} + \int_{X_{23}}g_{23}\,d\mu_{23} &= \int_X F\,d\mu \\
&= \int_{X_{12}}f_{12}\,d\mu_{12}  + \int_{X_{13}}f_{13}\,d\mu_{13}  + \int_{X_{23}}f_{23}\,d\mu_{23},
\end{align*}
and therefore $\{g_{ij}\}$ is a solution to the relaxed dual problem satisfying \cref{st:bounded_solution_of_relaxed_problem}.

Since $\{g_{ij}\} \in \Psi_c(\mu_{ij})$, there exists a tuple of subsets $Y_{ij} \subset X_{ij}$ such that $\mu_{ij}(Y_{ij}) = 0$ and if $(x_i, x_j) \not\in Y_{ij}$ for all $\{i, j\}$, then
\[
g_{12}(x_1, x_2) + g_{13}(x_1, x_3) + g_{23}(x_2, x_3) \le c(x_1, x_2, x_3).
\]
Consider the tuple of functions $\{\widehat{g}_{ij}\}$: $\widehat{g}_{ij}(x_i, x_j) = g(x_i, x_j)$ if $(x_i, x_j) \not\in Y_{ij}$, and $\widehat{g}_{ij}(x_i, x_j) = -26\frac{2}{3} \|c\|_{\infty}$ otherwise. We have $\widehat{g}_{ij}(x_i, x_j) = g_{ij}(x_i, x_j)$ almost everywhere, and one can easily verify that the inequality
\[
\widehat{g}_{12}(x_1, x_2) + \widehat{g}_{13}(x_1, x_3) + \widehat{g}_{23}(x_2, x_3) \le c(x_1, x_2, x_3)
\]
holds for all points $(x_1, x_2, x_3) \in X$. Thus, $\{\widehat{g}_{ij}\}$ is a solution to the standard dual problem satisfying \cref{st:bounded_solution_of_standard_problem}.
\end{proof}

\subsection{Uniqueness of a continuous dual solution for the cost function \texorpdfstring{$x_1x_2x_3$}{xyz}}

Let us recall to the reader our main example of the multistochastic $(3, 2)$-problem:
\begin{problem}\label{prob:dual_for_xyz}
For $1 \le i \le 3$, let $X_i = [0, 1]$, let $\mu_{ij}$ be the restriction of the Lebesgue measure to the square $[0, 1]^2$, and let $c(x_1, x_2, x_3) = x_1x_2x_3$.

\textbf{\upshape Primal problem.} Find a uniting measure $\pi \in \Pi(\mu_{ij})$ such that
\[
\int x_1x_2x_3\,d\pi \to \min.
\]

\textbf{\upshape Dual problem.} Find a tuple of functions $\{f_{ij}\} \subset L^1([0, 1]^2)$ such that 
\begin{align*}
&\sum_{\{i, j\} \in \mathcal{I}_{3,2}} f_{ij}(x_i, x_j) \le x_1x_2x_3 \;\text{for all $(x_1, x_2, x_3) \in [0, 1]^3$,}\\
&\sum_{\{i, j\} \in \mathcal{I}_{3,2}}\int_0^1\int_0^1 f_{ij}(x_i, x_j)\,dx_idx_j \to \max.
\end{align*}

\end{problem}
In \cite{GlKoZi} the authors describe solutions to this problems. First, we define a binary operator $\oplus$ (called "bitwise exclusive or" or just "xor") on the segment $[0, 1]$. Given $x$ and $y$ on $[0, 1]$, we consider their binary representations $x = \overline{0{,}x_1x_2x_3\dots}_2$, $y = \overline{0{,}y_1y_2y_3\dots}_2$. We agree that every dyadic rational number less then 1 has a finite numbers of units in its decomposition. The number $1$ will be always decomposed as follows: $1 = \overline{0{,}111\dots}_2$. Then we define $x \oplus y = \overline{0{,}x_1\oplus y_1\,x_2\oplus y_2\dots}_2$, where $\oplus$ is an addition in $\mathbb{F}_2$. Using this binary operation, the solutions to the primal problem can be described as follows:
\begin{theorem}[Primal problem solution]\label{thm:primal_solution_for_xyz}
Consider the mapping $T \colon [0, 1]^2 \to [0, 1]^3$, $(x, y) \mapsto (x, y, x \oplus y)$. Denote by $\pi$ the image of the Lebesgue measure restricted to the square $[0, 1]^2$ under the mapping $T$. Then $\pi$ is a solution to primal \cref{prob:dual_for_xyz}.
\end{theorem}

In \cite{GlKoZi} the authors show that $\pi$ is concentrated on the set \[
\{(x, y, z) \in [0, 1]^3 \colon x \oplus y \oplus z = 0\},\] and this set is a self-similar fractal, which is called "Sierpi{\'n}sky tetrahedron". Let us verify  for the completeness of the picture the following description of the support of $\pi$.
\begin{definition}
Denote by $J_n^{a_1, a_2, a_3}$ the image of $[0, 1]^3$ under the mapping
\[(x_1, x_2, x_3) \mapsto \left(\frac{a_1 + x_1}{2^n}, \frac{a_2 + x_2}{2^n}, \frac{a_3 + x_3}{2^n}\right).
\]
Let 
\[
J_n = \bigcup_{\substack{0 \le a_i < 2^n \\ a_1 \oplus a_2 \oplus a_3 = 0}}J_n^{a_1, a_2, a_3},
\]
One can find  images of $J_1$, $J_2$ and $J_3$ on \cref{fig:sets_J}. Denote
\[
S = \bigcap_{n \ge 1}J_n.
\]
The set $S$ is called \textit{Sierpi{\'n}sky tetrahedron}.
\end{definition}
\begin{figure}

\begin{minipage}[b]{0.32\textwidth}
    \centering
    \includegraphics[width=\textwidth]{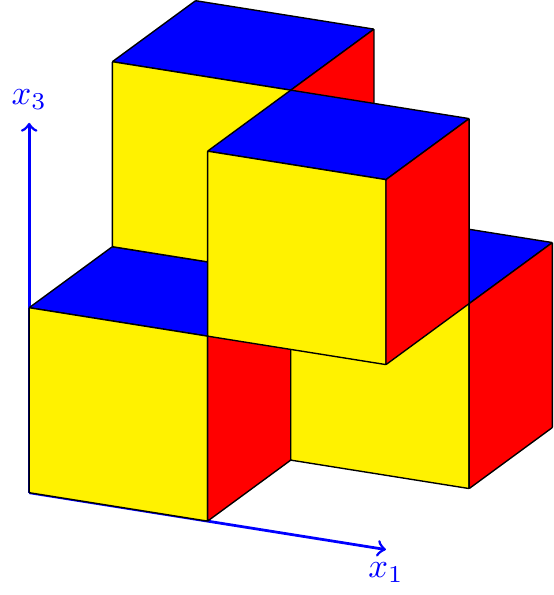}
\end{minipage}
\begin{minipage}[b]{0.32\textwidth}
    \centering
    \includegraphics[width=\textwidth]{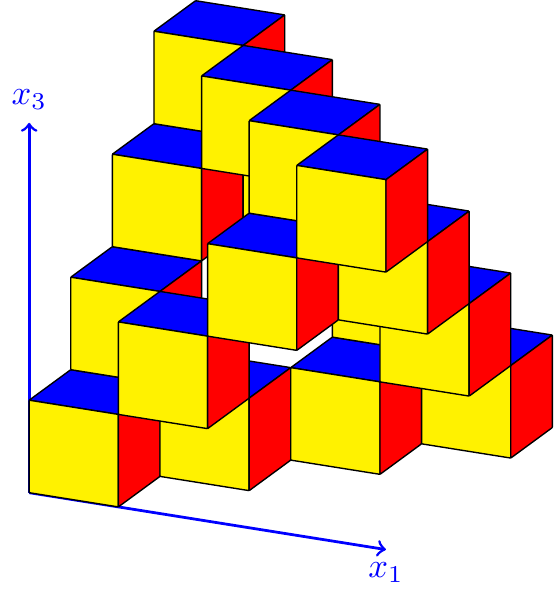}
\end{minipage}
\begin{minipage}[b]{0.32\textwidth}
    \centering
    \includegraphics[width=\textwidth]{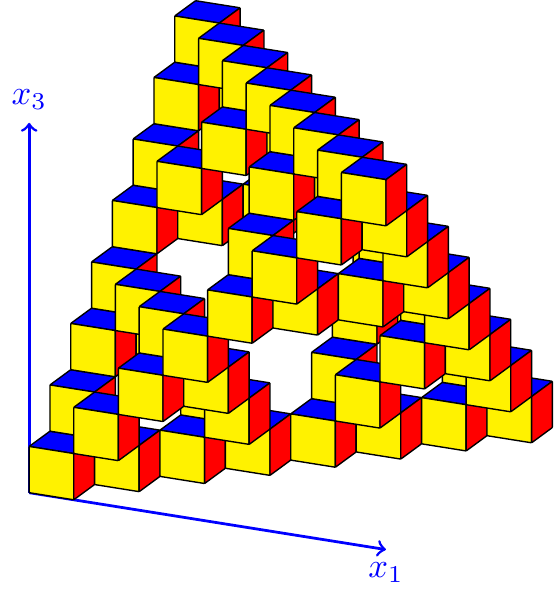}
\end{minipage}

\caption{The sets $J_1$, $J_2$ and $J_3$}
    \label{fig:sets_J}
\end{figure}
\begin{lemma}\label{lem:Jn_binary_representation}
The set $J_n$ contains a point $(x_1, x_2, x_3)$ if and only if there exist binary representations of each coordinates $x_i = \sum_{k = 1}^\infty x_{i,k} / 2^k$ such that $x_{1,k} \oplus x_{2,k} \oplus x_{3,k} = 0$ for all $1 \le k \le n$.
\end{lemma}
\begin{proof}
First, suppose that $(x_1, x_2, x_3) \in J_n$. By construction, there exist integers $a_1, a_2, a_3$ such that $0 \le a_i < 2^n$, bitwise xor of $a_1, a_2$ and $a_3$ is zero, and $(x_1, x_2, x_3) \in J_n^{a_1, a_2, a_3}$. Since
\[
J_n^{a_1, a_2, a_3} = \left[\frac{a_1}{2^n}, \frac{a_1 + 1}{2^n}\right] \times \left[\frac{a_2}{2^n}, \frac{a_2 + 1}{2^n}\right] \times \left[\frac{a_3}{2^n}, \frac{a_3 + 1}{2^n}\right],
\]
we conclude that $x_i = (a_i + y_i) / 2^n$ for all $1 \le i \le 3$, where $0 \le y_i \le 1$. 

Since $a_i < 2^n$, the binary representation of $a_i$ contains at most $n$ digits. Let $\overline{a_{i,1}a_{i,2}\dots a_{i,n}}_2$ be the binary representation of $a_i$ supplemented by zeros up to length $n$. Since $a_1 \oplus a_2 \oplus a_3 = 0$, we have $a_{1, k} \oplus a_{2,k} \oplus a_{3,k} = 0$ for all $1 \le k \le n$. Hence, if $y_i = \sum_{k = 1}^\infty y_{i,k} / 2^k$, then 
\[
x_i = \sum_{k = 1}^n \frac{a_{i,k}}{2^k} + \sum_{k = n + 1}^{\infty}\frac{y_{i, k - n}}{2^k}
\]
provided by $x_i = (a_i + y_i) / 2^n$. This equation provides a binary representation of each coordinates $x_i = \sum_{k = 1}^{\infty}x_{i,k} / 2^k$ such that $x_{1,k} \oplus x_{2,k} \oplus x_{3,k} = 0$ for all $1 \le k \le n$.

Suppose that $(x_1, x_2, x_3)$ is a point on $[0, 1]^3$ and $x_i = \sum_{k = 1}^{\infty}x_{i,k} / 2^k$ for $1 \le i \le 3$, where all $x_{i,k}$ are 0 or 1, and $x_{1,k} \oplus x_{2,k} \oplus x_{3, k} = 0$ for all $1 \le k \le n$. Denote by $a_i$ an integer formed by the first $n$ digits of $x_i$ after radix point. We have $x_i = (a_i + y_i) / 2^n$ for $1 \le i \le 3$, where $0 \le y_i \le 1$, and therefore $(x_1, x_2, x_3) \in J_n^{a_1, a_2, a_3}$. In addition, $0 \le a_i < 2^n$, and since $x_{1, k} \oplus x_{2, k} \oplus x_{3, k} = 0$ for all $1 \le k \le n$, we conclude that $a_1 \oplus a_2 \oplus a_3 = 0$. Thus, $(x_1, x_2, x_3) \in J_n^{a_1, a_2, a_3 } \subset J_n$.
\end{proof}
Using that, we can describe all points of the Sierpi{\'n}sky tetrahedron in terms of their binary representations.
\begin{proposition} \label{prop:S_binary_representation}
The Sierpi{\'n}sky tetrahedron $S$ contains a point $(x_1, x_2, x_3)$ if and only if there exist binary representations of each coordinates $x_i = \sum_{k = 1}^{\infty}x_{i,k}/2^k$ such that \[
    x_{1,k} \oplus x_{2,k} \oplus x_{3,k} = 0 \text{ for all $k$};
    \]
\end{proposition}
\begin{proof}
Suppose that $(x_1, x_2, x_3)$ is a point on $[0, 1]^3$ and $x_i = \sum_{k = 1}^{\infty}x_{i,k} / 2^k$ for $1 \le i \le 3$, where all $x_{i,k}$ are 0 or 1, and $x_{1,k} \oplus x_{2,k} \oplus x_{3, k} = 0$ for all $k$. Then it follows from \cref{lem:Jn_binary_representation} that $(x_1, x_2, x_3)$ is contained in $J_n$ for all $n$. Thus, \[
(x_1, x_2, x_3) \in \bigcap_{n \ge 1}J_n = S.\]

Suppose that $(x_1, x_2, x_3) \in S$. Then $(x_1, x_2, x_3) \in J_n$ for all $n$, and therefore there exist binary representations of each coordinates $x_i = \sum_{k = 1}^\infty x_{i,k}^n / 2^k$ such that $x_{1, k}^n \oplus x_{2,k}^n \oplus x_{2,k}^n = 0$ for all $1 \le k \le n$. For any nonnegative real number, there are at most two binary representations of this number, and therefore there exist at most eight tuples of binary representations of the point $(x_1, x_2, x_3)$. Hence, there exists at least one of them $x_i = \sum_{k = 1}^\infty x_{i, k}/2^k$ such that the property $x_{1, k} \oplus x_{2, k} \oplus x_{3, k} = 0$ for all $1 \le k \le n$ holds for an infinite number of $n$. Thus, $x_{1, k} \oplus x_{2, k} \oplus x_{3, k} = 0$ for all $k$.


\end{proof}

\begin{proposition}
The Sierpi{\'n}sky tetrahedron $S$ has the following properties:
\begin{assertions}
    \item \label{st:S_closed} the set $S$ a closed subset of $[0, 1]^3$;
    \item \label{st:S_contains_xor} a point $(x, y, x \oplus y)$ is contained in $S$ for all $x, y \in [0, 1]$;
    \item \label{st:S_fractal} if $S^{a_1, a_2, a_3}_n$ is the image of $S$ under a mapping \[(x_1, x_2, x_3) \mapsto \left(\frac{a_1 + x_1}{2^n}, \frac{a_2 + x_2}{2^n}, \frac{a_3 + x_3}{2^n}\right),\] then
    
    \[
    S = \bigcup_{\substack{0 \le a_i < 2^n, \\ a_1 \oplus a_2 \oplus a_3 = 0}}S^{a_1, a_2, a_3}_n
    \]
\end{assertions}
\end{proposition}
\begin{proof}
The set $J_n$ is closed since $J_n$ is a finite union of closed sets. Thus, since $S$ is an intersection of the closed sets $J_n$, we conclude that $S$ is closed too, and this implies \cref{st:S_closed}.

\Cref{st:S_contains_xor} trivially holds by \cref{prop:S_binary_representation}.

Let us verify \cref{st:S_fractal}. Suppose that $(x_1, x_2, x_3) \in S$. By \cref{prop:S_binary_representation}, there exist binary representations $x_i = \sum_{k = 1}^nx_{i,k} / 2^k$ such that $x_{1, k} \oplus x_{2, k} \oplus x_{3, k} = 0$ for all k. Denote by $a_i$ an integer formed by the first $n$ digits of $x_i$ after radix point. We have $0 \le a_i < 2^n$, and since $x_{1, k} \oplus x_{2, k} \oplus x_{3, k} = 0$ for all $k$, we conclude that $a_1 \oplus a_2 \oplus a_3 = 0$. In addition, $x_i = (a_i + y_i) / 2^n$, where $y_i = \sum_{k = 1}^n x_{i,n + k} / 2^k$. By \cref{prop:S_binary_representation} $(y_1, y_2, y_3) \in S$, and therefore $(x_1, x_2, x_3) \in S^{a_1, a_2, a_3}_n$. Thus
\[
S \subseteq \bigcup_{\substack{0 \le a_i < 2^n, \\ a_1 \oplus a_2 \oplus a_3 = 0}}S^{a_1, a_2, a_3}_n
\]

Suppose that $(x_1, x_2, x_3) \in S_n^{a_1, a_2, a_3}$, where $0 \le a_i < 2^n$ and $a_1 \oplus a_2 \oplus a_3 = 0$. Since $a_i < 2^n$, the binary representation of $a_i$ contains at most $n$ digits. Let $\overline{a_{i,1}a_{i,2}\dots a_{i,n}}_2$ be the binary representation of $a_i$ supplemented by zeros up to length $n$. Since $a_1 \oplus a_2 \oplus a_3 = 0$, we have $a_{1, k} \oplus a_{2,k} \oplus a_{3,k} = 0$ for all $1 \le k \le n$. 

By construction, there exists a point $(y_1, y_2, y_3) \in S$ such that $x_i = (a_i + y_i) / 2^n$. By \cref{prop:S_binary_representation}, there exist binary representations $y_i = \sum_{k = 1}^\infty y_{i,k} / 2^k$ such that $y_{1,k} \oplus y_{2,k} \oplus y_{3,k} = 0$. Hence,
\[
x_i = \frac{a_i + y_i}{2^n} = \sum_{k = 1}^n\frac{a_{i,k}}{2^k} + \sum_{k = n + 1}^{\infty}\frac{y_{i, k - n}}{2^k},
\]
and therefore by \cref{prop:S_binary_representation} $(x_1, x_2, x_3) \in S$. Thus,
\[
S \supseteq \bigcup_{\substack{0 \le a_i < 2^n, \\ a_1 \oplus a_2 \oplus a_3 = 0}}S^{a_1, a_2, a_3}_n,
\]
and this completes the proof of \cref{st:S_fractal}.
\end{proof}

Following the proof of the main result in \cite{GlKoZi} the reader can extract  the following statement:
\begin{theorem}\label{thm:optimal_xyz_on_Jn}
For $1 \le i \le 3$, let $X_i = [0, 1]$, let $\mu_{ij}$ be the Lebesgue measure restricted to the square $[0, 1]^2$, and let $c(x_1, x_2, x_3) = x_1x_2x_3$. If the measure $\pi$ is uniting for $\{\mu_{ij}\}$ and $\supp(\pi) \nsubseteq J_n$ for some $n$, then there exists a measure $\widetilde{\pi} \in \Pi(\mu_{ij})$ such that
\[
\int_{[0, 1]^3}x_1x_2x_3\,\widetilde{\pi}(dx_1, dx_2, dx_3) < \int_{[0, 1]^3}x_1x_2x_3\,\pi(dx_1, dx_2, dx_3).
\]
\end{theorem}

If $\pi$ is a solution to primal \cref{prob:dual_for_xyz}, then it follows from \cref{thm:optimal_xyz_on_Jn} that $\supp(\pi) \subseteq J_n$ for all $n$. Hence, $\supp(\pi) \subseteq \cap_{n \ge 1} J_n$, and this implies the following proposition.
\begin{proposition}\label{prop:optimal_xyz_on_J}
If $\pi$ is a solution to primal \cref{prob:dual_for_xyz}, then $\supp(\pi) \subseteq S$, where $S$ is the Sierpi{\'n}sky tetrahedron.
\end{proposition}

Using that, let us prove that there exists a unique solution to primal \cref{prob:dual_for_xyz}.
\begin{lemma}\label{lem:optimal_pi_at_most_one}
There exists at most one measure $\pi$ on $[0, 1]^3$ such that $\supp(\pi) \subseteq S$ and $\prj_{12}(\pi)$ coincides with the Lebesgue measure $\mu_{12}$ on the square $[0, 1]^2$. 
\end{lemma}
\begin{proof}
Let $\Gamma = \{(x, y, x \oplus y) \colon (x, y) \in [0, 1]^2\}$. It follows from \cref{st:S_contains_xor} that $\Gamma \subseteq S$. Consider the set $S_b = S \backslash \Gamma$, and consider a point $(x_1, x_2, x_3) \in S_b$. Suppose that both points $x_1$ and $x_2$ are not dyadic rationals. If $x$ is not a dyadic rational, then there exists a unique binary representation of $x$. Hence, it follows from \cref{prop:S_binary_representation} that there exists at most one $z \in [0, 1]$ such that $(x_1, x_2, z) \in S$. By \cref{st:S_contains_xor} we have $(x_1, x_2, x_1 \oplus x_2) \in S$, and therefore $x_3 = x_1 \oplus x_2$. Thus, $(x_1, x_2, x_3) \in \Gamma$, and this contradicts the point selection.

This contradiction proves that if $(x_1, x_2, x_3) \in S_b$, then at least one of $x_1$ and $x_2$ is a dyadic rational. Hence, $\mu_{12}(\prj_{12}(S_b)) = 0$, and therefore $\pi(S_b) = 0$ provided by $\prj_{12}(\pi) = \mu_{12}$. Thus, since $\supp(\pi) \subseteq S$, we get $\pi(\Gamma) = 1$.

Let $A$ be a measurable subset of $[0, 1]^3$. Since $\pi(\Gamma) = 1$, we have $\pi(A \backslash \Gamma) = 0$, and therefore
\begin{equation}\label{eq:measure_A_wrt_pi}
\pi(A) = \pi(A \cap \Gamma).
\end{equation}
Denote $A_\Gamma = A \cap \Gamma$. The set $A_\Gamma$ is a measurable subset of $\Gamma$. Since for each $(x_1, x_2) \in [0, 1]^2$ there exists exactly one $x_3$ such that $(x_1, x_2, x_3) \in \Gamma$, we get
\[
A_\Gamma = (\prj_{12}(A_\Gamma) \times X_3) \cap \Gamma.
\]
Applying equation \cref{eq:measure_A_wrt_pi} to the set $\prj_{12}(A_\Gamma) \times X_3$, we get
\[
\pi((\prj_{12}(A_\Gamma) \times X_3) \cap \Gamma) = \pi(\prj_{12}(A_\Gamma) \times X_3) = \mu_{12}(\prj_{12}(A_\Gamma))
\]
provided by $\prj_{12}(\pi) = \mu_{12}$. From all equations above we get
\[
\pi(A) = \pi(A_\Gamma) = \pi((\prj_{12}(A_\Gamma) \times X_3) \cap \Gamma) = \mu_{12}(\prj_{12}(A_\Gamma)).
\]

Thus, the measure of the set $A$ with respect to $\pi$ is independent on $\pi$, and therefore there exists at most one measure $\pi$ such that $\supp(\pi) \subseteq S$ and $\prj_{12}(\pi) = \mu_{12}$.
\end{proof}
\begin{theorem}
There exists a unique solution $\pi$ to primal \cref{prob:dual_for_xyz}.
\end{theorem}
\begin{proof}
If $\pi$ is a solution to the problem, then $\prj_{12}(\pi) = \mu_{12}$, and it follows from \cref{prop:optimal_xyz_on_J} that $\supp(\pi) \subseteq S$. By \cref{lem:optimal_pi_at_most_one}, there exists at most one measure $\pi$ with that properties. Thus, there exists at most one solution to primal \cref{prob:dual_for_xyz}.

The existence of a solution follows from \cref{thm:primal_solution_exists}.
\end{proof}

Finally, let us find exactly the support of the solution to primal \cref{prob:dual_for_xyz}.
\begin{proposition}\label{prop:supp_pi}
If $\pi$ is the solution to primal \cref{prob:dual_for_xyz}, then $\supp(\pi) = S$.
\end{proposition}
\begin{proof}
It follows from \cref{prop:optimal_xyz_on_J} that $\supp(\pi) \subseteq S \subset J_n$ for all $n$, and therefore $\pi(J_n) = 1$. By definition of $J_n$,
\[
J_n = \bigcup_{\substack{0 \le a_i < 2^n, \\ a_1 \oplus a_2 \oplus a_3 = 0}}J_n^{a_1, a_2, a_3}.
\]
We have \[
\prj_{12}(J_n^{a_1, a_2, a_3}) = \left[\frac{a_1}{2^n}, \frac{a_1 + 1}{2^n}\right] \times \left[\frac{a_2}{2^n}, \frac{a_2 + 1}{2^n}\right]. \]
For each pair $a_1, a_2$ such that $0 \le a_1, a_2 < 2^n$ there exists a unique $a_3$ such that $0 \le a_3 < 2^n$ and $a_1 \oplus a_2 \oplus a_3 = 0$. Hence, projections to $X_1 \times X_2$ of all components of $J_n$ overlapping by the sets of measure zero with respect to $\mu_{12}$, and therefore
\begin{equation}\label{eq:pi_J_decomposition_eq}
\pi(J_n^{a_1, a_2, a_3}) = \mu_{12}(\prj_{12}(J_n^{a_1, a_2, a_3})) = \frac{1}{4^n}\; \text{ if $a_1 \oplus a_2 \oplus a_3 = 0$.}
\end{equation}

Suppose that $\supp(\pi) \ne S$. Since $\supp(\pi)$ is closed, there exist a point $x_0 \in S$ and a non-negative integer $n$ such that if $|x - x_0| < 2^{1-n}$, then $x$ is not contained in $\supp(\pi)$. Since $x_0 
\in S \subset J_n$, there exist integers $a_1, a_2, a_3$ such that $0 \le a_1, a_2, a_3 < 2^n$, bitwise xor of $a_1, a_2, a_3$ is zero, and $x_0 \in J^{a_1, a_2, a_3}_n$. We have
\[
J^{a_1, a_2, a_3}_n = \left[\frac{a_1}{2^n}, \frac{a_1 + 1}{2^n}\right] \times \left[\frac{a_2}{2^n}, \frac{a_2 + 1}{2^n}\right] \times \left[\frac{a_3}{2^n}, \frac{a_3 + 1}{2^n}\right];
\]
hence, $\mathrm{diam}(J^{a_1, a_2, a_3}_n) < 2^{1 - n}$, and therefore $\supp(\pi) \cap J^{a_1, a_2, a_3}_n = \varnothing$. This contradicts equation \cref{eq:pi_J_decomposition_eq}. 
\end{proof}
In \cite{GlKoZi} the authors also found a solution to the dual \cref{prob:dual_for_xyz}.
\begin{theorem}[Dual problem solution]
Denote
\[
f(x, y) = \int_0^x\int_0^y s\oplus t\,dsdt - \frac{1}{4}\int_0^x\int_0^x s\oplus t\,dsdt - \frac{1}{4}\int_0^y\int_0^y s\oplus t\,dsdt.
\]
Then the tuple of functions $f_{ij}\colon (x_i, x_j) \mapsto f(x_i, x_j)$ is a solution to dual \cref{prob:dual_for_xyz}.
\end{theorem}
This solution to the dual problem is not unique. First, for $1 \le i \le 3$ let $f_i$ be an integrable function on the segment $[0, 1]$. Consider the following functions
\begin{align*}
    &\widehat{f}_{12}(x_1, x_2) = f_{12}(x_1, x_2) + f_1(x_1) - f_2(x_2),\\
    &\widehat{f}_{23}(x_2, x_3) = f_{23}(x_2, x_3) + f_2(x_2) - f_3(x_3),\\
    &\widehat{f}_{13}(x_1, x_3) = f_{13}(x_1, x_3) + f_3(x_3) - f_1(x_1).
\end{align*}
Clearly
\begin{align*}
    &\sum_{\{i, j\} \in \mathcal{I}_{32}}\widehat{f}_{ij}(x_i, x_j) = \sum_{\{i, j\} \in \mathcal{I}_{32}}f_{ij}(x_i, x_j)\;\text{ for all $(x_1, x_2, x_3) \in [0, 1]^3$}\\
    \intertext{and}
    &\sum_{\{i, j\} \in \mathcal{I}_{32}}\int_0^1\int_0^1 \widehat{f}_{ij}(x_i, x_j)\,dx_idx_j = \sum_{\{i, j\} \in \mathcal{I}_{32}}\int_0^1\int_0^1 f_{ij}(x_i, x_j)\,dx_idx_j,
\end{align*}
and therefore the functions $\{\widehat{f}_{ij}
\}$ are also the solution to the dual problem.

In what follows, we prove that there is no other continuous solutions to the related dual problem.
\begin{lemma}\label{lem:F_estimated_on_Jn}
If a tuple of functions $\{f_{ij}\}$ is a solution to dual \cref{prob:dual_for_xyz}, function $f_{ij}$ is continuous for all $\{i, j\} \in \mathcal{I}_{3,2}$, and $a_1$, $a_2$ and $a_3$ are non-negative integers such that $0 \le a_1, a_2, a_3 < 2^n$ and $a_1 \oplus a_2 \oplus a_3 = 0$, then 
\[
    \left|F(x_1, x_2, x_3) - x_1x_2x_3\right| \le \frac{13}{2^{3n}} \text{ for all $(x_1, x_2, x_3) \in J_n^{a_1, a_2, a_3}$,}
\]
where
\begin{align*}
&F(x_1, x_2, x_3) = f_{12}(x_1, x_2) + f_{13}(x_1, x_3) + f_{23}(x_2, x_3)
\intertext{and}
&J_n^{a_1, a_2, a_3} = \left[\frac{a_1}{2^n}, \frac{a_1 + 1}{2^n}\right] \times \left[\frac{a_2}{2^n}, \frac{a_2 + 1}{2^n}\right] \times \left[\frac{a_3}{2^n}, \frac{a_3 + 1}{2^n}\right].
\end{align*}
\end{lemma}
\begin{proof}
Since $\{f_{ij}\}$ is a solution to the dual problem, we have
\begin{equation}\label{eq:F_le_xyz}
F(x_1, x_2, x_3) \le x_1x_2x_2 \text{ for all $(x_1, x_2, x_3) \in [0, 1]^3$.}
\end{equation}
Let $\pi$ be the solution to primal \cref{prob:dual_for_xyz}. We have $\int F\,d\pi = \int x_1x_2x_2\,d\pi$, and therefore \[
 F(x_1, x_2, x_3) = x_1x_2x_3 \text{ for $\pi$-a.e. $(x_1, x_2, x_3) \in [0, 1]^3$.}
\]
The function $F(x_1, x_2, x_3) - x_1x_2x_3$ is continuous; hence, the equation holds for all $(x_1, x_2, x_3) \in \supp(\pi)$. By \cref{prop:supp_pi}, the support of $\pi$ coincides with the Sierpi{\'n}sky tetrahedron $S$, and therefore 
\begin{equation}\label{eq:F_eq_xyz_on_S}
F(x_1, x_2, x_3) = x_1x_2x_3 \text{ for all $(x_1, x_2, x_3) \in S$.}
\end{equation}

Consider the following functions:
\begin{align*}
    &\widehat{f}_{12}(x_1, x_2) = 2^{3n}f_{12}\left(\frac{a_1 + x_1}{2^n}, \frac{a_2 + x_2}{2^n}\right) - \frac{a_1a_2a_3}{3} - \frac{a_2a_3x_1 + a_1a_3x_2}{2} - a_3x_1x_2,\\
    &\widehat{f}_{13}(x_1, x_3) = 2^{3n}f_{13}\left(\frac{a_1 + x_1}{2^n}, \frac{a_3 + x_3}{2^n}\right) - \frac{a_1a_2a_3}{3} - \frac{a_2a_3x_1 + a_1a_2x_3}{2} - a_2x_1x_3,\\
    &\widehat{f}_{23}(x_2, x_3) = 2^{3n}f_{23}\left(\frac{a_2 + x_2}{2^n}, \frac{a_3 + x_3}{2^n}\right) - \frac{a_1a_2a_3}{3} - \frac{a_1a_3x_2 + a_1a_2x_3}{2} - a_1x_2x_3,
\end{align*}
where $0 \le x_i \le 1$ for $1 \le i \le 3$. We claim that $\{\widehat{f}_{ij}\}$ is a solution to the dual problem.

First, one can easily verify that
\begin{multline}\label{eq:sum_widehat_f}
\widehat{f}_{12}(x_1, x_2) + \widehat{f}_{13}(x_1, x_3) + \widehat{f}_{23}(x_2, x_3) =\\ 2^{3n}\left[F\left(\frac{a_1 + x_1}{2^n}, \frac{a_2 + x_2}{2^n}, \frac{a_3 + x_3}{2^n}\right) - \frac{a_1 + x_1}{2^n} \cdot \frac{a_2 + x_2}{2^n} \cdot \frac{a_3 + x_3}{2^n}\right] + x_1x_2x_3.
\end{multline}
Using inequality \cref{eq:F_le_xyz}, we conclude that
\begin{equation}\label{eq:widehat_f_le_xyz}
\widehat{f}_{12}(x_1, x_2) + \widehat{f}_{13}(x_1, x_3) + \widehat{f}_{23}(x_2, x_3) \le x_1x_2x_3 \text{ for all $(x_1, x_2, x_3) \in [0, 1]^3$.}
\end{equation}
If $(x_1, x_2, x_3) \in S$, then
\[
\left(\frac{a_1 + x_1}{2^n}, \frac{a_2 + x_2}{2^n}, \frac{a_3 + x_3}{2^n}\right) \in S_n^{a_1, a_2, a_3}.
\]
By \cref{st:S_fractal}, $S^{a_1, a_2, a_3}_n \subset S$; hence, if $(x_1, x_2, x_3) \in S$, then by \cref{eq:F_eq_xyz_on_S} we get
\[
F\left(\frac{a_1 + x_1}{2^n}, \frac{a_2 + x_2}{2^n}, \frac{a_3 + x_3}{2^n}\right) = \frac{a_1 + x_1}{2^n} \cdot \frac{a_2 + x_2}{2^n} \cdot \frac{a_3 + x_3}{2^n},
\]
and therefore
\[
\widehat{f}_{12}(x_1, x_2) + \widehat{f}_{13}(x_1, x_3) + \widehat{f}_{23}(x_2, x_3) = x_1x_2x_3 \text{ for all $(x_1, x_2, x_3) \in S$.}
\]
Since $\supp(\pi) = S$, we have
\begin{multline}\label{eq:widehat_f_sum_int}
\int_{[0, 1]^2}\widehat{f}_{12}(x_1, x_2)\,dx_1dx_2 + \int_{[0, 1]^2}\widehat{f}_{13}(x_1, x_3)\,dx_1dx_3 + \int_{[0, 1]^2}\widehat{f}_{23}(x_2, x_3)\,dx_2dx_3 \\
= \int_{[0, 1]^3}\left(\widehat{f}_{12}(x_1, x_2) + \widehat{f}_{13}(x_1, x_3) + \widehat{f}_{23}(x_2, x_3)\right)\,d\pi = \int_{[0, 1]^3} x_1x_2x_3\,d\pi.
\end{multline}
By equations \cref{eq:widehat_f_le_xyz,,eq:widehat_f_sum_int} we conclude that $\{\widehat{f}_{ij}\}$ is a solution to dual \cref{prob:dual_for_xyz}.

The cost function $x_1x_2x_3$ is non-negative and $\mu_{ij} = \mu_i \otimes \mu_j$ for all $\{i, j\} \in \mathcal{I}_{3,2}$. Thus, we are under  assumptions of \cref{thm:boundedness_F_in_32}. We have $\norm{x_1x_2x_3}_{\infty} = 1$, where $0 \le x_i \le 1$ for all $1 \le i \le 3$, and therefore
\[
-12 \le \widehat{f}_{12}(x_1, x_2) + \widehat{f}_{13}(x_1, x_3) + \widehat{f}_{23}(x_2, x_3) \le x_1x_2x_3 \le 1
\]
for almost all $(x_1, x_2, x_3) \in [0, 1]^3$. Since all $\widehat{f}_{ij}$ are continuous, we conclude that inequalities holds for all points, and therefore
\[
\left|\widehat{f}_{12}(x_1, x_2) + \widehat{f}_{13}(x_1, x_3) + \widehat{f}_{23}(x_2, x_3)\right| \le 12 \text{ for all $(x_1, x_2, x_3) \in [0, 1]^3$.}
\]
Using equation \cref{eq:sum_widehat_f}, we conclude that
\[
\left|F\left(\frac{a_1 + x_1}{2^n}, \frac{a_2 + x_2}{2^n}, \frac{a_3 + x_3}{2^n}\right) - \frac{a_1 + x_1}{2^n} \cdot \frac{a_2 + x_2}{2^n} \cdot \frac{a_3 + x_3}{2^n}\right| \le \frac{12 + x_1x_2x_3}{2^{3n}} \le \frac{13}{2^{3n}}
\]
for all $(x_1, x_2, x_3) \in [0, 1]^3$, and therefore 
\[
    \left|F(x_1, x_2, x_3) - x_1x_2x_3\right| \le \frac{13}{2^{3n}} \text{ for all $(x_1, x_2, x_3) \in J_n^{a_1, a_2, a_3}$.}\]
\end{proof}
\begin{lemma}\label{lem:f_change_sign_sum_est}
Let $\{f_{ij}\}$ be a solution to the dual \cref{prob:dual_for_xyz}. If $\{i, j\} \in \mathcal{I}_{3, 2}$, a number $n$ is a positive integer, numbers $a_i$ and $a_j$ are non-negative integers such that $0 \le a_i, a_j < 2^n$, and $(x_i, x_j)$ and $(y_i, y_j)$ are arbitrary points in the square 
\[
\left[\frac{a_i}{2^n}, \frac{a_i + 1}{2^n}\right] \times \left[\frac{a_j}{2^n}, \frac{a_j + 1}{2^n}\right],
\]
then
\[
\left|f_{ij}(x_i, x_j) - f_{ij}(y_i, x_j) - f_{ij}(x_i, y_j) + f_{ij}(y_i, y_j) - \int_{x_i}^{y_i}\int_{x_j}^{y_j}s \oplus t\,dsdt\right| \le \frac{54}{2^{3n}}.
\]
\end{lemma}

Without loss of generality it can be assumed that $\{i, j\} = \{1, 2\}$. Let $a_3 = a_1 \oplus a_2$, and let $(x_1, x_2, x_3)$ and $(y_1, y_2, y_3)$ be arbitrary points of the cube $J_n^{a_1, a_2, a_3}$.
We have
\begin{multline}\label{eq:F_change_sign_sum}
F(x_1, x_2, x_3) - F(y_1, x_2, x_3) - F(x_1, y_2, x_3) + F(y_1, y_2, x_3) \\= f_{12}(x_1, x_2) - f_{12}(y_1, x_2) - f_{12}(x_1, y_2) + f_{12}(y_1, y_2).
\end{multline}
In addition,
\begin{equation}\label{eq:xyz_change_sign_sum}
x_1x_2x_3 - y_1x_2x_3 - x_1y_2x_3 + y_1y_2x_3 = x_3(x_1 - y_1)(x_2 - y_2).
\end{equation}
On the other hand, it follows from \cref{lem:F_estimated_on_Jn} that
\begin{align*}
    |&F(x_1, x_2, x_3) + F(y_1, y_2, x_3) - x_1x_2x_3 - y_1y_2x_3 \\
    &-F(y_1, x_2, x_3) - F(x_1, y_2, x_3) + y_1x_2x_3 + x_1y_2x_3| \le 4 \cdot \frac{13}{2^{3n}} = \frac{52}{2^{3n}}.
\end{align*}
Thus, taking into account equations \cref{eq:F_change_sign_sum,,eq:xyz_change_sign_sum}, we get
\begin{equation}\label{eq:f_change_sign_sum_with_x3}
\left|f_{12}(x_1, x_2) - f_{12}(y_1, x_2) - f_{12}(x_1, y_2) + f_{12}(y_1, y_2) - x_3(x_1 - y_1)(x_2 - y_2)\right| \le \frac{52}{2^{3n}}.
\end{equation}

Since $(x_1, x_2, x_3) \in J_n^{a_1, a_2, a_3}$, we have $|a_3 / 2^n - x_3| \le 2^{-n}$. Since $(y_1, y_2, y_3) \in J_n^{a_1, a_2, a_3}$, we also have $|x_1 - y_1| \le 2^{-n}$ and $|x_2 - y_2| \le 2^{-n}$. Thus,
\begin{equation}\label{eq:x3(x1-y1)(x2-y2)_est}
    \left|x_3(x_1 - y_1)(x_2 - y_2) - \frac{a_3}{2^n}(x_1 - y_1)(x_2 - y_2)\right| = \left|\frac{a_3}{2^n} - x_3\right|\cdot|x_1 - y_1| \cdot |x_2 - y_2| \le \frac{1}{2^{3n}}.
\end{equation}
Next, let $t$ be a point on the interval $(a_1 / 2^n, (a_1 + 1) / 2^n)$, and let $s$ be a point on an interval $(a_2 / 2^n, (a_2 + 1) / 2^n)$. One can easily verify that
\[
\frac{a_1 \oplus a_2}{2^n} \le s \oplus t \le \frac{(a_1 \oplus a_2) + 1}{2^n},
\]
and therefore, since $a_1 \oplus a_2 = a_3$, we get
\begin{equation}\label{eq:int_s_xor_t_est}
    \left| \int_{x_1}^{y_1}\int_{x_2}^{y_2}s \oplus t\,dsdt - \frac{a_3}{2^n}(x_1 - y_1)(x_2 - y_2)\right| \le \frac{1}{2^n}\cdot|x_1 - y_1|\cdot|x_2 - y_2| \le \frac{1}{2^{3n}}.
\end{equation}
Summarizing inequalities \cref{eq:f_change_sign_sum_with_x3,,eq:x3(x1-y1)(x2-y2)_est,,eq:int_s_xor_t_est}, we conclude that
\begin{equation*}
    \left|f_{12}(x_1, x_2) - f_{12}(y_1, x_2) - f_{12}(x_1, y_2) + f_{12}(y_1, y_2) - \int_{x_1}^{y_1}\int_{x_2}^{y_2}s \oplus t\,dsdt\right| \le \frac{54}{2^{3n}}.
\end{equation*}

\begin{lemma}\label{lem:f_ij_equals_integral}
If a tuple of functions $\{f_{ij}\}$ is a solution to dual \cref{prob:dual_for_xyz} and $f_{ij}$ is continuous for all $\{i, j\} \in \mathcal{I}_{3,2}$, then 
\[
f_{ij}(x_i, x_j) - f_{ij}(x_i, 0) - f_{ij}(0, x_j) + f_{ij}(0, 0) = \int_0^{x_i}\int_0^{x_j}s \oplus t\,dsdt
\]
for all $(x_i, x_j) \in [0, 1]^3$.
\end{lemma}
\begin{proof}
Let $\{u_k\}_{k = 0}^N$ and $\{v_l\}_{l = 0}^M$ be arbitrary points on the segment $[0, 1]$. One can easily verify that
\begin{align}\label{eq:f_sum_on_squares}
\begin{split}
    \sum_{k = 1}^N\sum_{l = 1}^M&\left(f_{ij}(u_k, v_l) - f_{ij}(u_{k - 1}, v_l) - f_{ij}(u_k, v_{l - 1}) + f_{ij}(u_{k - 1}, v_{l - 1})\right)\\
     &=f_{ij}(u_N, v_M) - f_{ij}(u_0, v_M) - f_{ij}(u_N, v_0) + f_{ij}(u_0, v_0)
\end{split}
\end{align}
and
\begin{equation}\label{eq:int_xor_sum_on_squares}
    \sum_{k = 1}^N\sum_{l = 1}^M\int_{u_{k - 1}}^{u_k}\int_{v_{l - 1}}^{v_l}s\oplus t\,dsdt = \int_{u_0}^{u_N}\int_{v_0}^{v_M}s \oplus t\,dsdt.
\end{equation}

Let $(x_i, x_j)$ be an arbitrary point on the square $[0, 1]^2$. Let $N = \lceil 2^nx_i \rceil$, and let $M = \lceil 2^nx_j\rceil$. Finally, let $u_k = k / 2^n$ for all $0 \le k < N$ and $u_N = x_i$, and similarly let $v_l = l / 2^n$ for all $0 \le l < M$ and $v_M = x_j$. By construction, both points $(u_{k - 1}, v_{l - 1})$ and $(u_{k}, v_{l})$ belong to the square
\[
\left[\frac{k - 1}{2^n}, \frac{k}{2^n}\right] \times \left[\frac{l - 1}{2^n}, \frac{l}{2^n}\right],
\]
and therefore by \cref{lem:f_change_sign_sum_est} we have
\begin{equation*}
    \left|f_{ij}(u_k, v_l) - f_{ij}(u_{k - 1}, v_l) - f_{ij}(u_k, v_{l - 1}) + f_{ij}(u_{k - 1}, v_{l - 1}) - \int_{u_{k - 1}}^{u_k}\int_{v_{l - 1}}^{v_l}s \oplus t\,dsdt\right| \le \frac{54}{2^{3n}}
\end{equation*}
for all $1 \le k \le N$ and for all $1 \le l \le M$. 

Taking into account equations \cref{eq:f_sum_on_squares,,eq:int_xor_sum_on_squares}, we conclude that
\begin{align*}
&\left|f_{ij}(x_i, x_j) - f_{ij}(x_i, 0) - f_{ij}(0, x_j) + f_{ij}(0, 0) - \int_0^{x_i}\int_0^{x_j}s \oplus t\,dsdt\right| \\
    &\le\begin{aligned}[t] \sum_{k = 1}^N\sum_{l = 1}^M\Big|f_{ij}(u_k, v_l) - f_{ij}(u_{k - 1}, v_l) &- f_{ij}(u_k, v_{l - 1})\\&+ f_{ij}(u_{k - 1}, v_{l - 1}) 
- \int_{u_{k - 1}}^{u_k}\int_{v_{l - 1}}^{v_l}s \oplus t\,dsdt\Big|
\end{aligned}\\
&\le\sum_{k = 1}^N\sum_{l = 1}^M \frac{54}{2^{3n}} = \frac{54\cdot N \cdot M}{2^{3n}}.
\end{align*}
Thus, since $N, M \le 2^n$, we get
\[
\left|f_{ij}(x_i, x_j) - f_{ij}(x_i, 0) - f_{ij}(0, x_j) + f_{ij}(0, 0) - \int_0^{x_i}\int_0^{x_j}s \oplus t\,dsdt\right| \le \frac{54}{2^n}
\]
for all $(x_i, x_j) \in [0, 1]^2$ and for every positive integer $n$, and therefore
\[
f_{ij}(x_i, x_j) - f_{ij}(x_i, 0) - f_{ij}(0, x_j) + f_{ij}(0, 0) = \int_0^{x_i}\int_0^{x_j}s \oplus t\,dsdt.
\]
\end{proof}
\begin{theorem}
If a tuple of functions $\{f_{ij}\}$ is a solution to \cref{prob:dual_for_xyz} and $f_{ij}$ is continuous for all $\{i, j\} \in \mathcal{I}_{3,2}$, then there exist continuous functions $f_i \colon [0, 1] \to \mathbb{R}$, $1 \le i \le 3$, such that 
\begin{align*}
&f_{12}(x_1, x_2) = f(x_1, x_2) + f_1(x_1) - f_2(x_2), \\
&f_{23}(x_2, x_3) = f(x_2, x_3) + f_2(x_2) - f_3(x_3),\\ 
\intertext{and}
&f_{13}(x_1, x_3) = f(x_1, x_3) + f_3(x_3) - f_1(x_1),
\end{align*}
where
\[
f(x, y) = \int_0^x\int_0^y s\oplus t\,dsdt - \frac{1}{4}\int_0^x\int_0^x s\oplus t\,dsdt - \frac{1}{4}\int_0^y\int_0^y s\oplus t\,dsdt.
\]
\end{theorem}
\begin{proof}
First, consider the function
\[
F(x_1, x_2, x_3) = f_{12}(x_1, x_2) + f_{13}(x_1, x_3) + f_{23}(x_2, x_3).
\]
It follows from equation \cref{eq:F_eq_xyz_on_S} that
\[
F(x_1, x_2, x_3) = x_1x_2x_3 \text{ for all $(x_1, x_2, x_3) \in S$.}
\]
By \cref{st:S_contains_xor}, all the points $(0, x, x)$, $(x, 0, x)$ and $(x, x, 0)$ are contained in $S$, and therefore
\begin{equation}\label{eq:F_on_diagonal_0}
F(0, x, x) = F(x, 0, x) = F(x, x, 0) = 0 \text{ for all $x \in [0, 1]$.}
\end{equation}
In particular, taking $x = 0$, we conclude that
\begin{equation}\label{eq:sum_zeros_f}
f_{12}(0, 0) + f_{13}(0, 0) + f_{23}(0, 0) = F(0, 0, 0) = 0.
\end{equation}

Denote $\widehat{f}_{ij}(x_i, x_i) = f_{ij}(x_i, x_j) - f_{ij}(0, 0)$. We have $\widehat{f}_{ij}(0, 0) = 0$; it follows from \cref{eq:sum_zeros_f} that
\[
F(x_1, x_2, x_3) = \widehat{f}_{12}(x_1, x_2) + \widehat{f}_{13}(x_1, x_3) + \widehat{f}_{23}(x_2, x_3).
\]
By \cref{lem:f_ij_equals_integral} we have
\begin{equation}\label{eq:widehat_f_ij_from_lemma}
\widehat{f}_{ij}(x_i, x_j) = \int_{0}^{x_i}\int_0^{x_j}s \oplus t\, dsdt + \widehat{f}_{ij}(x_i, 0) + \widehat{f}_{ij}(0, x_j),
\end{equation}
and therefore
\begin{align}\label{eq:F_sum_with_int_and_phi}
\begin{split}
F(x_1, x_2, x_3) = &\int_{0}^{x_1}\int_{0}^{x_2}s \oplus t\,dsdt + \int_{0}^{x_1}\int_{0}^{x_3}s \oplus t\,dsdt + \int_{0}^{x_2}\int_{0}^{x_3}s \oplus t\,dsdt \\
&+ \varphi_1(x_1) + \varphi_2(x_2) + \varphi_3(x_3),
\end{split}
\end{align}
where
\begin{align}\label{eq:varphi_def}
\begin{split}
    &\varphi_1(x_1) = \widehat{f}_{12}(x_1, 0) + \widehat{f}_{13}(x_1, 0),\\
    &\varphi_2(x_2) = \widehat{f}_{12}(0, x_2) + \widehat{f}_{23}(x_2, 0),\\
    &\varphi_3(x_3) = \widehat{f}_{13}(0, x_3) + \widehat{f}_{23}(0, x_3).
\end{split}
\end{align}

Since $\widehat{f}_{i, j}(0, 0) = 0$ for all $\{i, j \} \in \mathcal{I}_{3, 2}$, we have $\varphi_i(0) = 0$ for all $1 \le i \le 3$. Hence, using equations \cref{eq:F_on_diagonal_0,,eq:F_sum_with_int_and_phi} we get
\begin{align*}
    &0 = F(0, x, x) = \int_0^x\int_0^x s \oplus t\,dsdt + \varphi_2(x) + \varphi_3(x),\\
    &0 = F(x, 0, x) = \int_0^x\int_0^x s \oplus t\,dsdt + \varphi_1(x) + \varphi_3(x),\\
    &0 = F(x, x, 0) = \int_0^x\int_0^x s \oplus t\,dsdt + \varphi_1(x) + \varphi_2(x)
\end{align*}
for all $x \in [0, 1]$. Thus, we obtain
\begin{equation}\label{eq:varphi_explicit_value}
\varphi_i(x) = -\frac{1}{2}\int_0^x\int_0^x s \oplus t\,dsdt
\end{equation}
for all $x \in [0, 1]$ for $1 \le i \le 3$.

Consider the functions $f_i(x_i)$, $1 \le i \le 3$, satisfying the following equations:
\begin{align}\label{eq:f_ij_cycle1_repr}
\begin{split}
    &\widehat{f}_{12}(x_1, 0) = f_1(x_1) -\frac{1}{4}\int_0^{x_1}\int_0^{x_1} s \oplus t\,dsdt,\\
    &\widehat{f}_{23}(x_2, 0) = f_2(x_2) -\frac{1}{4}\int_0^{x_2}\int_0^{x_2} s \oplus t\,dsdt,\\
    &\widehat{f}_{13}(0, x_3) = f_3(x_3) -\frac{1}{4}\int_0^{x_3}\int_0^{x_3} s \oplus t\,dsdt.
\end{split}
\end{align}
The function $f_i$ is continuous for $1 \le i \le 3$. Combining equations \cref{eq:varphi_def,eq:varphi_explicit_value} we get
\[
\widehat{f}_{12}(0, x_2) = \varphi_2(x_2) - \widehat{f}_{23}(x_2, 0) = -\frac{1}{2}\int_0^{x_2}\int_0^{x_2} s \oplus t\,dsd - \widehat{f}_{23}(x_2, 0),
\]
and using the representation of $\widehat{f}_{23}$ from equation \cref{eq:f_ij_cycle1_repr} we get
\begin{equation}\label{eq:f_ij_cycle2_repr}
\widehat{f}_{12}(0, x_2) = -f_2(x_2) - \frac{1}{4}\int_{0}^{x_2}\int_{0}^{x_2} s \oplus t\,dsdt.
\end{equation}
Substituting equations \cref{eq:f_ij_cycle1_repr,eq:f_ij_cycle2_repr} into \cref{eq:widehat_f_ij_from_lemma} we obtain the following relation:
\begin{align*}
&\widehat{f}_{12}(x_1, x_2) = \int_0^{x_1}\int_0^{x_2} s \oplus t\,dsdt + \widehat{f}_{12}(x_1, 0) + \widehat{f}_{12}(0, x_2)\\
&=\int_0^{x_1}\int_0^{x_2} s \oplus t\,dsdt - \frac{1}{4}\int_0^{x_1}\int_0^{x_1} s \oplus t\,dsdt - \frac{1}{4}\int_0^{x_2}\int_0^{x_2} s \oplus t\,dsdt + f_1(x_1) - f_2(x_2)\\
&= f(x_1, x_2) + f_1(x_1) - f_2(x_2).
\end{align*}
Similarly, we conclude that $\widehat{f}_{23}(x_2, x_3) = f(x_2, x_3) + f_2(x_2) - f_3(x_3)$ and $\widehat{f}_{13}(x_1, x_3) = f(x_1, x_3) + f_3(x_3) - f_1(x_1)$.

Finally, since $f_{12}(0, 0) + f_{13}(0, 0) + f_{23}(0, 0) = 0$, there exist real numbers $C_1$, $C_2$ and $C_3$ such that $f_{12}(0, 0) = C_1 - C_2$, $f_{23}(0, 0) = C_2 - C_3$ and $f_{13}(0, 0) = C_3 - C_1$. Thus, 
\begin{align*}
&f_{12}(x_1, x_2) = \widehat{f}_{12}(x_1, x_2) + f_{12}(0, 0) = f(x_1, x_2) + (f_1(x_1) + C_1) - (f_2(x_2) + C_2),\\
&f_{23}(x_2, x_3) = \widehat{f}_{23}(x_2, x_3) + f_{23}(0, 0) = f(x_2, x_3) + (f_2(x_2) + C_2) - (f_3(x_3) + C_3),\\
&f_{11}(x_1, x_3) = \widehat{f}_{13}(x_1, x_3) + f_{13}(0, 0) = f(x_1, x_3) + (f_3(x_3) + C_3) - (f_1(x_1) + C_1).
\end{align*}
\end{proof}

\subsection{Example of a discontinuous solution to a dual problem}

It is known that any dual multimarginal problem admits a regular
 solution provided the cost function is regular. For instance, applying the Legendre-type transformation, the reader can easily verify that for a Lipschitz cost functions there exists a Lipschitz dual solution. 
In this section we prove that a natural solution to the dual $(3,2)$-problem can be even discontinuous and (in a sense) unique.

Consider the following $(3, 2)$-problem.

\begin{problem}\label{prob:discontinuous}
For $1 \le i \le 3$, let $X_i = [0, 1]$, let $\mu_{ij}$ be the restriction of the Lebesgue measure onto the square $[0, 1]^2$, and let $c(x_1, x_2, x_3) = \max(0, x_1 + x_2 + 3x_3 - 3)$.

\textbf{\upshape Primal problem.} Find a uniting measure $\pi \in \Pi(\mu_{ij})$ such that
\[
\int c(x_1, x_2, x_3)\,d\pi \to \min.
\]

\textbf{\upshape Dual problem.} Find a tuple of functions $\{f_{ij}\} \subset L^1([0, 1]^2)$ such that 
\begin{align*}
&\sum_{\{i, j\} \in \mathcal{I}_{3,2}} f_{ij}(x_i, x_j) \le c(x_1, x_2, x_3) \;\text{for all $(x_1, x_2, x_3) \in [0, 1]^3$,}\\
&\sum_{\{i, j\} \in \mathcal{I}_{3,2}}\int_0^1\int_0^1 f_{ij}(x_i, x_j)\,dx_idx_j \to \max.
\end{align*}
\end{problem}

The cost function $c(x_1, x_2, x_3) = \max(0, x_1 + x_2 + 3x_3 - 3)$ is Lipschitz continuous, and the tuple of measures $\{\mu_{ij}\}$ is redicible; hence, there is no duality gap, and solutions to both primal and dual problems exist.

\begin{proposition}\label{prop:discontinuous_F_is_compatible}
Let
\begin{align*}
    &f_{12}(x_1, x_2) = 0 \text{ for all points $(x_1, x_2) \in [0, 1]^2$};\\
    &f_{13}(x_1, x_3) = \begin{cases}
    0, &\text{if $x_3 < \frac{2}{3}$},\\
    x_1 + \frac{3}{2}x_3 - \frac{3}{2}, &\text{if $x_3 \ge \frac{2}{3}$};
    \end{cases}\\
    &f_{23}(x_2, x_3) = \begin{cases}
    0, &\text{if $x_3 < \frac{2}{3}$},\\
    x_2 + \frac{3}{2}x_3 - \frac{3}{2}, &\text{if $x_3 \ge \frac{2}{3}$}.
    \end{cases}
\end{align*}
Denote $F(x_1, x_2, x_3) = f_{12}(x_1, x_2) + f_{13}(x_1, x_3) + f_{23}(x_2, x_3)$. Then 
\begin{assertions}
    \item \label{st:discontinuous_F_le_c}$F(x_1, x_2, x_3) \le c(x_1, x_2, x_3)$ for all $(x_1, x_2, x_3) \in [0, 1]^3$;
    \item \label{st:F_eq_c_pi113_ae} if the value of $x_1 + x_2 + 3x_3$ is integer and $(x_1, x_2, x_3) \ne (0, 0, 2/3)$, then $F(x_1, x_2, x_3) = c(x_1, x_2, x_3)$.
\end{assertions}
\end{proposition}
\begin{proof}
First, one can easily verify the following representation for the function $F$:
\begin{align}\label{eq:discontinuous_F_repr}
    F(x_1, x_2, x_3) = \begin{cases}
    0, &\text{if $x_3 < \frac{2}{3}$},\\
    x_1 + x_2 + 3x_3 - 3, &\text{if $x_3 \ge \frac{2}{3}$}.
    \end{cases}
\end{align}
Thus, $F(x_1, x_2, x_3) \le \max(0, x_1 + x_2 + 3x_3 - 3) = c(x_1, x_2, x_3)$ for all $(x_1, x_2, x_3) \in [0, 1]^3$, and this implies \cref{st:discontinuous_F_le_c}.

Suppose that the value of $x_1 + x_2 + 3x_3$ is integer. Consider the case $x_3 < 2/3$. Equation \cref{eq:discontinuous_F_repr} implies that $F(x_1, x_2, x_3) = 0$. Since $x_1, x_2 \le 1$, we have $x_1 + x_2 + 3x_3 < 4$, and therefore $x_1 + x_2 + 3x_3 \le 3$. Thus, $c(x_1, x_2, x_3) = \max(x_1 + x_2 + 3x_3 - 3, 0) = 0 = F(x_1, x_2, x_3)$.

Consider the case $x_3 \ge 2/3$. By equation \cref{eq:discontinuous_F_repr}, $F(x_1, x_2, x_3) = x_1 + x_2 + 3x_3 - 3$. If $(x_1, x_2, x_3) \ne (0, 0, 2/3)$, then $x_1 + x_2 + 3x_3 > 2$, and therefore, since $x_1 + x_2 + 3x_3$ is integer, $x_1 + x_2 + 3x_3 \ge 3$. Thus, if $(x_1, x_2, x_3) \ne (0, 0, 2/3)$, then $c(x_1, x_2, x_3) = x_1 + x_2 + 3x_3 - 3 = F(x_1, x_2, x_3)$, and this implies \cref{st:F_eq_c_pi113_ae}.
\end{proof}

We claim that the constructed tuple of functions $\{f_{ij}\}$ is a solution to the dual \cref{prob:discontinuous}. By \cref{prop:discontinuous_F_is_compatible} it is enough to find a measure $\pi \in \Pi(\mu_{ij})$ such that $\pi$ is concentrated on the set $\{(x_1, x_2, x_3) \colon \mathrm{frac}(x_1 + x_2 + 3x_3) = 0\}$.
The proof of the following lemma is easy and is left to the reader.

\begin{lemma}
There exists a measure $\pi_{1, 1, 1}$ concentrated on the set \[\{(x_1, x_2, x_3) \colon \mathrm{frac}(x_1 + x_2 + x_3) = 0\}\] 
such that $\prj_{ij}(\pi)$ 
coincides  with the Lebesgue measure restricted to the square $[0, 1]^2$ for all $\{i, j\} \in \mathcal{I}_{3, 2}$. 
\end{lemma}
Using this lemma, we prove a more general statement.
\begin{proposition}\label{prop:fraction_part_measure}
Assume we are given positive integers $a_1$, $a_2$ and $a_3$. Then there exists a measure $\pi_{a_1, a_2, a_3} \in \Pi(\mu_{ij})$ concentrated on the set \[\{(x_1, x_2, x_3) \colon \mathrm{frac}(a_1x_1 + a_2x_2 + a_3x_3) = 0\}.\]
\end{proposition}
\begin{proof}
Let $t_1$, $t_2$ and $t_3$ be non-negative integers such that $0 \le t_i < a_i$ for $1 \le i \le 3$. Consider the mapping
\[
T \colon (x_1, x_2, x_3) \mapsto \left(\frac{x_1 + t_1}{a_1}, \frac{x_2 + t_2}{a_2}, \frac{x_3 + t_3}{a_3}\right). 
\]
Let $\pi_{a_1, a_2, a_3}^{t_1, t_2, t_3}$ be the image of the measure $\pi_{1, 1, 1}$ under the mapping $T$. First, if $(y_1, y_2, y_3) = T(x_1, x_2, x_3)$, then $a_1y_1 + a_2y_2 + a_3y_3 = (x_1 + x_2 + x_3) + (t_1 + t_2 + t_3)$. Hence, \[\mathrm{frac}(x_1 + x_2 + x_3) = \mathrm{frac}(a_1y_1 + a_2y_2 + a_3y_3),\]
and therefore, since $\pi_{1, 1, 1}$ is concentrated on the set $\{(x_1, x_2, x_3) \colon \mathrm{frac}(x_1 + x_2 + x_3) = 0\}$, the measure $\pi_{a_1, a_2, a_3}^{t_1, t_2, t_3}$ is concentrated on the set $\{(y_1, y_2, y_3) \colon \mathrm{frac}(a_1y_1 + a_2y_2 + a_3y_3)=0\}$.

In addition, for all $\{i, j\} \in \mathcal{I}_{3, 2}$ the measure $\prj_{ij}(\pi_{a_1, a_2, a_3}^{t_1, t_2, t_3})$ is the image of $\prj_{ij}(\pi_{1, 1, 1})$ under the mapping
\[
(x_i, x_j) \mapsto \left(\frac{x_i + t_i}{a_i}, \frac{x_j + t_j}{a_j}\right).
\]
Thus, $\prj_{ij}(\pi_{a_1, a_2, a_3}^{t_1, t_2, t_3})$ is proportional to the Lebesgue measure restricted to the square
\begin{equation}\label{eq:pi_a1a2a3_projection}
\left[\frac{t_i}{a_i}, \frac{t_i + 1}{a_i}\right] \times \left[\frac{t_j}{a_j}, \frac{t_j + 1}{a_j}\right].
\end{equation}

Let 
\[
\pi_{a_1, a_2, a_3} = \frac{1}{a_1a_2a_3}\sum_{0 \le t_i < a_i}\pi_{a_1, a_2, a_3}^{t_1, t_2, t_3}.
\]
The measure $\pi_{a_1, a_2, a_3}$ is a probability measure concentrated on the set \[\{(y_1, y_2, y_3) \colon \mathrm{frac}(a_1y_1 + a_2y_2 + a_3y_3)=0\}.\]
In addition, it follows from \cref{eq:pi_a1a2a3_projection} that $\prj_{ij}(\pi_{a_1, a_2, a_3})$ is the Lebesgue measure restricted to the square $[0, 1]^2$.
\end{proof}

Using this proposition, we immediately obtain the following theorem.
\begin{theorem}
\label{dualdiscontinuous}
The tuple of functions $\{f_{ij}\}$ described in \cref{prop:discontinuous_F_is_compatible} is a solution to the dual \cref{prob:discontinuous}, and the measure $\pi_{1, 1, 3} \in \Pi(\mu_{ij})$, concentrated on the set \[\{(x_1, x_2, x_3) \colon \mathrm{frac}(x_1 + x_2 + 3x_3)=0\},\] is a solution to the primal \cref{prob:discontinuous}.
\end{theorem}
Unlike \cref{prob:dual_for_xyz}, a solution to the primal \cref{prob:discontinuous} is non-unique. 
\begin{proposition}\label{prop:improved_primal_solution_of_discontinuous}
Let $\pi_1$ be the restriction of the Lebesgue measure to the set $\{(x_1, x_2, x_3) \colon 0 \le x_1, x_2 \le 1, 0 \le x_3 \le 1/3\}$, and let $\pi_2$ be the image of the measure $\widehat{\pi}_{1, 1, 2}$ described in \cref{prop:fraction_part_measure} under the mapping
\[
T \colon (x_1, x_2, x_3) \mapsto \left(x_1, x_2, \frac{2}{3}x_3 + \frac{1}{3}\right).
\]

Then the measure $\pi = \pi_1 + \frac{2}{3}\widehat{\pi}_{1,1,2} $ is uniting and the function $F(x_1, x_2, x_3)$ described in \cref{prop:discontinuous_F_is_compatible}  satisfies: $F(x_1,x_2,x_3)=c(x_1, x_2, x_3)$ $\pi$-a.e. Consequently, the measure $\pi$ is a solution to the primal \cref{prob:discontinuous} (see \cref{fig:fracsolution}).
\end{proposition}
\begin{figure}
    \centering
    \includegraphics[width=0.45\textwidth]{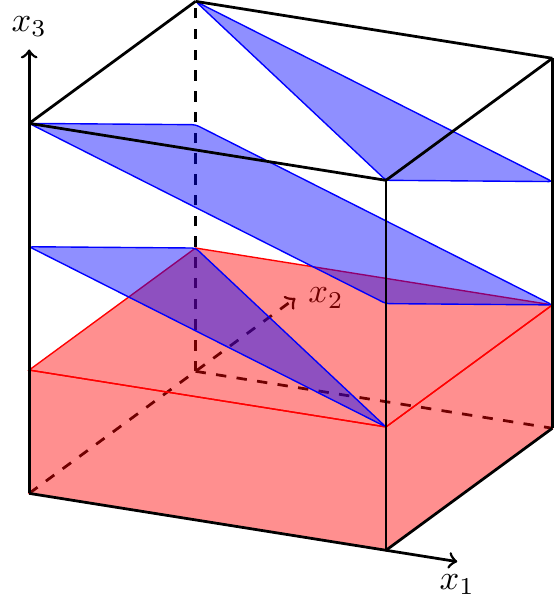}
    \caption{The support of the solution $\pi$ described in \cref{prop:improved_primal_solution_of_discontinuous}. The support of the measure $\pi_1$ is red, and support of $\widehat{\pi}_{1, 1, 2}$ is blue.}
\label{fig:fracsolution}
\end{figure}
\begin{proof}
By construction, $\prj_{12}(\pi_1)$ is proportional to the restriction of the Lebesgue measure to the square $[0, 1]^2$. The mapping $T$ does not change the projection of a measure onto the space $X_{12}$, and therefore $\prj_{12}(\widehat{\pi}_{1, 1, 2})$ is also proportional to the restriction of the Lebesgue measure to the square $[0, 1]^2$. Thus, $\prj_{12}(\pi) = \mu_{12}$.

The measure $\prj_{13}(\pi_1)$ coincides with the restriction of the Lebesgue measure to the rectangle $\{(x_1, x_3) \colon 0 \le x_1 \le 1, 0 \le x_3 \le 1/3\}$. The measure $\prj_{13}(\widehat{\pi}_{1,1,2})$ is the image of $\prj_{13}(\pi_{1, 1, 2})$ under the mapping
\[
(x_1, x_3) \mapsto \left(x_1, \frac{2}{3}x_3 + \frac{1}{3}\right).
\]
Thus, $\frac{2}{3}\prj_{13}(\widehat{\pi}_{1,1,2})$ coincides with the restriction of the Lebesgue measure to the rectangle $\{(x_1, x_3) \colon 0 \le x_1 \le 1, 1/3 \le x_3 \le 1\}$, and therefore $\prj_{13}(\pi) = \mu_{13}$. Similarly, $\prj_{23}(\pi) = \mu_{23}$, and we conclude that $\pi \in \Pi(\mu_{ij})$.

Let $(x_1, x_2, x_3)$ be a point in $[0, 1]^3$ such that $x_3 \le 1/3$. By equation \cref{eq:discontinuous_F_repr} we have $F(x_1, x_2, x_3) = 0$. In addition, $x_1 + x_2 + 3x_3 - 3 \le 0$, and therefore $c(x_1, x_2, x_3) = 0$. Thus, since $\supp(\pi_1) = \{(x_1, x_2, x_3) \in [0, 1]^3 \colon 0 \le x_3 \le 1/3\}$, we conclude that $F(x_1, x_2, x_3) = c(x_1, x_2, x_3)$ $\pi_1$-a.e.

Let $(x_1, x_2, x_3)$ be an arbitrary point in the cube $[0, 1]^2$, and let $(y_1, y_2, y_3) = T(x_1, x_2, x_3)$. We have $y_1 + y_2 + 3y_3 = x_1 + x_2 + 2x_3 + 1$, and therefore \[
\mathrm{frac}(y_1 + y_2 + 3y_3) = \mathrm{frac}(x_1 + x_2 + 2x_3).
\]
Hence, we conclude that $\widehat{\pi}_{1, 1, 2}$ is concentrated on the set $\{(x_1, x_2, x_3) \colon\mathrm{frac}(x_1 + x_2 + 3x_3) = 0\}$, and therefore by \cref{st:F_eq_c_pi113_ae} $F(x_1, x_2, x_3) = c(x_1, x_2, x_3)$ $\widehat{\pi}_{1, 1, 2}$-a.e.

Thus, $F(x_1, x_2, x_3) = c(x_1, x_2, x_3)$ for $\pi$-almost all points $(x_1, x_2, x_3) \in [0, 1]^3$, and the measure $\pi$ is a solution to the primal \cref{prob:discontinuous}.
\end{proof}

Unlike the primal problem, the dual problem admits a unique solution in the following sense.

\begin{proposition}\label{prop:dual_discontinuous_solution_unique}
Let $\{g_{ij}\}$ be a solution to the relaxed dual \cref{prob:discontinuous}. Then the equation
\[
g_{12}(x_1, x_2) + g_{13}(x_1, x_3) + g_{23}(x_2, x_3) = f_{12}(x_1, x_2) + f_{13}(x_1, x_3) + f_{23}(x_2, x_3)
\]
holds for almost all $(x_1, x_2, x_3) \in [0, 1]^3$, where the tuple of functions $\{f_{ij}\}$ is defined in \cref{prop:discontinuous_F_is_compatible}.
\end{proposition}

First, let us verify the following statement.
\begin{lemma}\label{lem:g_12_is_zero}
Let $\{g_{ij}\}$ be a solution to the relaxed dual \cref{prob:discontinuous}. Then there exist integrable functions $\varphi_1$ and $\varphi_2$ such that $g_{12}(x_1, x_2) = \varphi_1(x_1) + \varphi_2(x_2)$ almost everywhere.
\end{lemma}
\begin{proof}
Consider the finite $(3, 2)$-function \[
G(x_1, x_2, x_3) = g_{12}(x_1, x_2) + g_{13}(x_1, x_3) + g_{23}(x_2, x_3).\]
Since $\{g_{ij}\}$ is a solution to the relaxed dual problem, the equation $G(x_1, x_2, x_3) = c(x_1, x_2, x_3)$ holds $\pi$-almost everywhere, where $\pi$ is a solution to the primal problem defined in \cref{prop:improved_primal_solution_of_discontinuous}. In particular, $G(x_1, x_2, x_3) = c(x_1, x_2, x_3)$ for almost all points $(x_1, x_2, x_3) \in [0, 1]^3$ such that $0 \le x_3 \le 1/3$. Since $c(x_1, x_2, x_3) = \max(x_1 + x_2 + 3x_3 - 3, 0) = 0$ if $x_3 \le 1/3$, we conclude that $G(x_1, x_2, x_3) = 0$ for almost all $(x_1, x_2, x_3) \in [0, 1]^3$ such that $0 \le x_3 \le 1/3$.

In particular, there exists a point $0 \le x_3^{(0)} \le 1/3$ such that the equation $G(x_1, x_2, x_3^{(0)}) = 0$ holds for almost all $(x_1, x_2) \in [0, 1]^2$. Hence, if we denote $\varphi_1(x_1) = -g_{13}(x_1, x_3^{(0)})$ and $\varphi_2(x_2) = -g_{23}(x_2, x_3^{(0)})$, then the equation
\[
g_{12}(x_1, x_2) = -g_{13}(x_1, x_3^{(0)}) -g_{23}(x_2, x_3^{(0)}) = \varphi_1(x_1) + \varphi_2(x_2)
\]
holds for almost all $(x_1, x_2) \in [0, 1]^2$.

Let us verify that $\varphi_1$ and $\varphi_2$ are integrable. Since $g_{12}$ is integrable, it follows from the Fubini-Tonelli theorem that for almost all $x_2^{(0)} \in [0, 1]$ the function $x_1 \mapsto g_{12}(x_1, x_2^{(0)}) = \varphi(x_1) + \varphi_2(x_2^{(0)})$ is also integrable. Since $\varphi_2(x_2^{(0)})$ is a constant, we conclude that $\varphi_1(x_1)$ is integrable. The integrability of $\varphi_2$ is proven in the same manner.
\end{proof}
It follows from \cref{lem:g_12_is_zero} that if $\{g_{ij}\}$ is a solution to the relaxed dual problem, then we can set $\widehat{g}_{12}(x_1, x_2) = 0$, $\widehat{g}_{13}(x_1, x_3) = g_{13}(x_1, x_3) + \varphi_1(x_1)$ and $\widehat{g}_{23}(x_2, x_3) = g_{23}(x_2, x_3) + \varphi_2(x_2)$. Then the equation
\[
g_{12}(x_1, x_2) + g_{13}(x_1, x_3) + g_{23}(x_2, x_3) = \widehat{g}_{12}(x_1, x_2) + \widehat{g}_{13}(x_1, x_3) + \widehat{g}_{23}(x_2, x_3)
\]
holds for all $(x_1, x_2, x_3) \in [0, 1]^3$ except a zero $(3, 2)$-thickness set, and therefore the tuple of functions $\{\widehat{g}_{ij}\}$ is also a solution to the relaxed dual problem. Thus, in \cref{prop:dual_discontinuous_solution_unique} we may additionally assume that $g_{12}(x_1, x_2) = 0$ for all $(x_1, x_2) \in [0, 1]^2$.

\begin{lemma}\label{lem:sum_le_1_plus_eps}
Let $\varphi_1$ and $\varphi_2$ be integrable functions defined on the segment $[0, 1]$. Suppose that there exists a real $\eps > 0$ such that the inequality $\varphi_1(x_1) + \varphi_2(x_2) \le 0$ holds for almost all points $(x_1, x_2)$ such that $0 \le x_1 + x_2 \le 1 + \eps$. Then \[
\int_0^1\varphi_1(x_1)\,dx_1 + \int_0^1\varphi_2(x_2) \le 0.\]

Moreover, if the equality is achieved, then $\varphi_1(x_1) + \varphi_2(x_2) = 0$ for almost all $(x_1, x_2) \in [0, 1]^2$. The same is true if we replace the inequality $0 \le x_1 + x_2 \le 1 + \eps$ with $1 - \eps \le x_1 + x_2 \le 2$.
\end{lemma}
\begin{proof}
Without loss of generality we may assume that $\eps = 1 / n$ for some positive integer $n$. Consider the set $A_1 = \{(x_1, x_2) \in [0, 1]^2 \colon \min(x_1, x_2) \le 1 / (2n)\}$. Let $\mu_1$ be the restriction of the Lebesgue measure to the set $A_1$. One can easily verify that if $\rho$ is the density of the projection of $\mu_1$ to the axis, then $\rho(x) = 1$ if $0 \le x \le 1 / (2n)$ and $\rho(x) = 1 / (2n)$ if $1 / (2n) < x \le 1$. In addition, if $\min(x_1, x_2) \le 1 / (2n)$, then $0 \le x_1 + x_2 \le 1 + 1 / (2n)$, and therefore the inequality $\varphi_1(x_1) + \varphi_2(x_2) \le 0$ holds $\mu_1$-almost everywhere.

Consider the set $A_2 = \{(x_1, x_2) \in [0, 1]^2 \colon \lfloor 2nx_1\rfloor + \lfloor 2nx_2\rfloor = 2n\}$. Let $\mu_2$ be the restriction of the Lebesgue measure to the set $A_2$. If $\lfloor 2nx_1\rfloor + \lfloor 2nx_2\rfloor = 2n$, then $2nx_1 + 2nx_2 < 2n + 2$, and therefore $x_1 + x_2 < 1 +  1 / n$. Hence, $\varphi_1(x_1) + \varphi_2(x_2) \le 0$ for $\mu_2$-almost all points $(x_1, x_2)$. In addition, the projection of $\mu_2$ to the axis is proportional to the restriction of the Lebesgue measure to the segment $[1 / (2n), 1]$, and the density of this projection is equal to $1 / (2n)$ on this segment. See \cref{fig:discontinuous_square} for the visualization of the sets $A_1$ and $A_2$.

\begin{figure}
\centering
\includegraphics[width=0.33\textwidth]{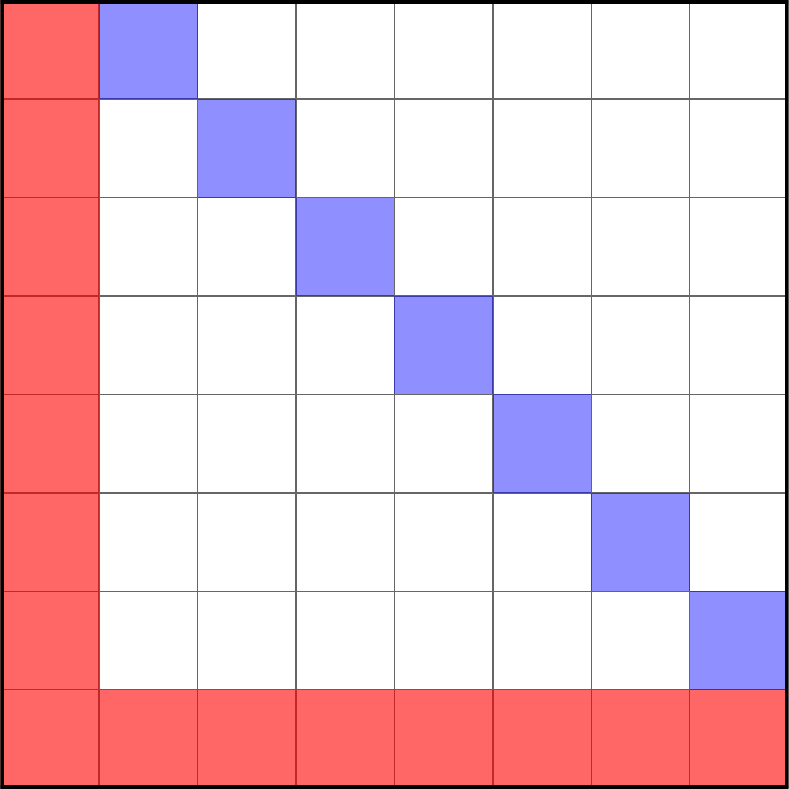}
\caption{The supports of the measures $\mu_1$ and $\mu_2$ for the case $\eps = \frac{1}{4}$. The set $A_1$ is colored red, and the set $A_2$ is blue.}
\label{fig:discontinuous_square}
\end{figure}

Consider the measure $\mu = \mu_1 + (2n - 1)\mu_2$. The projections of this measure to the axes coincides with the restriction of the Lebesgue measure to the segment $[0, 1]$. In addition, $\supp(\mu) \subset \{(x_1, x_2) \in [0, 1]^2 \colon 0 \le x_1 + x_2 \le 1 + 1/n\}$. Thus, we have
\[
\int_0^1\varphi_1(x_1)\,dx_1 + \int_0^1\varphi_2(x_2)\,dx_2 = \int_{[0, 1]^2}(\varphi_1(x_1) + \varphi_2(x_2))\,\mu(dx_1, dx_2) \le 0.
\]

Assume that the equality holds. Then $\varphi_1(x_1) + \varphi_2(x_2) = 0$ $\mu$-almost everywhere. In particular, $\varphi_1(x_1) + \varphi_2(x_2) = 0$ for almost all points $(x_1, x_2) \in A_1$, and therefore this equation holds for almost all points $(x_1, x_2)$ such that $0 \le x_2 \le 1 / (2n)$. Thus, by the Fubini-Tonelli theorem there exists a point $x_2^{(0)} \in [0, 1/(2n)]$ such that the equation $\varphi_1(x_1) + \varphi_2(x_2^{(0)}) = 0$ holds for almost all $x_1 \in [0, 1]$, and therefore there exists a constant $C_1 = -\varphi_2(x_2^{(0)})$ such that $\varphi_1(x_1) = C_1$ almost everywhere.

Similarly, there exists a constant $C_2$ such that $\varphi_2(x_2) = C_2$ almost everywhere. Then 
\[
0 = \int_0^1\varphi_1(x_1)\,dx_1 + \int_0^1\varphi_2(x_2)\,dx_2 = C_1 + C_2,
\]
and therefore $\varphi_1(x_1) + \varphi_2(x_2) = 0$. The case of the inequality $1 - \eps \le x_1 + x_2 \le 2$ is proven in the same manner.
\end{proof}
\begin{proof}[Proof of \cref{prop:dual_discontinuous_solution_unique}] By \cref{lem:g_12_is_zero} we may assume that $g_{12} \equiv 0$. Consider the finite $(3, 2)$-function
\[
G(x_1, x_2, x_3) = g_{13}(x_1, x_3) + g_{23}(x_2, x_3).
\]
The function $G$ is integrable and the inequality $G(x_1, x_2, x_3) \le c(x_1, x_2, x_3)$ holds for almost all points $(x_1, x_2, x_3) \in [0, 1]^3$. Hence, there exists a set $A \subseteq [0, 1]$ with full measure such that if $x_3^{(0)} \in A$, then the function $G(\cdot, \cdot, x_3^{(0)})$ is integrable and the inequality $G(x_1, x_2, x^{(0)}_3) \le c(x_1, x_2, x_3^{(0)})$ holds for almost all $(x_1, x_2) \in [0, 1]^2$.

Assume that $x_3^{(0)} \in A$ and that $x_3^{(0)} < 2 / 3$. Consider the $(3, 2)$-function
\[
F(x_1, x_2, x_3) = f_{12}(x_1, x_2) + f_{13}(x_1, x_3) + f_{23}(x_2, x_3) = f_{13}(x_1, x_3) + f_{23}(x_2, x_3).
\]
By equation \cref{eq:discontinuous_F_repr} we have $F(x_1, x_2, x_3^{(0)}) = 0$ for all $(x_1, x_2) \in [0, 1]^2$.

Denote $\eps = 2 - 3x_3^{(0)}$. We have $\eps > 0$. If $x_1 + x_2 \le 1 + \eps$, then $x_1 + x_2 + 3x_3^{(0)} - 3 \le 0$, and therefore \[
c(x_1, x_2, x_3^{(0)}) = \max(x_1 + x_2 + 3x_3^{(0)} - 3) =  0 = F(x_1, x_2, x_3^{(0)}).
\]
In addition, since $G(x_1, x_2, x_3^{(0)}) \le c(x_1, x_2, x_3^{(0)})$ for almost all points $(x_1, x_2)$, we conclude that the inequality $G(x_1, x_2, x_3^{(0)}) \le F(x_1, x_2, x_3^{(0)})$ holds for almost all points $(x_1, x_2)$ such that $0 \le x_1 + x_2 \le 1 + \eps$.

Consider the functions 
\begin{equation}\label{eq:discontinuous_varphi_definition}
\varphi_1(x_1) = g_{13}(x_1, x_3^{(0)}) - f_{13}(x_1, x_3^{(0)})\text{ and }\varphi_2(x_2) = g_{23}(x_2, x_3^{(0)}) - f_{23}(x_2, x_3^{(0)}).
\end{equation}
We have \[
\varphi_1(x_1) + \varphi_2(x_2) = G(x_1, x_2, x_3^{(0)}) - F(x_1, x_2, x_3^{(0)}).\]
Hence, the function $\varphi_1(x_1) + \varphi_2(x_2)$ is integrable on $[0, 1]^2$, and therefore both functions $\varphi_1$ and $\varphi_2$ are integrable on $[0, 1]$. In addition, the inequality $\varphi_1(x_1) + \varphi_2(x_2) \le 0$ holds for almost all points $(x_1, x_2)$ such that $0 \le x_1 + x_2 \le 1 + \eps$. Thus, it follows from \cref{lem:sum_le_1_plus_eps} that
\[
\int_{[0, 1]^2}\left(G(x_1, x_2, x_3^{(0)}) - F(x_1, x_2, x_3^{(0)})\right)\,dx_1dx_2 = \int_0^1\varphi_1(x_1)\,dx_1 + \int_0^1\varphi_2(x_2)\,dx_2 
\le 0.
\]
Moreover, if the equality holds, then $G(x_1, x_2, x_3^{(0)}) = F(x_1, x_2, x_3^{(0)})$ almost everywhere.

Assume that $x_3^{(0)} \in A$ and that $x_3^{(0)} > 2 / 3$. By equation \cref{eq:discontinuous_F_repr} we have \[
F(x_1, x_2, x_3^{(0)}) = x_1 + x_2 + 3x_3^{(0)} - 3.
\]
Denote $\eps = 3x_3^{(0)} - 2 > 0$. If $x_1 + x_2 > 1 - \eps$, then $x_1 + x_2 + 3x_3 - 3 > 0$, and therefore
\[
c(x_1, x_2, x_3^{(0)}) = \max(x_1 + x_2 + 3x_3^{(0)} - 3, 0) = x_1 + x_2 + 3x_3^{(0)} - 3 = F(x_1, x_2, x_3^{(0)}).
\]
Hence, since $G(x_1, x_2, x_3^{(0)}) \le c(x_1, x_2, x_3^{(0)})$ for almost all $(x_1, x_2)$, we conclude that $\varphi_1(x_1) + \varphi_2(x_2) \le 0$ for almost all points $(x_1, x_2)$ such that $1 - \eps \le x_1 + x_2 \le 2$, where the functions $\varphi_1$ and $\varphi_2$ are defined in equation \cref{eq:discontinuous_varphi_definition}. Thus, it follows from \cref{lem:sum_le_1_plus_eps} that
\begin{equation}\label{eq:discontinuos_G_le_F}
\int_{[0, 1]^2}G(x_1, x_2, x_3^{(0)})\,dx_1dx_2 \le \int_{[0, 1]^2}F(x_1, x_2, x_3^{(0)})\,dx_1dx_2,
\end{equation}
and if the equality holds, then $G(x_1, x_2, x_3^{(0)}) = F(x_1, x_2, x_3^{(0)})$ for almost all $(x_1, x_2)$.

Summarizing this results, we conclude that if $x_3^{(0)} \in A$ and if $x_3^{(0)} \ne 2/3$, then inequality \cref{eq:discontinuos_G_le_F} holds, and therefore, since $A$ is a set of full measure, we have
\[
\int_{[0, 1]^3}G(x_1, x_2, x_3)\,dx_1dx_2dx_3 \le \int_{[0, 1]^3}F(x_1, x_2, x_3)\,dx_1dx_2dx_3.
\]
Since $\{g_{ij}\}$ is a solution to the relaxed dual problem, the equality holds, and therefore the equality in inequality \cref{eq:discontinuos_G_le_F} is achieved for almost all $x_3^{(0)}$. Thus, for almost all $x_3^{(0)} \in [0, 1]$ the equation $F(x_1, x_2, x_3^{(0)}) = G(x_1, x_2, x_3^{(0)})$ holds for almost $(x_1, x_2) \in [0, 1]^2$, and therefore
\[
g_{12}(x_1, x_2) + g_{13}(x_1, x_3) + g_{23}(x_2, x_3) = f_{12}(x_1, x_2) + f_{13}(x_1, x_3) + f_{23}(x_2, x_3)
\]
almost everywhere.
\end{proof}
\bibliographystyle{abbrvurl}
\bibliography{references}

\end{document}